\documentclass[a4paper,reqno]{amsart}

\usepackage{amsmath}
\usepackage{amsthm}
\usepackage{amssymb}
\usepackage{amscd} %diagrams

\usepackage{bbm} %bold numbers
\usepackage{mathtools,mathrsfs} %nice calligraphy

\usepackage{verbatim} %DO NOT comment this line (comment environment)

\usepackage{enumitem}

\newtheoremstyle{newremark}
  {5pt}
  {5pt}
  {\rmfamily}
  {}
  {\rmfamily\bf}
  {.}
  {.5em}
  {}

\newtheorem{theorem}{Theorem}
\newtheorem{lemma}[theorem]{Lemma}
\newtheorem{corollary}[theorem]{Corollary}
\newtheorem{proposition}[theorem]{Proposition}

\theoremstyle{newremark}
\newtheorem{remark}[theorem]{Remark}
\newtheorem{definition}[theorem]{Definition}

\newtheorem*{definition*}{Definition} %no numbering for Theorem*
\newtheorem*{notations*}{Notations}

\numberwithin{theorem}{section}
\numberwithin{equation}{section}

\newcommand{\N}{\mathbb{N}} %natural numbers
\newcommand{\R}{\mathbb{R}} %real numbers

\newcommand{\Rn}{\R^n}
\newcommand{\Z}{\mathbb{Z}} %integers
\def\Xint#1{\mathchoice
{\XXint\displaystyle\textstyle{#1}}%
{\XXint\textstyle\scriptstyle{#1}}%
{\XXint\scriptstyle\scriptscriptstyle{#1}}%
{\XXint\scriptscriptstyle
\scriptscriptstyle{#1}}%
\!\int}
\def\XXint#1#2#3{{%
\setbox0=\hbox{$#1{#2#3}{\int}$}
\vcenter{\hbox{$#2#3$}}\kern-.5\wd0}}

\def\dashint{\Xint-}

\renewcommand{\leq}{\leqslant}
\renewcommand{\geq}{\geqslant}
\renewcommand{\subset}{\subseteq}

\newcommand{\res}{\mathop{\hbox{\vrule height 7pt width .5pt depth 0pt
\vrule height .5pt width 6pt depth 0pt}}\nolimits}

\newcommand{\eps}{\varepsilon}

\newcommand{\e}{{\rm e}}

\newcommand{\de}{{\rm d}}

\usepackage[colorinlistoftodos,prependcaption,textsize=tiny]{todonotes}

\begin{document}

%=================
% TITLE AND AUTHOR
%=================

\title[\bf Partial regularity for fractional harmonic maps]{Partial regularity for\\ fractional harmonic maps into spheres}

\author{Vincent Millot}
\address{LAMA, Univ Paris Est Creteil, Univ Gustave Eiffel, UPEM, CNRS, F-94010, Cr\'eteil, France}
\email{vincent.millot@u-pec.fr}

\author{Marc Pegon}
\address{Universit\'e de Paris, Laboratoire Jacques-Louis Lions (LJLL), F-75013 Paris, France}
\email{mpegon@math.univ-paris-diderot.fr}

\author{Armin Schikorra}
\address{University of Pittsburgh, Department of Mathematics, 301 Thackeray Hall, Pittsburgh, PA15260, USA}
\email{armin@pitt.edu}

%\date{\today}

%=========
% ABSTRACT
%=========

\begin{abstract}
This article addresses the regularity issue for stationary or minimizing fractional harmonic maps into spheres of order $s\in(0,1)$ in arbitrary dimensions. It is shown that such fractional harmonic maps are $C^\infty$ away from a small closed singular set. The Hausdorff dimension of the singular set is also estimated in terms of  $s\in(0,1)$ and the stationarity/minimality assumption.  
\end{abstract}

%\keywords{keywords}

%\thanks{{\it Aknowledgements.} }

\maketitle

%==================
% TABLE OF CONTENTS
%==================

\tableofcontents

%%%%%%%%%%%%%%%%%%%%%%%%%%%%%%%%%%%%%%%%%%%%%%%%%%%%%%%%%%

\section{Introduction}

The theory of fractional harmonic maps  into a manifold is quite recent. It is has been initiated some years ago by F. Da Lio and T. Rivi\`ere in \cite{DaLiRi1,DaLiRi2}. In those first articles, they have introduced and studied $1/2$-harmonic maps from the real line into a smooth and compact closed submanifold $\mathcal{N}\subset\R^d$. A map $u:\R\to\mathcal{N}$ is said to be a  $1/2$-harmonic map into $\mathcal{N}$ if it is a critical point of the $1/2$-Dirichlet energy  
$$\mathcal{E}_{\frac{1}{2}}(u,\R):=\frac{1}{2}\int_{\R}\big|(-\Delta)^{\frac{1}{4}}u\big|^2\,\de x=\frac{1}{4\pi}\iint_{\R\times\R}\frac{|u(x)-u(y)|^2}{|x-y|^2}\,\de x\de y \,,$$ 
among all maps with values into $\mathcal{N}$, or equivalently, if it satisfies the Euler-Lagrange equation 
\begin{equation}\label{eq1/2harmintro1}
(-\Delta)^{\frac{1}{2}}u\perp{\rm Tan}(u,\mathcal{N})
\end{equation}
in the distributional sense. Here $(-\Delta)^s$ denotes the integro-differential (multiplier) operator associated to the Fourier symbol $(2\pi|\xi|)^{2s}$, $s\in(0,1)$. The notion of $1/2$-harmonic map into $\mathcal{N}$ appears in several geometrical problems, such as free boundary minimal surfaces or  Steklov eigenvalue problems, see \cite{DaLi2} and references therein. The special case $\mathcal{N}=\mathbb{S}^{d-1}$ is important for both geometrical and analytical issues. From the analytical point of view, it enlightens the internal structure  of equation \eqref{eq1/2harmintro1}. Indeed, the Lagrange multiplier associated to the constraint to be $\mathbb{S}^{d-1}$-valued takes a very simple form, and \eqref{eq1/2harmintro1} reduces to the equation
\begin{equation}\label{eq1/2harmsphereintro}
(-\Delta)^{\frac{1}{2}}u(x)=\left(\frac{1}{2\pi}\int_{\R}\frac{|u(x)-u(y)|^2}{|x-y|^2}\,\de y\right)u(x)\,, 
\end{equation}
 which is in clear analogy with the equation for usual harmonic maps from a $2$d-domain into the sphere. In particular, there is a similar analytical issue concerning regularity of solutions since the right hand side of \eqref{eq1/2harmsphereintro} has {\sl a priori} no better integrability than $L^1$, and elliptic linear theory does not  apply. In their pioneering work  \cite{DaLiRi1}, F. Da Lio and T. Rivi\`ere proved complete smoothness of $1/2$-harmonic maps through a reformulation of equation \eqref{eq1/2harmsphereintro} in terms of algebraic quantities, the ``3-terms commutators",  exhibiting some compensation phenomena.  In  \cite{DaLiRi2} (dealing with arbitrary targets), smoothness  of $1/2$-harmonic maps follows from a more general compensation result for nonlocal systems with antisymmetric potential, in the spirit of~\cite{Riv2}.  
 In the same stream of ideas, K. Mazowiecka and the third author obtained in \cite{MazSchi} a new proof of the regularity of $1/2$-harmonic maps, very close  to the original argument of F.~H\'elein~\cite{Hel1} to prove smoothness of harmonic maps from surfaces into spheres (see also~\cite{Hel2}). Once again, the key point in \cite{MazSchi} is to rewrite the right hand side of \eqref{eq1/2harmsphereintro} to discover a suitable ``fractional div-curl structure". From the new form of the equation, they deduce that  $(-\Delta)^{\frac{1}{2}}u$ belongs (essentially) to the Hardy space~$\mathcal{H}^1$ by applying their main result \cite[Theorem 2.1]{MazSchi}, a generalization to the fractional setting of the div-curl estimate of R. Coifman, P.L. Lions, Y.~Meyer,  and S.~Semmes~\cite{CLMS}. Continuity of solutions is then a consequence of Calder\'on-Zygmund theory, from which it is possible to  deduce $C^\infty$-regularity. 

Several generalizations of the regularity result of \cite{DaLiRi1,DaLiRi2} have been obtained, e.g. for critical points of   higher order or/and $p$-power type energies (still in the corresponding critical dimension), see   \cite{DaLi3,DaLiSchi1,DaLiSchi2,MazSchi,Schi1,Schi2,Schi3}. 
The regularity theory for $1/2$-harmonic maps into a manifold in higher dimensions has been addressed in \cite{Moser} and \cite{MS} (see also \cite{MilPeg}). In higher dimensions, the theory provides  partial regularity (i.e. regularity away from a ``small'' singular set) for stationary $1/2$-harmonic maps (i.e. critical points for both {\sl inner and outer variations}), and energy minimizing $1/2$-harmonic maps. It can be seen as the analogue of the partial regularity theory for harmonic maps by R. Schoen and K. Uhlenbeck \cite{SchUhl1,SchUhl2} in the minimizing case, and by L.C. Evans \cite{Evans} and F. Bethuel \cite{Bet} in the stationary case. In \cite{MS}, the argument consists in considering the harmonic extension to the upper half space in one more dimension provided by the convolution with the Poisson kernel. The extended map is then harmonic and satisfies a nonlinear Neumann boundary condition which fits within the (previously known) theory of harmonic maps with partially free boundary, see \cite{Duz1,Duz2,GJ,HL,Schev}. 
\vskip3pt

The purpose of this article is to extend the regularity theory  for fractional harmonic maps in arbitrary dimensions to the context of {\sl $s$-harmonic maps}, i.e., when the operator $(-\Delta)^{\frac{1}{2}}$ is replaced by $(-\Delta)^{s}$ with arbitrary power $s\in(0,1)$. As a first attempt in this direction, we only consider the case where the target manifold $\mathcal{N}$ is the standard unit sphere $\mathbb{S}^{d-1}$ of $\R^d$, $d\geq 2$. We now describe the functional setting.

Given $s\in(0,1)$ and $\Omega\subset\R^n$ a bounded open set, the fractional $s$-Dirichlet energy in $\Omega$ of a measurable map $u:\R^n\to\R^d$ is defined by 
\begin{equation}\label{defsdirenerg}
\mathcal{E}_s(u,\Omega):=\frac{\gamma_{n,s}}{4}\iint_{(\R^n\times\R^n)\setminus(\Omega^c\times\Omega^c)}\frac{|u(x)-u(y)|^2}{|x-y|^{n+2s}}\,\de x\de y\,,
%+\frac{\gamma_{n,s}}{2}\iint_{\Omega\times(\R^n\setminus\Omega)}\frac{|u(x)-u(y)|^2}{|x-y|^{n+2s}}\,\de x\de y \,,
\end{equation}
where  $\Omega^c$ denotes the complement of $\Omega$, i.e. $\Omega^c:=\R^n\setminus\Omega$. The normalisation constant $\gamma_{n,s}>0$, whose precise value is given by \eqref{defHsandgammans}, is chosen in such a way that 
$$ \mathcal{E}_s(u,\Omega)=\frac{1}{2}\int_{\R^n}\big|(-\Delta)^{\frac{s}{2}}u\big|^2\,\de x\qquad\forall u\in \mathscr{D}(\Omega;\R^d)\,.$$
Following \cite{MS,MSK},  we denote by $\widehat H^s(\Omega;\R^d)$ the Hilbert space made of $L^2_{\rm loc}(\R^n)$-maps $u$ such that $\mathcal{E}_s(u,\Omega)<\infty$, and we set 
$$\widehat H^s(\Omega;\mathbb{S}^{d-1}):= \Big\{u\in \widehat H^s(\Omega;\R^d) : u(x)\in\mathbb{S}^{d-1}\text{ for a.e. }x\in\R^n\Big\}\,.$$
We then define weakly $s$-harmonic maps in $\Omega$ as critical points of $ \mathcal{E}_s(u,\Omega)$ in the (nonlinear) space $\widehat H^s(\Omega;\mathbb{S}^{d-1})$. 
More precisely, we say that a map $u\in \widehat H^s(\Omega;\mathbb{S}^{d-1})$ is a {\sl weakly $s$-harmonic map} in $\Omega$ into $\mathbb{S}^{d-1}$ if 
$$\left[\frac{\de}{\de t}\mathcal{E}_s\Big(\frac{u+t\varphi}{|u+t\varphi|},\Omega\Big)\right]_{t=0}=0\qquad\forall \varphi\in\mathscr{D}(\Omega,\R^d)\,.$$
Exactly as \eqref{eq1/2harmsphereintro}, the Euler-Lagrange equation reads 
\begin{equation}\label{sharmmapeqintro}
(-\Delta)^su(x)=\left(\frac{\gamma_{n,s}}{2}\int_{\R^n}\frac{|u(x)-u(y)|^2}{|x-y|^{n+2s}}\,\de y\right)u(x)\quad\text{in $\mathscr{D}^\prime(\Omega)$}\,, 
\end{equation}
where $(-\Delta)^s$ is the integro-differential operator given by 
$$(-\Delta)^s u(x):={\rm p.v.}\left(\gamma_{n,s}\int_{\R^n}\frac{u(x)-u(y)}{|x-y|^{n+2s}}\,\de y\right)\,, $$
and the notation ${\rm p.v.}$ means that the integral is taken in the Cauchy principal value sense. We refer to Section \ref{prelim} and \ref{ELandConsLaw} for the precise  weak (variational) formulation of equation \eqref{sharmmapeqintro}. 
\vskip3pt

Once again, the right hand side in \eqref{sharmmapeqintro} has a priori no better integrability than $L^1$, and linear elliptic theory does not apply to determine the regularity of solutions. However, in the case $n\leq 2s$, that is $n=1$ and $s\in[1/2,1)$, the equation is {\sl subcritical}. For $n=1$ and $s=1/2$, this is the result of \cite{DaLiRi1,DaLiRi2}. For $n=1$ and $s\in(1/2,1)$, solutions are at least H\"older continuous by the embedding $H^s\hookrightarrow C^{0,s-1/2}$, and this is to enough to reach $C^\infty$-smoothness by applying Schauder type estimates for the fractional Laplacian.

\begin{theorem}\label{mainthm1}
Assume that $n=1$ and $s\in[1/2,1)$. If $u\in \widehat H^s(\Omega;\mathbb{S}^{d-1})$ is a weakly  $s$-harmonic map in $\Omega$, then $u\in C^\infty(\Omega)$. 
\end{theorem}

On the other hand, the case $n>2s$ is {\sl supercritical}, and by analogy with (usual) weakly harmonic maps in dimension at least $3$, we do not expect any regularity without further assumptions. Indeed, in his groundbreaking article \cite{Riv1}, T. Rivi\`ere has constructed a weakly harmonic map from the $3$-dimensional ball into $\mathbb{S}^2$ which is everywhere discontinuous. A natural extra assumption to assume on a weakly $s$-harmonic map is {\sl stationarity}, that is 
$$\left[\frac{\de}{\de t}\mathcal{E}_s\big(u\circ\phi_{t},\Omega\big)\right]_{t=0}=0\qquad\forall X\in C^1_c(\Omega;\R^n)\,, $$
where $\{\phi_t\}_{t\in\R}$ denotes the integral flow of the vector field $X$. According to the standard terminology in calculus of variations, a weakly $s$-harmonic map in $\Omega$ is a critical point of $\mathcal{E}_s(\cdot,\Omega)$ with respect to outer variations (i.e. in the target), a stationary map is a critical point of $\mathcal{E}_s(\cdot,\Omega)$ with respect to inner variations (i.e. in the domain), and thus  a {\sl stationary weakly $s$-harmonic map} in $\Omega$ is a critical point of $\mathcal{E}_s(\cdot,\Omega)$ with respect to both inner and outer variations. 

Our second main result provides partial regularity for such maps. In its statement, 
the {\sl singular set} of $u$ in $\Omega$ is defined as 
$${\rm sing}(u):=\Omega\setminus\big\{x\in\Omega: \text{$u$ is continuous in a neighborhood of $x$}\big\}\,, $$
${\rm dim}_{\mathcal{H}}$ denotes the Hausdorff dimension, and $\mathcal{H}^{n-1}$ is the $(n-1)$-dimensional Hausdorff measure. 

\begin{theorem}\label{mainthm2}
Assume that $s\in(0,1)$ and $n> 2s$. If $u\in \widehat H^s(\Omega;\mathbb{S}^{d-1})$ is a stationary weakly  $s$-harmonic map in $\Omega$, then $u\in C^\infty(\Omega\setminus{\rm sing}(u))$ and 
\begin{enumerate}
\item for $s>1/2$ and $n\geq 3$, ${\rm dim}_{\mathcal{H}}\,{\rm sing}(u)\leq n-2$; 
\vskip3pt
\item for $s>1/2$ and $n=2$, ${\rm sing}(u)$ is locally finite in $\Omega$; 
\vskip3pt
\item  for $s=1/2$ and $n\geq 2$, $\mathcal{H}^{n-1}({\rm sing}(u))=0$; 
\vskip3pt
\item  for $s<1/2$ and $n\geq 2$, ${\rm dim}_{\mathcal{H}}\,{\rm sing}(u)\leq n-1$; 
\vskip3pt
\item  for $s<1/2$ and $n=1$, ${\rm sing}(u)$ is locally finite in $\Omega$. 
\end{enumerate}
\end{theorem}
\vskip3pt

The other common assumption to consider is energy minimality. We say that a map $u\in \widehat H^s(\Omega;\mathbb{S}^{d-1})$ is a minimizing $s$-harmonic map in $\Omega$ if 
$$\mathcal{E}_s(u,\Omega)\leq \mathcal{E}_s(v,\Omega) $$
for every competitor $v\in \widehat H^s(\Omega;\mathbb{S}^{d-1})$ such  that $v-u$ is compactly supported in $\Omega$. Notice that minimality implies criticality with respect to both inner and outer variations, 
so that a minimizing $s$-harmonic map in $\Omega$ is in particular a stationary weakly  $s$-harmonic map in $\Omega$. However, minimality implies a stronger partial regularity, at least for $s\in(0,1/2)$. 

\begin{theorem}\label{mainthm3}
Assume that $s\in(0,1)$ and $n>2s$. If $u\in \widehat H^s(\Omega;\mathbb{S}^{d-1})$ is a minimizing  $s$-harmonic map in~$\Omega$, then $u\in C^\infty(\Omega\setminus{\rm sing}(u))$ and 
\begin{enumerate}
\item for $n\geq 3$, ${\rm dim}_{\mathcal{H}}\,{\rm sing}(u)\leq n-2$; 
\vskip3pt
\item for  $n=2$, ${\rm sing}(u)$ is locally finite in $\Omega$; 
%\vskip3pt
%\item  for $s=1/2$ and $n\geq 3$, $\mathscr{H}^{n-1}({\rm sing}(u))=0$; 
%\vskip3pt
%\item  for $s<1/2$ and $n\geq 2$, ${\rm dim}_{\mathscr{H}}\,{\rm sing}(u)\leq n-1$; 
\vskip3pt
\item  for   $n=1$, ${\rm sing}(u)=\emptyset$ (i.e., $u\in C^\infty(\Omega)$). 
\end{enumerate}
\end{theorem}

Before describing  the way we prove Theorem \ref{mainthm2} and Theorem \ref{mainthm3}, let us comment on the sharpness of the results above.

\begin{remark}\label{rem1intro}
In the case $s\in(0,1/2)$, essentially no better regularity than the one coming from the energy space can be expected from a weakly $s$-harmonic map in $\Omega$. 
Indeed, for an arbitrary set $E\subset \R^n$ such that the characteristic function $\chi_E$ belongs to $\widehat H^s(\Omega)$, consider the function $u:=\chi_E-\chi_{E^c}$. Identifying $\R^2$ with the complex plane $\mathbb{C}$, we can see $u$ as a map from $\R^n$ into~$\mathbb{S}^1$, and it belongs to $\widehat H^s(\Omega;\mathbb{S}^1)$. It has been observed in \cite[Remark 1.7]{MSK} that $u$ is a weakly $s$-harmonic map in $\Omega$ into~$\mathbb{S}^1$, i.e., it satisfies \eqref{sharmmapeqintro}. 
For $s=1/2$, we believe that, in the spirit of \cite{Riv1}, it should be possible to construct an example of a $1/2$-harmonic map from the $2$-dimensional disc into $\mathbb{S}^1$ which is discontinuous everywhere using the material in \cite{MPis}. However, for $s\in(1/2,1)$ and $n=2$, it remains open whether or not such pathological example do exist.
%, or if weakly $s$-harmonic maps do enjoy some partial regularity. 
\end{remark}

\begin{remark}
For $s\in(0,1/2)$, the partial regularity for stationary weakly $s$-harmonic maps is sharp in the sense that the size of the singular set can not be improved. Following Remark \ref{rem1intro} above and \cite[Remark 1.7]{MSK}, for a set $E\subset \R^n$ such that  $\chi_E\in\widehat H^s(\Omega)$, the map $u:=\chi_E-\chi_{E^c}$ is a weakly $s$-harmonic map in $\Omega$ into $\mathbb{S}^1$, and 
$$\mathcal{E}_s(u,\Omega)= \gamma_{n,s}P_{2s}(E,\Omega)\,,$$
where $P_{2s}(E,\Omega)$ is the fractional $2s$-perimeter of $E$ in $\Omega$ introduced by L. Caffarelli, J.M. Roquejoffre, and O. Savin in \cite{CRS}, and it is given by 
$$P_{2s}(E,\Omega)=\left(\iint_{(E\cap\Omega)\times(E^c\cap\Omega)}+\iint_{(E\cap\Omega^c)\times(E^c\cap\Omega)}+\iint_{(E\cap\Omega)\times(E^c\cap\Omega^c)}\right)\frac{\de x\de y}{|x-y|^{n+2s}} \,.$$
Therefore, $u$ is a stationary weakly $s$-harmonic map in $\Omega$ if and only if $E$ is stationary in $\Omega$ for the shape functional $P_{2s}(\cdot,\Omega)$ (see \cite{MSK}). This includes the case where $\partial E$ is a nonlocal minimal surface in the sense of \cite{CRS}. In particular, if $E$ is a half space, then $u$ is a stationary weakly $s$-harmonic map in $\Omega$, and ${\rm sing}(u)=\partial E\cap\Omega$ is an hyperplane. 
\end{remark}

\begin{remark}
For arbitrary spheres, Theorem \ref{mainthm3} is sharp for $s=1/2$. Indeed, we know from~\cite[Theorem 1.4]{MilPeg} that the map $x/|x|$ is a minimizing $1/2$-harmonic map into $\mathbb{S}^1$ in the unit disc $D_1\subset\R^2$. The minimality of $x/|x|$ for $s\not=1/2$ is open, but one can check that it is at least a stationary $s$-harmonic map into $\mathbb{S}^1$ in $D_1$, showing that Theorem \ref{mainthm2} is sharp also for~$s\in[1/2,1)$.

For arbitrary $s\in(0,1)$, the following classical example suggests that Theorem \ref{mainthm3} might be  sharp anyway.  Consider the minimization problem (still in dimension $n=2$), 
$$\min\Big\{\mathcal{E}_s(u,D_1): u\in \widehat H^s(D_1,\mathbb{S}^1)\,,\;u(x)=x/|x|\text{ in $\R^2\setminus D_1$}\Big\}\,.$$
Existence of solutions follows easily from the direct method of calculus of variations, and any solution is obviously a minimizing $s$-harmonic map in~$D_1$. 
Since $x/|x|$ does not admit any $\mathbb{S}^1$-valued continuous extension to $D_1$, any solution must have at least one singular point in $\overline{D}_1$. 
\end{remark}

\begin{remark}
For $s=1/2$ and $d\geq 3$ (i.e., for $\mathbb{S}^2$ or higher dimensional target spheres), the size of the singular set of a minimizing $1/2$-harmonic map can be reduced. It has been proved in \cite[Theorem 1.3]{MilPeg} that in this case, ${\rm sing}(u)=\emptyset$ for $n=2$, it is locally finite for $n=3$, and ${\rm dim}_{\mathcal{H}}{\rm sing}(u)\leq n-3$ for $n\geq 4$. It would be interesting to know if this improvement persists for $s\not=1/2$. 
\end{remark}

The proofs of Theorems  \ref{mainthm1},  \ref{mainthm2}, and  \ref{mainthm3} rely on several ingredients that we now briefly describe. The first one consists in applying the so-called {\sl Caffarelli-Silvestre extension}  procedure~\cite{CaffSil} to the open half space $\R^{n+1}_+:=\R^n\times(0,+\infty)$. This extension (which may have originated in the probability literature~\cite{MolOs}) allows us to represent $(-\Delta)^s$ as the Dirichlet-to-Neumann operator associated with the degenerate elliptic operator $L_s:=-{\rm div}(z^{1-2s}\nabla\cdot)$, where $z\in(0,+\infty) $ denotes the extension variable. In this way (after extension), we can reformulate the $s$-harmonic map equation as a degenerate harmonic map equation with {\sl partially} free boundary, very much like in \cite{MS,MSK}. Under the stationarity assumption, the extended map satisfies a fundamental monotonicity formula, which in turn implies local  controls in the space BMO (bounded mean oscillation) of the $s$-harmonic map under consideration by its energy. 
%belongs (locally) to BMO (bounded mean oscillation).  

Probably the main step in the proof is an epsilon-regularity result where we show that under a (standard) smallness assumption on the energy $\mathcal{E}_s$ in a ball, then a (stationary) $s$-harmonic map is H\"older continuous in a smaller ball. The strategy we follow here is quite inspired from the argument of L.C. Evans \cite{Evans} making use of the conservation laws discovered by F.~H\'elein \cite{Hel1} and the duality $\mathcal{H}^1$/BMO. In our fractional setting, we make use of the fractional conservation laws together with the ``fractional div-curl lemma'' of K. Mazowiecka and the third author \cite{MazSchi}. A main difference with \cite{Evans} lies in the fact that an additional ``error term'' appears when rewriting the $s$-harmonic map equation in the suitable form where compensation can be seen. To control this error term in arbitrary dimensions, we make use of a recent embedding result  between Triebel-Lizorkin-Morrey type spaces \cite{Ho} and various characterizations of these spaces~\cite{SaYY,YangYuan}. 
 
Once H\"older continuity  is obtained, we prove Lipschitz continuity in an even smaller ball using an adjustment of the classical ``harmonic replacement'' technique, see \cite{Schoen}. More precisely, using the extension, we adapt an argument due to J. Roberts \cite{Rob} in the case of degenerate harmonic maps with free boundary (i.e., with homogeneous - degenerate - Neumann boundary condition). With Lipschitz continuity in hands, we are then able to derive $C^\infty$-regularity from Schauder estimates for the fractional Laplacian. 

To obtain the bounds on the size of the singular set, we follow somehow the usual dimension reduction argument of Almgren \& Federer for harmonic maps (see \cite{Sim}), which is based on 
the  strong compactness of blow-ups around points. Here compactness (for $s\not=1/2$) is obtained as in \cite{MSK}, and it is a consequence of the monotonicity formula together with Marstrand's Theorem (see e.g. \cite{Matti}). Finally, in the minimizing case and $s\in(0,1/2)$, we obtain an improvement on the size of the singular set (compared to the stationary case) from the triviality of the so-called ``tangent maps'' (i.e. blow-up limits), a consequence of the regularity of minimizing $s$-harmonic in one dimension proved in \cite{MSY}.

\subsection*{{Notation}}
Throughout the paper, $\R^n$ is often identified with $\partial  \mathbb{R}^{n+1}_+=\R^n\times\{0\}$. More generally, sets $A\subset\mathbb{R}^n$ can be identified with $A\times\{0\}\subset\partial  \mathbb{R}^{n+1}_+$. Points in $\mathbb{R}^{n+1}$ are written $\mathbf{x}=(x,z)$ with $x\in\mathbb{R}^n$ and $z\in\mathbb{R}$.  
We shall denote by $B_r(\mathbf{x})$ the open ball in $\mathbb{R}^{n+1}$ of radius $r$ centered at $\mathbf{x}=(x,z)$, while $D_r(x):= B_r(\mathbf{x})\cap\mathbb{R}^{n}$ is the open ball (or disc) in $\R^n$ centered at $x$. For an arbitrary set $G\subset  \mathbb{R}^{n+1}$, we write 
$$G^+:=G\cap \mathbb{R}^{n+1}_+\quad\text{ and }\quad\partial^+ G:=\partial G\cap \mathbb{R}^{n+1}_+\,.$$
If $G\subset\R^{n+1}_+$ is a bounded open set, we shall say that $G$ is  {\bf admissible} whenever 
\begin{itemize}
\item $\partial G$ is Lipschitz regular;  
\vskip2pt
\item the (relative) open set $\partial^0G\subset\partial\R^{n+1}_+$ defined by 
$$\partial^0G:=\Big\{\mathbf{x}\in\partial G\cap\partial\R^{n+1}_+ : B^+_{r}(\mathbf{x})\subset G \text{ for some $r>0$}\Big \}\,,$$
is non empty and has Lipschitz boundary; 
\vskip2pt

\item $\partial G=\partial^+ G\cup\overline{\partial^0G}\,$.
\end{itemize}

Finally, we usually denote by $C$ a generic positive constant which only depends on the dimension $n$ and $s\in(0,1)$, and possibly changing from line to line. If a constant depends on additional given parameters, we shall write those parameters using the subscript notation.

%%%%%%%%%%%%%%%%%%%%%%%%%%%%%%%%%%%%%%%%%%%%%%%%%%%%%%%
%%%%%%%%%%%%%%%%%%%%%%%%%%%%%%%%%%%%%%%%%%%%%%%%%%%%%%%
   								       						%%%%%%%%%%%%%%%%%%%
\section{Functional spaces, fractional operators, and compensated compactness } \label{prelim} %%%%%%%%%
								 						%%%%%%%%%%%%%%%%%%%
%%%%%%%%%%%%%%%%%%%%%%%%%%%%%%%%%%%%%%%%%%%%%%%%%%%%%%%
%%%%%%%%%%%%%%%%%%%%%%%%%%%%%%%%%%%%%%%%%%%%%%%%%%%%%%%

 \subsection{Fractional $H^{s}$-spaces}\label{secHs}
 
For an open set $\Omega\subset \mathbb{R}^n$,  the Sobolev-Slobodeckij space $H^{s}(\Omega)$ is made of all functions $u\in L^2(\Omega)$ such that\footnote{The normalization constant $\gamma_{n,s}$ is chosen in such a way that $\displaystyle [u]^2_{H^{s}(\R^n)}=\int_{\R^n}(2\pi|\xi|)^{2s}|\widehat u|^2\,\de\xi\,$, where $\widehat u$ denotes the (ordinary frequency) Fourier transform of $u$.}  
\begin{equation}\label{defHsandgammans}
[u]^2_{H^{s}(\Omega)}:=\frac{\gamma_{n,s}}{2}\iint_{\Omega\times \Omega} \frac{|u(x)-u(y)|^2}{|x-y|^{n+2s}}\,\de x\de y<\infty\,,\quad \gamma_{n,s}:=s\,2^{2s}\pi^{-\frac{n}{2}}\frac{\Gamma\big(\frac{n+2s}{2}\big)}{\Gamma(1-s)} \,.
\end{equation}
It is a separable Hilbert space normed by $\|\cdot\|^2_{H^{s}(\Omega)}:= \|\cdot\|^2_{L^2(\Omega)}+[\cdot]^2_{H^{s}(\Omega)}$. 
The space $ H^{s}_{\rm loc}(\Omega)$ denotes the class of functions whose restriction to any relatively compact  open subset $\Omega'$ of $\Omega$ belongs to $H^{s}(\Omega')$. 
The linear subspace $H^{s}_{00}(\Omega) \subset H^{s}(\mathbb{R}^n)$ is in turn defined by 
$$H^{s}_{00}(\Omega):=\big\{u\in H^{s}(\mathbb{R}^n) :  u=0 \text{ a.e. in } \R^n\setminus\Omega\big\}\,. $$
Endowed with the  induced norm,   
 $H^{s}_{00}(\Omega)$ is also a Hilbert space, and   
\begin{equation*}
%\label{defErond}
[u]^2_{H^{s}(\mathbb{R}^n)}=\frac{\gamma_{n,s}}{2}\iint_{(\R^n\times\R^n)\setminus(\Omega^c\times \Omega^c)} \frac{|u(x)-u(y)|^2}{|x-y|^{n+2s}}\,\de x\de y
=2\mathcal{E}_s(u,\Omega)\quad\forall u\in H^{s}_{00}(\Omega)\,,
%\nonumber&=\frac{\gamma_{n,s}}{2}\iint_{\Omega\times\Omega}  \frac{|v(x)-v(y)|^2}{|x-y|^{n+2s}}\,\de x\de y 
%+ \gamma_{n,s} \iint_{\Omega\times \Omega^{c}}  \frac{|v(x)|^2}{|x-y|^{n+2s}}\,\de x\de y \\
%\nonumber &
%&=[v]^2_{H^{s}(\Omega)} + \int_{\Omega} \boldsymbol{\rho}_\Omega(x) |v(x)|^2\,\de x\,,
\end{equation*}
where $\mathcal{E}_s(\cdot,\Omega)$ is  the $s$-Dirichlet energy defined in \eqref{defsdirenerg}. 
%and 
%$$ \boldsymbol{\rho}_\Omega(x) := \gamma_{n,s}\int_{\Omega^c}\frac{1}{|x-y|^{n+2s}}\,\de y\,.$$
%Since $s\in(0,1/2)$, if 

If $\Omega$ is bounded and its boundary is smooth enough (e.g. if  $\partial \Omega$ is Lipschitz regular), then 
%$$ \int_{\Omega} \boldsymbol{\rho}_\Omega(x) |v(x)|^2\,\de x\leq C_\Omega \|v\|^2_{H^{s}(\Omega)}\qquad \forall v\in H^s(\Omega)\,,$$
%for a constant $C_\Omega=C_\Omega(s)>0$. As a consequence, if $v\in H^s(\Omega)$ and $\widetilde v$ denotes the extension of $v$ by zero  outside $\Omega$, then  
%$$ \|v\|_{H^{s}(\Omega)}\leq \|\widetilde v\|_{H^{s}(\R^n)}\leq (C_\Omega+1)^{\frac{1}{2}}\|v\|_{H^{s}(\Omega)}\,.$$
%In particular, if $\partial \Omega$ is smooth enough, then $H^{s}_{00}(\Omega)=\big\{\widetilde v : v\in  H^s(\Omega)\big\}$ 
%(see \cite[Corollary~1.4.4.5]{G}), and  

\begin{equation}\label{densitysmoothH1/200}
H^{s}_{00}(\Omega)= \overline{\mathscr{D}(\Omega)}^{\,\|\cdot\|_{H^{s}(\mathbb{R}^n)}}
 \end{equation}
(see \cite[Theorem~1.4.2.2]{G}) . The topological dual space of $H^{s}_{00}(\Omega)$ is denoted by $H^{-s}(\Omega)$.
\vskip5pt

We are mostly interested in the class of functions 
$$\widehat{H}^{s}(\Omega):=\Big\{u\in L^{2}_{\rm loc}(\mathbb{R}^n) : \mathcal{E}_s(u,\Omega)<\infty\Big\} \,.$$
The following properties hold for any open subsets $\Omega$ and $\Omega'$ of $\R^n$: 
\begin{itemize}
\vskip5pt
\item $ \widehat{H}^{s}(\Omega)$ is a linear space; 
\vskip5pt
\item $ \widehat{H}^{s}(\Omega) \subset  \widehat{H}^{s}(\Omega')$ whenever $\Omega'\subset\Omega$, and    
$ \mathcal{E}_s(\cdot,\Omega')\leq \mathcal{E}_s(\cdot,\Omega)\,$;
\vskip5pt
\item  if $\Omega^\prime$ is  bounded, then $\widehat{H}^{s}(\Omega)\cap H^{s}_{\rm loc}(\mathbb{R}^n) \subset  \widehat{H}^{s}(\Omega')\,$;
\vskip5pt
\item if $\Omega$ is bounded, then $H^{s}_{\rm loc}(\mathbb{R}^n)\cap L^\infty(\mathbb{R}^n) \subset  \widehat{H}^{s}(\Omega)\,$.
\end{itemize}
From Lemma \ref{adminHchap} below, it follows that $\widehat{H}^{s}(\Omega)$ is  a Hilbert space for the scalar product induced by the Hilbertian norm $u\mapsto \|u\|_{\widehat H^s(\Omega)}:=\big(\|u\|^2_{L^2(\Omega)}+ \mathcal{E}_s(u,\Omega)\big)^{1/2}$  (see e.g. \cite{MSK} and \cite[proof of Lemma~2.1]{MS}). 

\begin{lemma}\label{adminHchap}
Let $x_0\in\Omega$ and $\rho>0$ be  such that $D_{\rho}(x_0)\subset\Omega$.  There exists a constant $C_\rho=C_\rho(\rho,n,s)>0$ such that 
$$\int_{\R^n}\frac{|u(x)|^2}{(|x-x_0|+1)^{n+2s}}\,\de x\leq C_{\rho}\left(\mathcal{E}_s\big(u,D_{\rho}(x_0)\big)+\|u\|^2_{L^2(D_{\rho}(x_0))}\right)$$
for every $u\in  \widehat{H}^{s}(\Omega)$. 
\end{lemma}

\begin{remark}\label{remweakcvHhat}
From the Hilbertian structure of  $\widehat{H}^{s}(\Omega)$, it follows that any bounded sequence $\{u_k\}$ in $\widehat{H}^{s}(\Omega)$ admits a subsequence 
converging weakly in $\widehat{H}^{s}(\Omega)$. In addition, if $u_k\rightharpoonup u$ weakly in $\widehat{H}^{s}(\Omega)$, then $u_k\to u$ strongly in $L^2(\Omega)$ by the compact embedding 
$H^{s}(\Omega)\hookrightarrow L^2(\Omega)$. In particular, $\|u_k\|_{L^2(\Omega)}\to \|u\|_{L^2(\Omega)}$.  Since $\liminf_k\|u_k\|_{\widehat{H}^{s}(\Omega)}\geq \|u\|_{\widehat{H}^{s}(\Omega)}$, it follows that 
$\liminf_k\mathcal{E}_s(u_k,\Omega)\geq \mathcal{E}_s(u,\Omega)$. 
\end{remark}
%
%\begin{remark}\label{remEDB}
%If $v\in \widehat{H}^{s}(\Omega)$, then  
%$v+ H^{s}_{00}(\Omega)\subset  \widehat{H}^{s}(\Omega)$.   
%Conversely,  if 
%$v=g$ a.e. in $\R^n\setminus\Omega$ for some functions $v$ and $g$ in $\widehat{H}^{s}(\Omega)$, then $v-g\in H^{s}_{00}(\Omega)$. 
%As a consequence, for $g\in \widehat{H}^{s}(\Omega)$, 
%$$H^{s}_g(\Omega):= \left\{v\in \widehat{H}^{s}(\Omega;\mathbb{R}^m) : v=g\, \text{ a.e. in $\R^n\setminus\Omega$}\right\}=g+H^{s}_{00}(\Omega)\,.$$ 
%Note that  $H^{s}_g(\Omega)\subset  H^{s}_{\rm loc}(\mathbb{R}^n)$ whenever $g\in \widehat{H}^{s}(\Omega)\cap H^{s}_{\rm loc}(\mathbb{R}^n)$. 
%\end{remark}
%

%%%%%%%%%%%%%%%%%%%%%%%%%%%%%%%%%%%%%%%%%%%%%%%%%%%%%%%

\subsection{Fractional operators and compensated compactness}\label{sectoperandcompcomp}

%%%%%%%%%%%%%%%%%%%%%%%%%%%%%%%%%%%%%%%%%%%%%%%%%%%%%%%

Given  an open set $\Omega\subset \mathbb{R}^n$, the fractional  Laplacian $(-\Delta)^s$ in $\Omega$ is defined as the continuous linear operator  $(-\Delta)^s: \widehat{H}^{s}(\Omega)\to (\widehat{H}^{s}(\Omega))^\prime$ induced by the quadratic form $\mathcal{E}_s(\cdot,\Omega)$. In other words, the weak form of the fractional Laplacian $ (-\Delta)^{s} u$   of a given function $u\in \widehat{H}^{s}(\Omega)$ is defined 
%its {\it distributional fractional Laplacian} $ (-\Delta)^{s} v$ 
through its action on $ \widehat{H}^{s}(\Omega)$ by 
\begin{equation}\label{deffraclap}
\big\langle  (-\Delta)^{s} u, \varphi\big\rangle_\Omega:=\frac{\gamma_{n,s}}{2}\iint_{(\R^n\times\R^n)\setminus(\Omega^c\times\Omega^c)}  \frac{\big(u(x)-u(y)\big)\big(\varphi(x)-\varphi(y)\big)}{|x-y|^{n+2s}}\,\de x\de y\,.
%\\ + \gamma_{n,s} \iint_{\Omega\times \Omega^c}  \frac{\big(v(x)-v(y)\big)\cdot\big(\varphi(x)-\varphi(y)\big)}{|x-y|^{n+2s}}\,\de x\de y \,. 
\end{equation}
%If $v$ is a smooth bounded function, then the distribution $ (-\Delta)^{s} v$ can be rewritten from \eqref{deffraclap} as a pointwise defined function which coincides with the one given by  formula~\eqref{formpvfraclap}. 
Notice that the restriction of the linear form $ (-\Delta)^{s} u$ to the subspace $H^{s}_{00}(\Omega)$ belongs to $H^{-s}(\Omega)$ with the  estimate 
%\begin{equation}\label{estinormH-1/2fraclap}
$\| (-\Delta)^{s} u\|^2_{H^{-s}(\Omega)}\leq 2\mathcal E_s(u,\Omega)$. 
%\end{equation}
%In this way, $(-\Delta)^{s} v $ appears to be the first outer variation of   $\mathcal{E}(\cdot,\Omega)$ at $v$ with respect to pertubations supported in $\Omega$, i.e.,  
%\begin{equation}\label{firstvarcalE}
%\big\langle  (-\Delta)^{s} v, \varphi\big\rangle_\Omega=\left[\frac{\de}{\de t} \mathcal E (v+t \varphi , \Omega)\right]_{t=0} 
%\end{equation}
%for all $\varphi \in H^{s}_{00}(\Omega)$.

\begin{remark}\label{localityLapls}
Notice the operator $(-\Delta)^s$ has  the following local property: if  $ u\in\widehat{H}^{s}(\Omega)$ and $\Omega^\prime\subset\Omega$ is an open subset, then
$$\big\langle(-\Delta)^su,\varphi\big\rangle_\Omega= \big\langle(-\Delta)^su,\varphi\big\rangle_{\Omega^\prime}\quad\forall \varphi\in H^s_{00}(\Omega^\prime)\,.$$ 
\end{remark}
\vskip5pt

Following \cite{MazSchi}, we now relate the fractional Laplacian $ (-\Delta)^{s}$ to suitable notions of {\sl fractional gradient} and {\sl fractional divergence}. To this purpose, we first need to recall  from \cite{MazSchi} the notion of (fractional) {\sl ``$s$-vector field''} over a  domain. The space of $s$-vector fields in $\Omega$, that we shall denote by $L^2_{\rm od}(\Omega)$ (in agreement with \cite{MazSchi}),  is defined as the Lebesgue space of $L^2$-{\sl scalar functions} over the open set $(\R^n\times\R^n)\setminus(\Omega^c\times\Omega^c)\subset \R^{2n}$
with respect to the measure $|x-y|^{-n}\de x\de y$. In other words, 
$$L^2_{\rm od}(\Omega):=\Big\{F:(\R^n\times\R^n)\setminus(\Omega^c\times\Omega^c)\to\R: \|F\|_{L^2_{\rm od}(\Omega)}<\infty\Big\}\,, $$
with  
$$\|F\|^2_{L^2_{\rm od}(\Omega)}:= \iint_{(\R^n\times\R^n)\setminus(\Omega^c\times\Omega^c)}\frac{|F(x,y)|^2}{|x-y|^n}\,\de x\de y \,.$$
We endow  $L^2_{\rm od}(\Omega)$ with the (pointwise)  product operator $\odot: L^2_{\rm od}(\Omega)\times  L^2_{\rm od}(\Omega) \times \to L^1(\Omega)$ given by
$$F\odot G(x):=\int_{\R^n}\frac{F(x,y)G(x,y)}{|x-y|^{n}}\,\de y\,. $$
Note that $\odot$ is a continuous  bilinear operator thanks to Fubini's theorem, and it plays the role of ``pointwise scalar product'' between two $s$-vector fields.  With this respect, we  define the (pointwise) ``squared modulus''  of a $s$-vector field $F\in L^2_{\rm od}(\Omega)$ by 
\begin{equation}\label{modsqsvectfield}
|F|^2:=F\odot F \in L^1(\Omega)\,. 
\end{equation}
The (fractional) $s$-gradient is defined in \cite{MazSchi} as a linear operator from the space of scalar valued functions $\widehat H^s(\Omega)$ into the space of $s$-vector fields over $\Omega$. More precisely, we define it as the continuous linear operator $\de_s:\widehat H^s(\Omega)\to L^2_{\rm od}(\Omega)$ given by 
\begin{equation}\label{defsgrad}
\de_su(x,y):= \frac{\sqrt{\gamma_{n,s}}}{\sqrt{2}}\,\frac{u(x)-u(y)}{|x-y|^s}\,.
\end{equation}
Obviously, one has 
$$\|{\rm d}_su\|^2_{L^2_{\rm od}(\Omega)}= 2\mathcal{E}_s(u,\Omega)\quad\text{and}\quad\big\| |{\rm d}_s u|^2\big\|_{L^1(\Omega)}\leq  2\mathcal{E}_s(u,\Omega)$$ 
for every $u\in \widehat H^s(\Omega)$. 
\vskip3pt

In turn, the (fractional) $s$-divergence, denoted by ${\rm div}_s$,  is defined by duality as the adjoint operator to the  $s$-gradient operator restricted to $H^s_{00}(\Omega)$. To do so, the main observation is that for $F\in L^2_{\rm od}(\Omega)$, we have 
$$F\odot {\rm d}_s\varphi \in L^1(\R^n) \quad\text{for every $\varphi\in H^s_{00}(\Omega)$}\,,$$
with
$$\|F\odot {\rm d}_s\varphi \|_{L^1(\R^n)}\leq \|F\|_{L^2_{\rm od}(\Omega)}[\varphi]_{H^s(\R^n)}\,. $$
In this way, we can indeed define ${\rm div}_s: L^2_{\rm od}(\Omega)\to H^{-s}(\Omega)$ as  the continuous linear operator given by 
$$\big\langle {\rm div}_s F,\varphi\big\rangle_\Omega:=\int_{\R^n} F\odot \de_s\varphi\,\de x \qquad\forall \varphi\in  H^s_{00}(\Omega)\,,$$
which satisfies the estimate 
%Note that, once again,  the restriction of ${\rm div}_s F$ to $H^s_{00}(\Omega)$ belongs to $H^{-s}(\Omega)$ with the estimate 
$\|{\rm div}_s F\|_{H^{-s}(\Omega)}\leq \|F\|_{L^2_{\rm od}(\Omega)}$ for all $F\in L^2_{\rm od}(\Omega)$. 
\vskip3pt

From the definition of ${\rm d}_s$ and ${\rm div}_s$, it  readily follows that 

\begin{proposition}\label{fracintegbypart}
We have $(-\Delta)^s={\rm div}_s({\rm d}_s)$, i.e., 
%Note that for every $u,\varphi \in \widehat H^s(\Omega)$, we have 
$$\big\langle  (-\Delta)^{s} u, \varphi\big\rangle_\Omega=\int_{\R^n} {\rm d}_s u\odot{\rm d}_s\varphi\,\de x 
%+ \frac{\gamma_{n,s}}{2}\iint_{\Omega\times\Omega^c}  \frac{\big(u(x)-u(y)\big)\big(\varphi(x)-\varphi(y)\big)}{|x-y|^{n+2s}}\,\de x\de y\,,
$$
for every $u \in \widehat H^s(\Omega)$ and every $\varphi\in H^s_{00}(\Omega)$. 
%which can be thought as some sort of integration by parts formula. In particular, in the case of the whole space $\Omega=\R^n$, we have  the identity $(-\Delta)^s={\rm div}_s({\rm d}_s)$. 
\end{proposition}

One of the main results in \cite{MazSchi} is a compensated compactness result relative to the $s$-gradient and $s$-divergence operators in the spirit of the classical ``div-curl'' lemma \cite{CLMS}. To present this result, let us recall that the space ${\rm BMO}(\R^n)$ is defined as the set of all $u\in L^1_{\rm loc}(\R^n)$ such that
$$[u]_{{\rm BMO}(\R^n)}:= \sup_{D_{r}(y)}\, \dashint_{D_r(y)}|u-(u)_{y,r}|\,\de x<+\infty\,, $$
where $(u)_{y,r}$ denotes the average of $u$ over the ball $D_r(y)$. The following theorem corresponds to  \cite[Proposition 2.4]{MazSchi}. 

\begin{theorem}\label{divcurlthm}
Let $F\in L^2_{\rm od}(\Omega)$ be such that 
$${\rm div}_s F=0 \quad \text{in $H^{-s}(\Omega)$}\,.$$
There exist a universal $\Lambda>1$ such that for every ball $D_r(x_0)$ satisfying $D_{\Lambda r}(x_0)\subset \Omega$,  
$$\left|\int_{\R^n}\big(F\odot{\rm d}_s u\big)\varphi\,\de x\right|\leq C\|F\|_{L^2_{\rm od}(\Omega)}\sqrt{\mathcal{E}_s(u,\Omega)} \Big([\varphi]_{\rm BMO(\R^n)} +r^{-n}\|\varphi\|_{L^1(\R^n)} \Big)$$
for every $u\in  \widehat H^s(\Omega)$ and $\varphi \in\mathscr{D}(D_r(x_0))$, and a constant  $C=C(n,s)$. 
\end{theorem}

\begin{remark}
In the statement of \cite[Proposition 2.4]{MazSchi}, the $s$-vector field $F$ is assumed to be $s$-divergence free in the whole $\R^n$ and $u\in H^s(\R^n)$. However, a careful reading of the proof  reveals that only the assumptions in Theorem \ref{divcurlthm} on $F$ and $u$ are used. 
\end{remark}

%%%%%%%%%%%%%%%%%%%%%%%%%%%%%%%%%%%%%%%%%%%%%%%%%%%%%%%

\subsection{Weighted Sobolev spaces}\label{secWeightSob}

%%%%%%%%%%%%%%%%%%%%%%%%%%%%%%%%%%%%%%%%%%%%%%%%%%%%%%%

For an open set $G\subset \R^{n+1}$, we consider the weighted $L^2$-space  
$$L^2(G,|z|^a\de \mathbf{x}):= \Big\{v\in L^1_{\rm loc}(G) :  |z|^{\frac{a}{2}} v \in L^2(G)\Big\} \quad \text{with }a:=1-2s\,,$$
normed by 
$$\|v\|^2_{L^2(G,|z|^a\de \mathbf{x})}:=\int_G |z|^{a}|v|^2\,\de \mathbf{x}\,. $$
Accordingly, we introduce the weighted Sobolev space 
$$H^1(G,|z|^a\de \mathbf{x}):= \Big\{v\in L^2(G,|z|^a\de \mathbf{x}) : \nabla v \in L^2(G,|z|^a\de \mathbf{x})\Big\} \,,$$ 
normed by 
$$\|v\|_{H^1(G,|z|^a\de \mathbf{x})}:=\|v\|_{L^2(G,|z|^a\de \mathbf{x})}+\|\nabla v\|_{L^2(G,|z|^a\de \mathbf{x})}\,. $$
Both $L^2(G,|z|^a\de \mathbf{x})$ and $H^1(G,|z|^a\de \mathbf{x})$ are separable Hilbert spaces when equipped with the scalar product induced by their respective Hilbertian norms. 

On $H^1(G,|z|^a\de \mathbf{x})$, we define {\sl the weighted Dirichlet energy ${\bf E}_s(\cdot,G)$} by setting 
\begin{equation}\label{defweightDir}
{\bf E}_s(v,G):=\frac{\boldsymbol{\delta}_s}{2}\int_G|z|^a|\nabla v|^2\,\de{\bf x}\quad\text{with } \boldsymbol{\delta}_s:=2^{2s-1}\frac{\Gamma(s)}{\Gamma(1-s)}\,.
\end{equation}
The relevance of the normalisation constant  $\boldsymbol{\delta}_s>0$ will be revealed in Section \ref{fracharmext} (see \eqref{normexth1/2}). 
\vskip3pt

%{\bf Useful ??} If  $\Omega$ denotes a (relatively) open subset of $ \partial\R^{n+1}_+\simeq \R^n$ such that  $\Omega\subset\partial G$, we also set 
%$$L^2_{{\rm loc}}(G\cup\Omega;\R^d, |z|^a\de \mathbf{x}):= \Big\{v\in L^1_{\rm loc}(G;\R^d) :  |z|^{\frac{a}{2}} v \in L^2_{\rm loc}(G\cup\Omega;\R^d)\Big\} \,,$$
%and
%$$ H^1_{{\rm loc}}(G\cup\Omega;\R^d,|z|^a\de \mathbf{x}) := \Big\{v\in L^2_{{\rm loc}}(G\cup\Omega;\R^d,|z|^a\de \mathbf{x}) : \nabla v \in L^2_{{\rm loc}}(G\cup\Omega;\R^{d\times n},|z|^a\de \mathbf{x})\Big\}\,.$$
%\vskip3pt

%\begin{remark}\label{trace}

Some relevant remarks about $H^1(G,|z|^a\de \mathbf{x})$ are in order. For a bounded admissible open set $G\subset \R^{n+1}_+$, the space $L^2(G,|z|^a\de \mathbf{x})$ embeds continuously into $L^\gamma(G)$  for every $1\leq \gamma<\frac{1}{1-s}$ whenever $s\in(0,1/2)$ by H\"older's inequality. For $s\in[1/2,1)$, we have 
$L^2(G,|z|^a\de \mathbf{x})\hookrightarrow L^2(G)$ continuously since $a\leq 0$.   In any case, it implies that 
\begin{equation}\label{contembedd}
H^1(G,|z|^a\de \mathbf{x})\hookrightarrow W^{1,\gamma}(G) 
\end{equation}
continuously for every $1< \gamma<\min\{\frac{1}{1-s},2\}$. As a first consequence,  $H^1(G,|z|^a\de \mathbf{x})\hookrightarrow L^{1}(G)$ with compact embedding. Secondly, for such $\gamma$'s,  the compact  linear trace operator 
\begin{equation}\label{conttracW1}
v\in W^{1,\gamma}(G)\mapsto v_{|\partial^0 G}\in L^1(\partial^0G)
\end{equation}
induces a compact linear trace operator from $H^1(G,|z|^a\de \mathbf{x})$ into $L^1(\partial^0G)$,  extending the usual trace of smooth functions. We shall denote by $v_{|\partial^0G}$ the trace of $v\in H^1(G,|z|^a\de \mathbf{x})$ on $\partial^0G$, or simply by $v$ if it is clear from the context. We may now recall the following Poincar\'e's inequality, see e.g. \cite[Lemma 2.5]{MSK}. 

\begin{lemma}\label{poincarebdry}
If $v\in H^1(B_r^+,|z|^a\de{\bf x})$, then 
$$ \big\| v - (v)_r  \big\|_{L^1(D_r)} \leq C r^{\frac{n+2s}{2}}\|\nabla v\|_{L^2(B_r^+,|z|^a\de{\bf x})}\,,$$
for a constant $C=C(n,s)$, where $(v)_r$ denotes the average of $v$ over $D_r$.  
\end{lemma}

The next lemma states that the trace $v_{|\partial^0G}$ has actually $H^s$-regularity, at least locally.  

\begin{lemma}\label{HsregtraceH1weight}
If $v\in H^1(B^+_{2r},|z|^a\de \mathbf{x})$, then the trace of $v$ on $\partial^0B_r^+\simeq D_r$ belongs to  $H^s(D_r)$, and  
%$$\iint_{D_r\times D_r} \frac{|v(x)-v(y)|^2}{|x-y|^2}\,\de x\de y
$$[v]^2_{H^s(D_r)}
\leq C \,{\bf E}_s(v,B^+_{2r})\,,$$
%\int_{B^+_{2r}}z^a|\nabla v|^2\,\de{\bf x}\,,$$
for a constant $C=C(n,s)$. 
\end{lemma}

\begin{proof}
The proof follows exactly the one in \cite[Lemma 2.3]{MSY} which is stated only in dimension $n=1$. We reproduce the proof (in arbitrary dimension) for convenience of the reader, slightly anticipating a well-known  identity presented in Section \ref{fracharmext} (see \eqref{normexth1/2}). 

Rescaling variables, we can assume that $r=1$. Moreover, we may assume without loss of generality that $v$ has a vanishing average over the half ball $B^+_{2}$. Let $\zeta\in C^\infty(B_{2};[0,1])$ be a cut-off function such that $\zeta({\bf x})=1$ for $|{\bf x}|\leq 1$,  $\zeta({\bf x})=0$ for $|{\bf x}|\geq 3/2$. The function $v_*:=\zeta v$ belongs to $H^1(\R^{n+1}_+, |z|^a\de{\bf x})$, and Poincar\'e's inequality in $H^1(\R^{n+1}_+, |z|^a\de{\bf x})$ (see e.g.~\cite{FKS}) yields 
\begin{equation}\label{esticpa1}
\int_{\R^{n+1}_+}z^a|\nabla v_*|^2\,\de {\bf x}\leq 2{\bf E}_s(v,B^+_{2})+ C\int_{B^+_{2}}z^a|v|^2\,\de{\bf x}\leq C{\bf E}_s(v,B^+_{2})\,,
\end{equation}
for a constant $C=C(\zeta,n,s)$.  On the other hand, it follows from \eqref{normexth1/2} in Section \ref{fracharmext} below that
\begin{equation}\label{esticpa2}
\iint_{D_1\times D_1}\frac{|v(x)-v(y)|^2}{|x-y|^{1+2s}}\,\de x\de y\leq \iint_{\R^n\times\R^n}\frac{|v_*(x)-v_*(y)|^2}{|x-y|^{n+2s}}\,\de x\de y
\leq C {\bf E}_s(v_*,\R^{n+1}_+)\,.
%\int_{\R^{n+1}_+}z^a|\nabla v_*|^2\,\de {\bf x}\,.
\end{equation}
Gathering \eqref{esticpa1} and \eqref{esticpa2} leads to the announced estimate.
\end{proof}

\subsection{Fractional harmonic extension and the Dirichlet-to-Neumann operator}\label{fracharmext}

%%%%%%%%%%%%%%%%%%%%%%%%%%%%%%%%%%%%%%%%%%%%%%%%%%%%%%%

Let us consider the so-called fractional Poisson kernel  $\mathbf{P}_{n,s}:\R^{n+1}_+\to [0,\infty)$ defined by 
\begin{equation}\label{defpoisskern}
\mathbf{P}_{n,s}(\mathbf{x}):=\sigma_{n,s}\,\frac{z^{2s}}{|\mathbf{x}|^{n+2s}}\qquad \text{with } \sigma_{n,s}:=\pi^{-\frac{n}{2}}\frac{\Gamma(\frac{n+2s}{2})}{\Gamma(s)}\,,
\end{equation}
where $\mathbf{x}:=(x,z)\in\mathbb{R}^{n+1}_+:=\R^n\times(0,\infty)$. The choice of the constant $\sigma_{n,s}$ is made in such a way that
%\footnote{Indeed, changing variables one obtains
%\begin{multline*}
%\int_{\R^n}\frac{1}{(|x|^2+1)^{\frac{n+2s}{2}}}\,\de x=|\mathbb{S}^{n-1}|\int_0^\infty\frac{r^{n-1}}{(r^2+1)^{\frac{n+2s}{2}}}\,\de r\\
%=\frac{|\mathbb{S}^{n-1}|}{2}\int_0^\infty\frac{t^{\frac{n}{2}-1}}{(t+1)^\frac{n+2s}{2}}\,\de t= \frac{|\mathbb{S}^{n-1}|}{2}\,{\rm B}(n/2,s)\,,
%\end{multline*}
%where ${\rm B}(\cdot,\cdot)$ denotes the Euler Beta function.}   
$\int_{\R^n}\mathbf{P}_{n,s}(x,z)\,\de x=1$ for every $z>0$ (see e.g. the computation in Remark \ref{remcomputmasspoisskern}).  As shown in  \cite{CaffSil} (see also \cite{MolOs}), the function $\mathbf{P}_{n,s}$ solves 
$$\begin{cases}
{\rm div}(z^{a}\nabla \mathbf{P}_{n,s})= 0 & \text{in $\R^{n+1}_+$}\,,\\
\mathbf{P}_{n,s}=\delta_0 & \text{on $\partial\R^{n+1}_+$}\,,
\end{cases}$$
where $\delta_0$ denotes the Dirac distribution at the origin. 
%In other words, the function $\mathbf{P}_{n,s}$ can be interpreted as the {\it ``fractional Poisson kernel''} by analogy with the standard case $s=1/2$. 

From now on, for a measurable function $u$ defined over $\mathbb{R}^n$, we shall denote by 
$u^\e$ its extension to the half-space $\mathbb{R}^{n+1}_+$ given by the convolution (in the $x$-variable) of $u$ with 
 $\mathbf{P}_{n,s}$,  i.e., 
\begin{equation}\label{poisson}
u^\e(x,z):= \sigma_{n,s}\int_{\mathbb{R}^n}\frac{z^{2s} u(y)}{(|x-y|^2+z^2)^{\frac{n+2s}{2}}}\,\de y\,.
\end{equation}
Notice that $u^\e$ is well defined if $u$ belongs to the Lebesgue space $L^1$ over $\R^n$ with respect to the probability measure 
\begin{equation}\label{defmeasm}
\mathfrak{m}_s:=\sigma_{n,s}(1+|y|^2)^{-\frac{n+2s}{2}}\,\de y\,.
\end{equation}
In particular,  $u^\e$ can be defined  whenever 
$u\in \widehat{H}^{s}(\Omega)$ for some  open set $\Omega\subset\R^n$ by Lemma~\ref{adminHchap}. 
Moreover, if  $u\in L^\infty(\R^n)$, then $u^\e\in L^\infty(\R_+^{n+1})$ and  
\begin{equation}\label{bdlinftyext}
\|u^\e\|_{L^\infty(\R_+^{n+1})} \leq \|u\|_{L^\infty(\R^n)}\,.
\end{equation}
For a function $u\in L^1(\R^n,\mathfrak{m}_s)$, the extension  $u^\e$ has a pointwise trace on $\partial\R^{n+1}_+\simeq\R^n$ which is equal to $u$ at every Lebesgue point. 
 In addition, $u^\e$ solves the equation 
\begin{equation}\label{eqextharm}
\begin{cases} 
{\rm div}(z^{a}\nabla u^\e) = 0 & \text{in $\mathbb{R}_+^{n+1}$}\,,\\
u^\e = u  & \text{on $\partial\R^{n+1}_+$}\,. 
\end{cases}
\end{equation}
By analogy with the standard case $s=1/2$ (for which \eqref{eqextharm} reduces to the Laplace equation), the map $u^\e$ is referred to as the {\it fractional harmonic extension} of $u$. 
\vskip3pt

%The following continuity property is elementary and can be obtained exactly as in \cite[Lemma~2.5]{MilSir}. 
%
%\begin{lemma}\label{context}
%For every $R>0$, the restriction operator $\mathfrak{R}_R:L^2(\R^n,\mathfrak{m}_s)\to L^2(B_R^+,|z|^a\de\mathbf{x})$ defined by 
%\begin{equation}\label{Pfrak}
%\mathfrak{R}_R(u):={u^\e}_{|B_R^+} \,,
%\end{equation}
%is continuous. 
%\end{lemma}

It has been proved in \cite{CaffSil} that $u^\e$  belongs to the weighted space $H^1(\mathbb{R}_+^{n+1},|z|^a\de\mathbf{x})$ whenever 
$u\in H^{s}(\mathbb{R}^n)$. Extending a well-known identity for $s=1/2$, the $H^{s}$-seminorm of $u$ coincides up to a multiplicative constant with the weighted $L^2$-norm of $\nabla u^\e$, and $u^\e$ turns out to minimize the weighted Dirichlet energy among all possible extensions. In other words, 
\begin{equation}\label{normexth1/2}
[u]^2_{H^{s}(\mathbb{R}^n)}={\bf E}_s(u^\e,\R^{n+1}_+)=\inf\Big\{{\bf E}_s(v,\R^{n+1}_+): v\in H^1(\R^{n+1}_+,|z|^a\de{\bf x})\,,\; v=u\text{ on }\R^n \Big\}
\end{equation}
for every $u\in H^s(\R^n)$ (thanks to the choice of the normalisation  factor $\boldsymbol{\delta}_s$ in \eqref{defweightDir}). 
%As in the case $s=1/2$ (which reduces to the usual harmonic extension), the map $u^\e$ turns out to be energy minimizing  
%
%
%\begin{lemma}[\cite{CaffSil}]\label{normexth1/2}
%Let $v\in H^{s}(\mathbb{R}^n)$, and let $v^\e$ be its fractional harmonic extension to~$\mathbb{R}^{n+1}_+$ given by \eqref{poisson}. Then 
%$v^\e$ belongs to $H^1(\mathbb{R}_+^{n+1},|z|^a\de \mathbf{x})$ and
%\begin{align}
%\nonumber [v]^2_{H^{s}(\mathbb{R}^n)} &= d_s\|\nabla v^\e\|^2_{L^2(\R_+^{n+1},|z|^a\de \mathbf{x})} \\
%\label{isomtry}&=\inf\left\{ d_s\|\nabla u\|^2_{L^2(\R_+^{n+1},|z|^a\de \mathbf{x})} : u\in H^1(\mathbb{R}^{n+1}_+,|z|^a\de \mathbf{x})\,, \ u=v \text{ on $\mathbb{R}^n$} \right\}\,,
%\end{align}
%where $d_s:=2^{2s-1}\frac{\Gamma(s)}{\Gamma(1-s)}$.
%\end{lemma}

%\begin{remark}\label{remtrace}
%Let  $G\subset \R^{n+1}_+$ be an admissible bounded  open set. For any function $u\in H^1(\mathbb{R}^{n+1}_+,|z|^a\de\mathbf{x})$ compactly supported in $G\cup\partial^0G$, 
%the trace $u_{|\mathbb{R}^n}$  belongs to $H^s_{00}(\partial^0G)$. Indeed, if $u$ is smooth in $\overline{\mathbb{R}^{n+1}_+}$, then we can apply identity  \eqref{isomtry}. In the general case, it suffices to apply the approximation procedure in Remark \ref{smoothapprox} to reach the conclusion. 
%\end{remark}
\vskip3pt

If $u\in  \widehat{H}^{s}(\Omega)$ for some open set $\Omega\subset\R^n$,  
we have the following estimates on $u^\e$, somehow extending the first equality in \eqref{normexth1/2} to the localized setting. 
%The proof follows closely the arguments in \cite[Lemma 2.7]{MilSir}, and we shall omit it. 

\begin{lemma}\label{hatH1/2toH1}
Let $\Omega\subset \mathbb{R}^n$ be an open set. 
For every $u\in \widehat{H}^{s}(\Omega)$, the  extension $u^\e$ 
given by \eqref{poisson} belongs to $H^1(G,|z|^a\de \mathbf{x})\cap L^2_{{\rm loc}}\big(\overline{\R^{n+1}_+},|z|^a\de \mathbf{x}\big)$ for every bounded admissible open set $G\subset\R^{n+1}_+$ satisfying $\overline{\partial^0G}\subset\Omega$. In addition, for every point ${\bf x}_0=(x_0,0)\in\Omega\times\{0\}$  and  $r>0$  such that $D_{3r}(x_0)\subset\Omega$, 
%there exist constants $C_{s,R,\rho}>0$  and $C_{s,\rho}>0$, independent of $v$ and $x_0$, such that  
\begin{equation}\label{contruextL2}
\|u^\e\|^2_{L^2(B_r^+({\bf x}_0),|z|^a\de \mathbf{x})}\leq C\left(r^{2}\mathcal{E}_s\big(u,D_{2r}(x_0)\big)+r^{2-2s}\|u\|^2_{L^2(D_{2r}(x_0))}\right)\,, 
\end{equation}
and 
\begin{equation}\label{contruextH1weight}
 {\bf E}_s\big(u^\e,B_r^+({\bf x}_0)\big)\leq C \mathcal{E}_s\big(u,D_{2r}(x_0)\big)\,, 
 \end{equation}
for a constant $C=C(n,s)$. 
\end{lemma}

\begin{proof}
Translating and rescaling variables, we can assume that $x_0=0$ and $r=1$. Then \eqref{contruextL2} follows from \cite[Lemma 2.10]{MSK} (which is stated for $s\in(0,1/2)$, but the proof is in fact valid for any $s\in(0,1)$). Denote by $\bar u$ the average of $u$ over $D_2$. 
Noticing that $(u-\bar u)^\e=u^\e-\bar u$, and applying \cite[Lemma 2.10]{MSK} to $u-\bar u$ yields 
$${\bf E}_s(u^\e,B_1^+)\leq C\big( \mathcal{E}_s(u,D_{2})+ \|u-\bar u\|^2_{L^2(D_2)}\big)\,.$$
On the other hand, by Poincar\'e's inequality in $H^s(D_2)$, we have
$$ \|u-\bar u\|^2_{L^2(D_2)}\leq C [u]^2_{H^s(D_2)}\leq  C  \mathcal{E}_s(u,D_{2})\,,$$
and \eqref{contruextH1weight} follows. 
\end{proof}

\begin{corollary}\label{contextHsH1}
Let $\Omega\subset \mathbb{R}^n$ be an open set, and $G\subset \R^{n+1}_+$ a bounded admissible open set such that $\overline{\partial^0 G}\subset\Omega$.  The extension operator $u\mapsto u^\e$ defines a continuous linear operator from $\widehat H^s(\Omega)$ into $H^1(G,|z|^a\de{\bf x})$. 
\end{corollary}

\begin{proof}
Set $\delta:={\rm dist}(\partial^0G,\Omega^c)$, and 
$$h_1:=\min\Big\{\frac{\delta}{12}\,,\, \inf\big\{{\rm dist}({\bf x},\partial\R^{n+1}_+): {\bf x}=(x,z)\in G\,,\;{\rm dist}((x,0),\partial^0G)\geq \delta/2\big\}\Big\}>0\,,$$
$$h_2:=\sup\Big\{{\rm dist}({\bf x},\partial\R^{n+1}_+): {\bf x}=(x,z)\in G\Big\}<+\infty\,.$$
We also consider a large radius $R>0$ in such a way that $G\subset D_R\times\R$, and we define 
$$\omega:=\Big\{x\in\R^n :{\rm dist}((x,0),\partial^0G)< \delta/2\  \Big\}\,,$$
and
$$G_*:=\big(\omega\times(0,h_1]\big)\cup \big(D_R\times(h_1,h_2\big)\big)\,.$$
By construction, $G_*$ is a bounded admissible open set satisfying $\overline{\partial^0 G_*}\subset\Omega$ and $G\subset G_*$. Therefore, it is enough to show that the extension operator is continuous from  
$\widehat H^s(\Omega)$ into $H^1(G_*,|z|^a\de{\bf x})$. In other words, we can assume without loss of generality that $G=G_*$. 
\vskip3pt

Covering $\omega\times(0,h_1]$ by finitely many half balls $B^+_{\delta/6}({\bf x_i})$ with ${\bf x}_i\in \omega\times\{0\}$, and applying Lemma \ref{hatH1/2toH1} in those balls, we infer that 
$u^\e\in H^1(\omega\times(0,h_1),|z|^a\de{\bf x})$, and 
$$\|u^\e\|^2_{H^1(\omega\times(0,h_1),|z|^a\de{\bf x})} \leq C_G\big(\mathcal{E}_s(u,\Omega)+\|u\|^2_{L^2(\Omega)}\big)\,, $$
for a constant $C_G=C_G(G,n,s)$. 

On the other hand, one may derive from formula \eqref{poisson} and Jensen's inequality that 
$$|u^\e({\bf x})|^2+ |\nabla u^\e({\bf x})|^2\leq C_G\int_{\R^n}\frac{|u(y)|^2}{(|x-y|^2+h_1^2)}\,\de y\quad \forall {\bf x}=(x,z)\in D_R\times(h_1,h_2)\,.$$
It then follows from Lemma \ref{adminHchap} that $u^\e\in H^1( D_R\times(h_1,h_2),|z|^a\de {\bf x})$ with 
$$\|u^\e\|^2_{H^1(D_R\times(h_1,h_2),|z|^a\de{\bf x})} \leq C_G\big(\mathcal{E}_s(u,\Omega)+\|u\|^2_{L^2(\Omega)}\big)\,, $$
which completes the proof. 
\end{proof}

%\begin{remark}\label{H1/2loctoH1loc}
%By the previous lemma, for any $v\in \widehat H^{s}(\Omega)\cap H^{s}_{\rm loc}(\mathbb{R}^n)$, the fractional harmonic extension $v^\e$ 
%belongs to $H^1_{{{\rm loc}}}(\overline{\R^{n+1}_+},|z|^a\de \mathbf{x})$, and for any $R>0$, 
%$$\big\| v^\e\big\|^2_{H^1(B_R^+,|z|^a\de \mathbf{x})}\leq C_{s,R} \left(\mathcal{E}\big(v,D_{2R}\big)+\|v\|^2_{L^2(D_{2R})}\right)\,. $$
%\end{remark}
%\vskip5pt

Another useful fact about the extension by convolution with  $\mathbf{P}_{n,s}$, is that it preserves some local H\"older continuity. It is very classical and follows from the explicit formula (and regularity) of $\mathbf{P}_{n,s}$. Details are left to the reader. 

\begin{lemma}\label{HoldTransf}
If $u\in L^\infty(\R^n)\cap C^{0,\beta}(D_R)$ for some $\beta\in(0,\min(1,2s))$, then  $u^\e\in C^{0,\beta}(B^+_{R/4})$, and 
\begin{equation}\label{Holdtransfesti}
R^{\beta}[u^\e]_{C^{0,\beta}(B^+_{R/4})}\leq C_\beta\big(R^{\beta}[u]_{C^{0,\beta}(D_R)}+\|u\|_{L^\infty(\R^n)}\big)\,, 
\end{equation}
for a constant $C_\beta=C_\beta(\beta,n,s)$. 
\end{lemma}

Let us now assume that $\Omega\subset\R^n$ is a bounded open set with Lipschitz boundary. If $u\in \widehat H^{s}(\Omega)$, the divergence free vector field $z^{a}\nabla u^\e$ admits a distributional normal trace on $\Omega$, that we denote by $\mathbf{\Lambda}^{(2s)}u$.  
More precisely, we define $\mathbf{\Lambda}^{(2s)} u$  through its action on a test function $\varphi\in \mathscr{D}(\Omega)$ by setting
\begin{equation}\label{defNeumOp}
\left\langle \mathbf{\Lambda}^{(2s)} u, \varphi\right\rangle_\Omega := \int_{\mathbb{R}^{n+1}_+}z^{a}\nabla u^\e\cdot\nabla\Phi\,\de \mathbf{x}\,,
\end{equation}
where $\Phi$ is any smooth extension of $\varphi$ compactly supported in  $\mathbb{R}_+^{n+1}\cup\Omega$. Note that the right-hand side of \eqref{defNeumOp} 
is well defined by Lemma~\ref{hatH1/2toH1}. By the divergence theorem, it is routine to check that 
the integral in \eqref{defNeumOp} does not depend on the choice of the extension $\Phi$. 
%In the light of \eqref{densitysmoothH1/200} and  \eqref{normexth1/2}, we infer that $\mathbf{\Lambda}^{(2s)}:\widehat H^{s}(\Omega)\to H^{-s}(\Omega)$ defines a continuous linear operator. 
It can be thought of as a {\it fractional  Dirichlet-to-Neumann operator}. Indeed, whenever $u$ is smooth, the distribution $\mathbf{\Lambda}^{(2s)}u$ is the pointwise-defined function given by 
$$\mathbf{\Lambda}^{(2s)} u(x)=-\lim_{z\downarrow0}z^{a}\partial_z u^\e(x,z)=2s\, \lim_{z\downarrow0} \frac{u^\e(x,0)-u^\e(x,z)}{z^{2s}}$$  
at  each point $x\in \Omega$.  
\vskip3pt

In the case $\Omega=\R^n$, it has been proved in \cite{CaffSil} that $\mathbf{\Lambda}^{(2s)}$ coincides with $ (-\Delta)^{s} $, up to the multiplicative factor $\boldsymbol{\delta}_s$. In the localized  setting, this identity still holds, see e.g. \cite[Lemma~2.12]{MSK} and \cite[Lemma 2.9]{MS}.

\begin{lemma}\label{repnormderfraclap}
If $\Omega\subset \mathbb{R}^n$ is a bounded open set with  Lipschitz boundary, then 
$$ (-\Delta)^{s} = \boldsymbol{\delta}_s \mathbf{\Lambda}^{(2s)} \text{ on $\widehat H^{s}(\Omega)$}\,.$$
\end{lemma}

One of the main consequences of Lemma \ref{repnormderfraclap} is a local counterpart of \eqref{normexth1/2} concerning the minimality of $u^\e$. 
This is the purpose of Corollary \ref{minenergdirchfrac} below, inspired from \cite[Lemma 7.2]{CRS}, and taken from \cite[Corollary 2.13]{MSK}.
% (note that \cite{MSK} focuses on the case $s\in(0,1/2)$, but the statement and proof of \cite[Corollary 2.13]{MSK} is in fact valid for every $s\in(0,1)$). 

\begin{corollary}\label{minenergdirchfrac}
Let $\Omega\subset \mathbb{R}^n$ be a bounded open set, and $G\subset \R^{n+1}_+$ an admissible bounded  open set  
such that $\overline{\partial^0 G}\subset \Omega$. 
Let $u\in \widehat H^{s}(\Omega;\R^d)$, and let $u^\e$ be its fractional harmonic extension to $\mathbb{R}^{n+1}_+$ given by~\eqref{poisson}. Then, 
\begin{equation}\label{ineqenergDfrac}
{\bf E}_s(v,G)-{\bf E}_s(u^\e,G)
\geq  \mathcal{E}_s(v,\Omega) -\mathcal{E}_s(u,\Omega)
\end{equation}
for all $v\in H^1(G;\R^d,|z|^a\de\mathbf{x})$ such that $v-u^\e$ is compactly supported in $G\cup \partial^0G$. 
In the right-hand side of \eqref{ineqenergDfrac}, the trace of $v$ on $\partial^0G$ is extended by $u$ outside $\partial^0G$. 
\end{corollary}

%%%%%%%%%%%%%%%%%%%%%%%%%%%%%%%%%%%%%%%%%%%%%%%%%%%%%%%

\subsection{Inner variations, monotonicity formula, and density functions}

%%%%%%%%%%%%%%%%%%%%%%%%%%%%%%%%%%%%%%%%%%%%%%%%%%%%%%%

In this section, our main goal is to present the {\sl monotonicity formula} satisfied by critical points of $\mathcal{E}_s(\cdot,\Omega)$ under {\sl inner variations}, i.e., by stationary points. 
We start recalling the notion of first inner variation, and then give an explicit formula to represent it. 

\begin{definition}
Let $\Omega\subset \mathbb{R}^n$ be a bounded open set. Given a map $u\in \widehat H^s(\Omega;\R^d)$ and a vector field $X\in C^1(\R^n;\R^n)$ compactly supported in $\Omega$, the first (inner) variation of $\mathcal{E}_s(\cdot,\Omega)$ at $u$ and evaluated at $X$ is defined as 
$$\delta\mathcal{E}_s(u,\Omega)[X]:= \left[\frac{\de}{\de t}\mathcal{E}_s(u\circ\phi_{-t},\Omega)\right]_{t=0}\,,$$
where $\{\phi_t\}_{t\in\R}$ denotes the integral flow on $\R^n$ generated by $X$, i.e., for every $x\in\R^n$, the map $t\mapsto\phi_t(x)$ is defined as the unique solution of the ordinary differential equation
$$ 
\begin{cases}
\displaystyle \frac{\de}{\de t}\phi_t(x)=X\big(\phi_t(x)\big)\,,\\[5pt]
\phi_0(x)=x\,.
\end{cases}
$$
\end{definition}

The following representation result for $\delta\mathcal{E}_s$ was obtained in \cite[Corollary 2.14]{MSK} as a direct consequence of Corollary \ref{minenergdirchfrac}. We reproduce here the proof for completeness. 

\begin{proposition}\label{represfirstvar}
Let $\Omega\subset \mathbb{R}^n$ be a bounded open set, and $G\subset \R^{n+1}_+$ an admissible bounded  open set  
such that $\overline{\partial^0 G}\subset \Omega$. For each $u\in \widehat H^{s}(\Omega;\R^d)$, and each $X\in C^1(\R^n;\R^n)$ compactly supported in $\partial^0 G$, we have 
\begin{multline}\label{calcfirstvar}
\delta\mathcal{E}_s(u,\Omega)[X]=\frac{\boldsymbol{\delta}_s}{2}\int_Gz^a\Big(|\nabla u^\e|^2{\rm div}{\bf X}-2\sum_{i,j=1}^{n+1}(\partial_iu^\e\cdot\partial_ju^\e)\partial_j{\bf X}_i\Big)\,\de{\bf x}\\
+\frac{\boldsymbol{\delta}_s a}{2}\int_Gz^{a-1}|\nabla u^\e|^2{\bf X}_{n+1}\,\de{\bf x}\,,
\end{multline}
where ${\bf X}=({\bf X}_1,\ldots,{\bf X}_{n+1})\in C^1(\overline G;\R^{n+1})$ is any vector field compactly supported in $G\cup\partial^0G$, and satisfying ${\bf X}=(X,0)$ on $\partial^0G$. 
\end{proposition}

\begin{proof}
Let ${\bf X}\in C^1(\overline G,\R^{n+1})$ be an arbitrary vector field compactly supported in $G\cup\partial^0G$ and satisfying ${\bf X}=(X,0)$ on $\partial^0G$. We consider a compactly supported $C^1$-extension of ${\bf X}$ to the whole space $\R^{n+1}$, still denoted by ${\bf X}$, such that ${\bf X}=(X,0)$ on $\R^n\times\{0\}\simeq\R^n$. We define $\{\boldsymbol{\Phi}_t\}_{t\in\R}$ as the integral flow on $\R^{n+1}$ generated by ${\bf X}$. Observe that $\boldsymbol{\Phi}_t=(\phi_t,0)$ on~$\R^n$, and ${\rm spt}(\boldsymbol{\Phi}_t-{\rm id}_{\R^{n+1}})\cap\overline{\R^{n+1}_+}\subset G\cup\partial^0G$. Then, $v_t:=u^\e\circ\boldsymbol{\Phi}_{-t}\in H^1(G;\R^d,|z|^a\de\mathbf{x})$ and ${\rm spt}(v_t-u^\e)\subset G\cup\partial^0G$. By Corollary \ref{minenergdirchfrac}, we have 
\begin{equation}\label{calcfirstvartrick}
{\bf E}_s(v_t,G)-{\bf E}_s(u^\e,G) \geq  \mathcal{E}_s(v_t,\Omega) -\mathcal{E}_s(u,\Omega)\quad\forall t\in\R\,.
\end{equation}
Since $v_t=u\circ\phi_{-t}$ on $\R^n$, dividing both sides of \eqref{calcfirstvartrick} by $t\not=0$, and letting $t\uparrow0$ and $t\downarrow0$ leads to 
\begin{equation}\label{equalitfirstvars}
\delta\mathcal{E}_s(u,\Omega)[X]= \left[\frac{\de}{\de t}\mathbf{E}_s(u^\e\circ\boldsymbol{\Phi}_{-t},G)\right]_{t=0}\,.
\end{equation}
On the other hand, standard computations (see e.g. \cite[Chapter 2.2]{Sim}) show that the right-hand side of \eqref{equalitfirstvars} is equal to the right-hand side of \eqref{calcfirstvar}. 
\end{proof}

\begin{definition}\label{defstatmap}
Let $\Omega\subset \mathbb{R}^n$ be a bounded open set. A map $u\in\widehat H^s(\Omega;\R^d)$ is said to be {\sl stationary} in $\Omega$ if $\delta\mathcal{E}_s(u,\Omega)=0$. 
\end{definition}

As we shall see in the next sections, stationarity is a crucial ingredient in the partial regularity theory since it implies the aforementioned monotonicity formula. This is the purpose of the following proposition  
whose proof follows exactly \cite[Proof of Lemma 4.2]{MSK} using vector fields in \eqref{calcfirstvar} of the form ${\bf X}=\eta(|{\bf x}-{\bf x}_0|)({\bf x}-{\bf x}_0)$ with $\eta(t)\sim \chi_{[0,r]}(t)$. 

\begin{proposition}\label{monotformula}
Let $\Omega\subset \mathbb{R}^n$ be a bounded open set. If  $u\in\widehat H^s(\Omega;\R^d)$ is stationary in $\Omega$,  then for every ${\bf x}_0=(x_0,0)\in\Omega\times\{0\}$, the ``density function'' 
$$r\in(0,{\rm dist}(x_0,\Omega^c))\mapsto \boldsymbol{\Theta}_s(u^\e,{\bf x}_0,r):=\frac{1}{r^{n-2s}}{\bf E}_s(u^\e,B^+_r({\bf x}_0))$$
is nondecreasing. Moreover, 
$$\boldsymbol{\Theta}_s(u^\e,{\bf x}_0,r)-\boldsymbol{\Theta}_s(u^\e,{\bf x}_0,\rho) = \boldsymbol{\delta}_s\int_{B^+_r({\bf x}_0)\setminus B^+_\rho({\bf x}_0)}z^a\frac{|({\bf x}-{\bf x}_0)\cdot\nabla u^\e|^2}{|{\bf x}-{\bf x}_0|^{n+2-2s}}\,\de {\bf x}$$
for every $0<\rho<r<{\rm dist}(x_0,\Omega^c)$. 
\end{proposition}

As a straightforward consequence, we have 

\begin{corollary}\label{corolmonotform}
Let $\Omega\subset \mathbb{R}^n$ be a bounded open set. If  $u\in\widehat H^s(\Omega;\R^d)$ is stationary in $\Omega$,  then for every $x_0\in\Omega$, the limit 
\begin{equation}\label{deflimitdens}
\boldsymbol{\Xi}_s(u,x_0):=\lim_{r\to 0}  \boldsymbol{\Theta}_s\big(u^\e, (x_0,0),r\big)
\end{equation}
exists, and the function $\boldsymbol{\Xi}_s(u,\cdot):\Omega\to[0,\infty)$ is upper semicontinuous. In addition, for every ${\bf x}_0=(x_0,0)\in\Omega\times\{0\}$, 
\begin{equation}\label{monotformCor}
\boldsymbol{\Theta}_s(u^\e,{\bf x}_0,r)-\boldsymbol{\Xi}_s(u,x_0) = \boldsymbol{\delta}_s\int_{B^+_r({\bf x}_0)}z^a\frac{|({\bf x}-{\bf x}_0)\cdot\nabla u^\e|^2}{|{\bf x}-{\bf x}_0|^{n+2-2s}}\,\de {\bf x}
\end{equation}
for every $0<r<{\rm dist}(x_0,\Omega^c)$. 
\end{corollary}

\begin{proof}
The existence of the limit in \eqref{deflimitdens} and \eqref{monotformCor} are direct consequences of the monotonicity formula established in Proposition \ref{monotformula}. Then the function $\boldsymbol{\Xi}_s(u,\cdot)$ is upper semicontinuous as a pointwise limit of a decreasing family of continuous functions. 
\end{proof}

As we previously said, the monotonicity of the density function  $r\mapsto\boldsymbol{\Theta}_s(u^\e,{\bf x}_0,r)$ is one of the most important ingredients to obtain partial regularity. We shall see in the next sections that the density function relative to the nonlocal energy $\mathcal{E}_s$ also plays a role. For $u\in \widehat H^s(\Omega;\R^d)$ and a point $x\in\Omega$, we define the density 
function $r\in(0,{\rm dist}(x,\Omega^c))\mapsto \boldsymbol{\theta}_s(u,x,r)$ by setting 
 \begin{equation}\label{defnonlocdensit}
 \boldsymbol{\theta}_s(u,x_0,r):=\frac{1}{r^{n-2s}}\mathcal{E}_s\big(u,D_r(x_0)\big)\,.
 \end{equation}
Now we aim to show that one density function is small if and only the other one is also small at a comparable scale. This is the purpose of  the following lemma.

\begin{lemma}\label{compardensities}
Let $\Omega\subset\R^n$ be an open set, and $u\in\widehat H^s(\Omega;\R^d)\cap L^\infty(\R^n)$ be such that  $\|u\|_{L^\infty(\R^n)}\leq M$. For every $\varepsilon>0$, there exists $\delta=\delta(n,s,M,\varepsilon)>0$ and $\alpha=\alpha(n,s,M,\varepsilon)\in(0,1/4]$ such that 
$$\boldsymbol{\Theta}_s(u^\e,{\bf x}_0,r)\leq\delta\quad\Longrightarrow\quad \boldsymbol{\theta}_s(u,x_0,\alpha r)\leq\varepsilon$$
for every ${\bf x}_0=(x_0,0)\in\Omega\times\{0\}$ and $r>0$ satisfying $\overline D_{r}(x_0)\subset\Omega$. 
\end{lemma}

\begin{proof}
Without loss of generality, we can assume that $x_0=0$. We give ourselves $\varepsilon>0$, and we shall choose the parameter $\alpha\in(0,1/4]$ later on. Using Lemma \ref{HsregtraceH1weight}, we first estimate
 \begin{align*}
 \mathcal{E}_s(u,D_{\alpha r})& \leq \frac{\gamma_{n,s}}{4}\iint_{D_{r/2}\times D_{r/2}}\frac{|u(x)-u(y)|^2}{|x-y|^{n+2s}}\,\de x\de y+ \frac{\gamma_{n,s}}{2}\iint_{D_{\alpha r}\times D^c_{r/2}}\frac{|u(x)-u(y)|^2}{|x-y|^{n+2s}}\,\de x\de y  \\
 &\leq C_1 {\bf E}_s(u^\e,B^+_{r}) + 2M^2\gamma_{n,s}\iint_{D_{\alpha r}\times D^c_{r/2}}\frac{\de x\de y}{|x-y|^{n+2s}}\,,
 \end{align*}
 where $C_1=C_1(n,s)>0$. Observe that for $(x,y)\in D_{\alpha r}\times D^c_{r/2}$, we have $|x-y|\geq |y|-\alpha r\geq \frac{1}{2}|y|$, so that 
 $$2\gamma_{n,s}\iint_{D_{\alpha r}\times D^c_{r/2}}\frac{\de x\de y}{|x-y|^{n+2s}}\leq 2^{n+2s+1}\gamma_{n,s}\iint_{D_{\alpha r}\times D^c_{r/2}}\frac{\de x\de y}{|y|^{n+2s}}=C_2\alpha^nr^{n-2s}\,,$$
 where $C_2=C_2(n,s)>0$. Consequently, 
 $$ \boldsymbol{\theta}_s(u,0,\alpha r)\leq \frac{C_1}{\alpha^{n-2s}} \boldsymbol{\Theta}_s(u^\e,0,r)+ C_2M^2\alpha^{2s}\,.  $$
 Choosing
 $$\alpha=\min\Big\{1/4,  \Big(\frac{\varepsilon}{2C_2M^2}\Big)^{1/2s}\Big\} \quad\text{and}\quad \delta:=\frac{\alpha^{n-2s}\varepsilon}{2C_1}\,, $$
 provides the desired conclusion.
\end{proof}

\begin{corollary}\label{corequivvanishdensities}
Let $\Omega\subset\R^n$ be an open set. If $u\in\widehat H^s(\Omega;\R^d)\cap L^\infty(\R^n)$, then 
$$\lim_{r\to 0} \boldsymbol{\theta}_s(u,x_0,r)=0\quad \Longleftrightarrow \quad \lim_{r\to 0} \boldsymbol{\Theta}_s(u^\e,{\bf x}_0,r)=0$$
for every ${\bf x}_0=(x_0,0)\in\Omega\times\{0\}$. 
\end{corollary}

\begin{proof}
By Lemma \ref{hatH1/2toH1}, we have 
$$  \boldsymbol{\Theta}_s(u^\e,{\bf x}_0,r)\leq C  \boldsymbol{\theta}_s(u,x_0,2r)\,,$$
for a constant $C>0$ depending only on $n$ and $s$, and implication $\Longrightarrow$ follows. The reverse implication is a straightforward application of Lemma  \ref{compardensities}. 
\end{proof}

%%%%%%%%%%%%%%%%%%%%%%%%%%%%%%%%%%%%%%%%%%%%%%%%%%%%%%%

\subsection{Energy monotonicity and mean oscillation estimates}

%%%%%%%%%%%%%%%%%%%%%%%%%%%%%%%%%%%%%%%%%%%%%%%%%%%%%%%

In the light of Proposition~\ref{monotformula}, the purpose of this section is to show a mean oscillation estimate for maps having a nondecreasing density function at every point. 
For $v\in H^1(B^+_R;\R^d,|z|^a\de{\bf x})$, a point ${\bf x}\in\partial^0B^+_R$, and $r\in(0,R-|{\bf x}|)$, we keep the notation      
$$\boldsymbol{\Theta}_s(v,{\bf x}_0,r):=\frac{1}{r^{n-2s}}\mathbf{E}_s\big(v,B^+_r({\bf x}_0)\big)\,.$$
The main estimate is the following. 

\begin{lemma}\label{cutoffbmo1}
Let $v\in H^1(B^+_{R};\R^d,|z|^a\de \mathbf{x})$ and $\zeta\in \mathscr{D}(D_{5R/8})$ be such that $0\leq \zeta\leq 1$, $\zeta\equiv 1$ in $D_{R/2}$, and $|\nabla \zeta|\leq L R^{-1}$ for some constant $L>0$.   
Assume that for every ${\bf x}\in\partial^0 B^+_{R}$, the density function 
$r\in(0,R-|{\bf x|})\mapsto  \boldsymbol{\Theta}_s(v,{\bf x},r)$ is non decreasing. Then $(\zeta v)_{|\R^n}$ belongs to ${\rm BMO}(\R^n)$ and 
% If the function 
%$r\in(0,R]\mapsto  \boldsymbol{\Theta}_s(v,{\bf x},r)$ is non decreasing for every ${\bf x}\in\partial^0 B^+_R$, then $v_{|D_R(x_0)}$ belongs to ${\rm BMO}(D_R(x_0))$ for every $x_0\in D_R$, and 
$$[\zeta v]^2_{{\rm BMO}(\R^n)}\leq C_L \big(\boldsymbol{\Theta}_s(v,0,R)+R^{2s-2-n}\|v\|^2_{L^2(B_R^+,|z|^a\de \mathbf{x})}\big) $$
for a constant $C_L=C(L,n,s)$. 
\end{lemma}

Before proving this lemma, let us recall that $u\in L^1(D_R)$ belongs to ${\rm BMO}(D_R)$ if 
$$[u]_{{\rm BMO}(D_R)}:= \sup_{D_r(y)\subset D_R}\dashint_{D_r(y)}|u-(u)_{y,r}|\,\de x<+\infty \,,$$
where $(u)_{y,r}$ denotes the average of $u$ over the ball $D_r(y)$. To prove Lemma \ref{cutoffbmo1}, we shall make use of 
the well-known John-Nirenberg inequality, see e.g. \cite[Section 6.3]{GiaMa}.

\begin{lemma}\label{JohnNir}
Let $u\in {\rm BMO}(D_R)$. For every $p\in[1,\infty)$, there exists a constant $C_p=C_p(n,p)$ such that 
$$ [u]^p_{{\rm BMO}(D_R)}\leq \sup_{D_r(y)\subset D_R}\, \dashint_{D_r(y)}|u-(u)_{y,r}|^p\,\de x\leq C_p [u]^p_{{\rm BMO}(D_R)}\,.$$
%where $(u)_{y,r}$ denotes the average of $u$ over the ball $D_r(y)$.
\end{lemma}
\vskip5pt

\begin{proof}[Proof of Lemma \ref{cutoffbmo1}]
{\it Step 1.} Rescaling variables, we may assume that $R=1$. Let us fix  an arbitrary  ball $D_r(y)\subset D_{1}$ with $y\in D_{7/8}$ and $0<r\leq 1/8$. Using  the Poincar\'e inequality in Lemma \ref{poincarebdry} and the monotonicity assumption on $\boldsymbol{\Theta}_s(v,{\bf x},\cdot)$, we estimate
$$\frac{1}{r^n}\int_{D_r(y)}\big|v-(v)_{y,r}\big|\,\de x\leq C \sqrt{\boldsymbol{\Theta}_s(v,{\bf y},r)}\leq C \sqrt{\boldsymbol{\Theta}_s(v,{\bf y},1/8)}\leq C  \sqrt{\boldsymbol{\Theta}_s(v,0,1)}\,,$$
where  ${\bf y}=(y,0)$ and $C=C(n,s)$.  In particular, $v_{|D_{7/8}}$ belongs to ${\rm BMO}(D_{7/8})$, and 
\begin{equation}\label{preestbmo}
[v]_{{\rm BMO}(D_{7/8})}\leq  C  \sqrt{\boldsymbol{\Theta}_s(v,0,1)}\,.
\end{equation}
By the John-Nirenberg inequality in Lemma \ref{JohnNir}, inequality \eqref{preestbmo}, the continuity of the trace operator (see Section \ref{secWeightSob}), and H\"older's inequality, it follows that 
\begin{multline}\label{Lnthrubmo}
\|v\|_{L^n(D_{7/8})}\leq \big\|v-(v)_{0,7/8}\big\|_{L^n(D_{7/8})}+ C \|v\|_{L^1(D_{7/8})}\\
\leq C\Big(  [v]_{{\rm BMO}(D_{7/8})}+ \|v\|_{L^1(D_{1})}\Big)\leq C\Big( \sqrt{\boldsymbol{\Theta}_s(v,0,1)}+  \|v\|_{L^2(B_1^+,|z|^a\de \mathbf{x})}\Big)\,.
\end{multline}

\noindent{\it Step 2.} Let us now consider a ball $D_r(y)\subset D_{7/8}$ with $y\in D_{3/4}$ and $0<r\leq 1/8$.  Since 
$$|\zeta v - (\zeta v)_{y,r}|\leq |\zeta v - \zeta (v)_{y,r}|+|\zeta (v)_{y,r} - (\zeta v)_{y,r}|\leq  | v - (v)_{y,r}|+ L r\, \dashint_{D_r(y)}|v|\,\de x\quad\text{on $D_{7/8}$}\,,$$
we can deduce from \eqref{preestbmo} and \eqref{Lnthrubmo} that 
\begin{multline*}
\frac{1}{r^n} \int_{D_r(y)}\big|\zeta v-(\zeta v)_{y,r}\big|\,\de x  \leq C_L\Big( \sqrt{\boldsymbol{\Theta}_s(v,0,1)} +r^{1-n}\|v\|_{L^1(D_r(y))} \Big)\\
\leq C_L\Big( \sqrt{\boldsymbol{\Theta}_s(v,0,1)} +\|v\|_{L^n(D_{7/8})} \Big)
\leq C_L\Big( \sqrt{\boldsymbol{\Theta}_s(v,0,1)} +\|v\|_{L^2(B_1^+,|z|^a\de \mathbf{x})} \Big)\,,
\end{multline*}
for a constant $C_L=C(L,n,s)$. 

Next, for a ball $D_r(y)$ with $y\not\in D_{3/4}$ and $0<r\leq 1/8$, we have 
$$ \frac{1}{r^n} \int_{D_r(y)}\big|\zeta v-(\zeta v)_{y,r}\big|\,\de x =0\,,$$
since $\zeta$ is supported in $D_{5/8}$.

Finally, for a ball $D_r(y)$ with $r>1/8$, we estimate 
$$\frac{1}{r^n} \int_{D_r(y)}\big|\zeta v-(\zeta v)_{y,r}\big|\,\de x  \leq C\int_{D_1}|\zeta v|\,\de x  \leq C\|v\|_{L^1(D_1)}\leq C\|v\|_{L^2(B_1^+,|z|^a\de \mathbf{x})}\,,$$
which completes the proof. 
\end{proof}

%\begin{remark}
%It is clear from the proof of Lemma ?? that the non decreasing assumption on the density function can be weakened. For instance, it is enough to assume that  
%the slope of $r\mapsto  \boldsymbol{\Theta}_s(v,{\bf x}_0,r)$ is uniformly bounded from below. 
%\end{remark}

\begin{corollary}\label{coroBMO}
Let $u\in \widehat H^s(D_{2R};\R^d)$ and $\zeta\in \mathscr{D}(D_{5R/8})$ be as in Lemma \ref{cutoffbmo1}. Assume that for every ${\bf x}\in\partial^0 B^+_{R}$, the density function 
$r\in(0,2R-|{\bf x|})\mapsto  \boldsymbol{\Theta}_s(u^\e,{\bf x},r)$ is non decreasing. Then $\zeta u$ belongs to ${\rm BMO}(\R^n)$ and 
% If the function 
%$r\in(0,R]\mapsto  \boldsymbol{\Theta}_s(v,{\bf x},r)$ is non decreasing for every ${\bf x}\in\partial^0 B^+_R$, then $v_{|D_R(x_0)}$ belongs to ${\rm BMO}(D_R(x_0))$ for every $x_0\in D_R$, and 
$$[\zeta u]^2_{{\rm BMO}(\R^n)}\leq C_L \big(\boldsymbol{\theta}_s(u,0,2R)+R^{-n}\|u\|^2_{L^2(D_{2R})}\big)\,,$$
for a constant $C_L=C(L,n,s)>0$. 
\end{corollary}

\begin{proof}
Apply Lemma \ref{cutoffbmo1} to $u^\e$ in $B_R^+$, and then conclude with  the help of Lemma \ref{hatH1/2toH1}. 
\end{proof}

%%%%%%%%%%%%%%%%%%%%%%%%%%%%%%%%%%%%%%%%%%%%%%%%%%%%%%%
%%%%%%%%%%%%%%%%%%%%%%%%%%%%%%%%%%%%%%%%%%%%%%%%%%%%%%%
   								       						%%%%%%%%%%%%%%%%%%%
\section{Fractional harmonic maps and weighted harmonic maps with  free boundary}\label{ELandConsLaw} %%%%%%%%%
								 						%%%%%%%%%%%%%%%%%%%
%%%%%%%%%%%%%%%%%%%%%%%%%%%%%%%%%%%%%%%%%%%%%%%%%%%%%%%
%%%%%%%%%%%%%%%%%%%%%%%%%%%%%%%%%%%%%%%%%%%%%%%%%%%%%%%

In this section, our goal is to review in details  the notion of weakly $s$-harmonic maps, the associated Euler-Lagrange equation, and more importantly 
to present its characterization in terms of fractional (nonlocal) conservation laws. We shall also prove at the end of this section that the fractional harmonic extension of an $s$-harmonic 
map satisfies a suitable (degenerate) partially free boundary condition, in the spirit of the classical harmonic map system with partially free boundary.

%%%%%%%%%%%%%%%%%%%%%%%%%%%%%%%%%%%%%%%%%%%%%%%%%%%%%%%

\subsection{Fractional harmonic maps into spheres and conservation laws}

%%%%%%%%%%%%%%%%%%%%%%%%%%%%%%%%%%%%%%%%%%%%%%%%%%%%%%%

\begin{definition}
Let $\Omega\subset\R^n$ be a bounded open set. A map $u\in \widehat H^s(\Omega;\mathbb{S}^{d-1})$ is said to be  {\sl a weakly $s$-harmonic map} in $\Omega$ (with values in $\mathbb{S}^{d-1}$) if 
$$\left[\frac{\de}{\de t} \mathcal{E}_s\Big(\frac{u+t\varphi}{|u+t\varphi|},\Omega\Big)\right]_{t=0}=0\qquad\forall \varphi\in\mathscr{D}(\Omega;\R^d) \,.$$
If $u$ is also stationary in $\Omega$ (in the sense of Definition \ref{defstatmap}), we say that $u$ is  {\sl a stationary weakly $s$-harmonic map} in  $\Omega$. 
\end{definition}

\begin{definition}
Let $\Omega\subset\R^n$ be a bounded open set. A map $u\in \widehat H^s(\Omega;\mathbb{S}^{d-1})$ is said to be {\sl a minimizing $s$-harmonic map} in  $\Omega$ (with values in $\mathbb{S}^{d-1}$) if 
$$\mathcal{E}_s(u,\Omega)\leq \mathcal{E}_s(w,\Omega) $$
for every  $w\in \widehat H^s(\Omega;\mathbb{S}^{d-1})$ such that ${\rm spt}(u-w)$ is compactly included in $\Omega$. 
\end{definition}

\begin{remark}\label{implicminstat}
A minimizing $s$-harmonic map in $\Omega$ is obviously a critical point with respect to both inner and (constrained) outer variations of the energy. In other words, if $u$ is a minimizing $s$-harmonic map in $\Omega$, then $u$ is also a stationary weakly $s$-harmonic map in $\Omega$.  
%$$ \text{minimizing $s$-harmonic }\Longrightarrow\text{ stationary weakly $s$-harmonic }\Longrightarrow \text{ weakly $s$-harmonic}$$
\end{remark}

\begin{remark}
If $u\in \widehat H^s(\Omega;\mathbb{S}^{d-1})$ is a weakly $s$-harmonic map in $\Omega$ (stationary, minimizing, respectively), then $u$ is also weakly $s$-harmonic in $\Omega^\prime$ (stationary, minimizing, respectively) for any open subset $\Omega^\prime\subset\Omega$. It can be directly checked from the definitions, or one can rely on the Euler-Lagrange equation presented below  and Remark \ref{localityLapls}. 
\end{remark}

\begin{proposition}\label{ELeqprop}
Let $\Omega\subset\R^n$ be a bounded open set. A map $u\in \widehat H^s(\Omega;\mathbb{S}^{d-1})$ is weakly $s$-harmonic in $\Omega$ if and only if 
\begin{equation}\label{ELtangtestfct}
\big\langle (-\Delta)^su,\varphi\big\rangle_\Omega=0
\end{equation}
for every $\varphi\in H^s_{00}(\Omega;\R^d)$ such that  ${\rm spt}(\varphi)\subset\Omega$ and $\varphi(x)\in {\rm Tan}(u(x),\mathbb{S}^{d-1})$ for a.e. $x\in\Omega$. 
Equivalently, 
\begin{equation}\label{ELeqorig}
(-\Delta)^su(x)= \Big(\frac{\gamma_{n,s}}{2}\int_{\R^n}\frac{|u(x)-u(y)|^2}{|x-y|^{n+2s}}\,\de y\Big)u(x)\quad\text{in $\mathscr{D}^\prime(\Omega)$}\,.
\end{equation}
\end{proposition}

\begin{proof}
Let $u\in \widehat H^s(\Omega;\mathbb{S}^{d-1})$, fix $\varphi\in\mathscr{D}(\Omega;\R^d)$, and notice that 
$$\left[\frac{\de}{\de t}\Big(\frac{u+t\varphi}{|u+t\varphi|}\Big) \right]_{t=0}=\varphi-(u\cdot\varphi)u \in H^{s}_{00}(\Omega;\R^d)\,.$$
Hence, 
$$\left[\frac{\de}{\de t} \mathcal{E}_s\Big(\frac{u+t\varphi}{|u+t\varphi|},\Omega\Big)\right]_{t=0}=\big\langle (-\Delta)^s u, \varphi\big\rangle_\Omega-\big\langle (-\Delta)^s u, (u\cdot\varphi)u\big\rangle_\Omega\,. $$
On the other hand, since $|u|^2=1$, we have 
\begin{multline*}
\big(u(x)-u(y)\big)\cdot\big((u(x)\cdot\varphi(x))u(x)-(u(y)\cdot\varphi(y))u(y) \big)\\
=\frac{1}{2}|u(x)- u(y)|^2u(x)\cdot\varphi(x) +\frac{1}{2}|u(x)- u(y)|^2u(y)\cdot\varphi(y)\,,
\end{multline*}
and it follows that 
\begin{equation}\label{computlagrangmult}
\big\langle (-\Delta)^s u, (u\cdot\varphi)u\big\rangle_\Omega=\int_\Omega\Big(\frac{\gamma_{n,s}}{2}\int_{\R^n}\frac{|u(x)-u(y)|^2}{|x-y|^{n+2s}}\,\de y\Big)u(x)\cdot\varphi(x)\,\de x\,. 
\end{equation}
Consequently, $u$ is weakly $s$-harmonic in $\Omega$ if and only if \eqref{ELeqorig} holds. 

By approximation, \eqref{ELeqorig}  also holds for any test function $\varphi\in H^s_{00}(\Omega;\R^d)\cap L^\infty(\R^n)$ compactly supported in $\Omega$. In view of  the right-hand side of  \eqref{ELeqorig}, \eqref{ELtangtestfct} clearly holds for every $\varphi\in H^s_{00}(\Omega;\R^d)\cap L^\infty(\R^n)$ compactly supported in $\Omega$ and satisfying $\varphi\cdot u=0$. By a standard truncation argument, it implies that \eqref{ELtangtestfct} holds for every $\varphi\in H^s_{00}(\Omega;\R^d)$ compactly supported in $\Omega$ and satisfying $\varphi\cdot u=0$. 

The other way around, if \eqref{ELtangtestfct} holds, then  the map $\varphi-(u\cdot\varphi)u$ with $\varphi\in \mathscr{D}(\Omega;\R^d)$ is admissible, and \eqref{ELtangtestfct}  combined with \eqref{computlagrangmult} shows that \eqref{ELeqorig} holds, i.e., $u$ is weakly $s$-harmonic in $\Omega$. 
\end{proof}

\begin{remark}\label{remsLaplorthog}
The variational equation \eqref{ELtangtestfct} corresponds to the weak formulation of the implicit equation
$$(-\Delta)^su\perp {\rm Tan}(u,\mathbb{S}^{d-1})\quad\text{in $\Omega$}\,,$$
and in equation \eqref{ELeqorig}, the Lagrange multiplier associated with the $\mathbb{S}^{d-1}$-constraint is made explicit. 
\end{remark}

\begin{remark}\label{remsmothhimplstat}
A  weakly $s$-harmonic map $u$ in $\Omega$ which is smooth in $\Omega$, is stationary in $\Omega$. 
 Indeed, if $X\in C^1(\Omega;\R^n)$ is compactly supported in $\Omega$, the smoothness of $u$ implies that 
$$\delta\mathcal{E}_s(u,\Omega)[X]= \big\langle (-\Delta)^su,X\cdot\nabla u\big\rangle_\Omega\,.$$
Since $|u|^2=1$, we have  $(X\cdot\nabla u)\cdot u=0$, and thus $\delta\mathcal{E}_s(u,\Omega)[X]=0$. 
\end{remark}

Now we rewrite the Euler-Lagrange equation \eqref{ELeqorig} in a more compact form using the fractional $s$-gradient ${\rm d}_su$ defined in Subsection \ref{sectoperandcompcomp}. More precisely, if $u=:(u^1,\ldots,u^d)$, then 
$$\frac{\gamma_{n,s}}{2}\int_{\R^n}\frac{|u(x)-u(y)|^2}{|x-y|^{n+2s}}\,\de y=\sum_{j=1}^d \frac{\gamma_{n,s}}{2}\int_{\R^n}\frac{|u^j(x)-u^j(y)|^2}{|x-y|^{n+2s}}\,\de y= \sum_{j=1}^d|{\rm d}_su^j|^2=:|{\rm d}_su|^2\,,$$
according to \eqref{modsqsvectfield} and \eqref{defsgrad}. We can thus rephrase Proposition \ref{ELeqprop} as follows: $u\in \widehat H^s(\Omega;\mathbb{S}^{d-1})$ is weakly $s$-harmonic in $\Omega$ if and only if 
\begin{equation}\label{ELeqsgrad}
(-\Delta)^su=|{\rm d}_su|^2u \quad\text{in $\mathscr{D}^\prime(\Omega)$}\,.
\end{equation}
Our aim  is to further rewrite equation \eqref{ELeqsgrad}, or more precisely its right-hand side, to reveal the fractional "div-curl structure" of Section  \ref{sectoperandcompcomp} in the spirit  of the well-known div-curl structure hidden in the classical equation for harmonic maps into spheres \cite{Hel1}. Following \cite{MazSchi}, the starting point is to notice that for  each $i,j\in\{1,\dots,d\}$,
%
%...{\bf COMPLETE} ...
%\vskip5pt
%
%
%We now observe that ... equation rewrites  
%$$(-\Delta)^su^i(x)=\sum^d_{j=1}|{\rm d}_su^j|^2u^i\quad\text{for $i=1,\ldots,d$}\,,$$
%in the distributional sense in $\Omega$.  
%For we first notice that  
\begin{align}\label{dopethrone}
\nonumber |{\rm d}_su^j|^2(x)u^i(x)&=\int_{\R^n} \frac{u^i(x){\rm d}_su^j(x,y){\rm d}_su^j(x,y)}{|x-y|^n}\,\de y\\
&=\begin{multlined}[t]
\int_{\R^n} \frac{u^i(x){\rm d}_su^j(x,y)-u^j(x){\rm d}_su^i(x,y)}{|x-y|^n}\,{\rm d}_su^j(x,y)\,\de y\\
\qquad\qquad\qquad\qquad\qquad\qquad\qquad\qquad+({\rm d}_su^i\odot{\rm d}_su^j)(x)u^j(x)\,.
\end{multlined}
\end{align}
Then, since $|u|^2=1$, we have  
\begin{align}\label{dopethrone2}
\nonumber\sum_{j=1}^d ({\rm d}_su^i\odot{\rm d}_su^j)(x)u^j(x)&=\sum_{j=1}^d\frac{\gamma_{n,s}}{2}\int_{\R^n}\frac{\big(u^j(x)-u^j(y)\big)u^j(x)}{|x-y|^{n+2s}}\big(u^i(x)-u^i(y)\big)\,\de y\\
&= \frac{\gamma_{n,s}}{4}\int_{\R^n}\frac{|u(x)-u(y)|^2}{|x-y|^{n+2s}}\big(u^i(x)-u^i(y)\big)\,\de y\,.
\end{align}
%Similarly to \eqref{dopethrone} and \eqref{dopethrone2}, we have 
%\begin{multline}\label{dopethrone3}
%\sum_{j=1}^d\int_{\Omega^c} \frac{u^i(x){\rm d}_su^j(x,y)-u^j(x){\rm d}_su^i(x,y)}{|x-y|^n}\,{\rm d}_su^j(x,y)\,\de y\\
%=\Big(\frac{\gamma_{n,s}}{2}\int_{\Omega^c}\frac{|u(x)-u(y)|^2}{|x-y|^{n+2s}}\,\de y\Big)u^i(x) 
%-\frac{\gamma_{n,s}}{4}\int_{\Omega^c}\frac{|u(x)-u(y)|^2}{|x-y|^{n+2s}}\big(u^i(x)-u^i(y)\big)\,\de y\,.
%\end{multline}
We can now introduce for $i,j\in\{1,\ldots,d\}$, 
\begin{equation}\label{defOmeij}
\boldsymbol{\Omega}^{ij}(x,y):=u^i(x){\rm d}_su^j(x,y)-u^j(x){\rm d}_su^i(x,y) \in L^2_{\rm od}(\Omega)\,,
\end{equation}
and
\begin{equation}\label{defTi}
T^i(x):=\frac{\gamma_{n,s}}{4} \int_{\R^n}\frac{|u(x)-u(y)|^2}{|x-y|^{n+2s}}\big(u^i(x)-u^i(y)\big)\,\de y \in L^1(\Omega)\,.
\end{equation}
%and 
%\begin{equation}\label{defFi}
%F^i(x):=\Big(\frac{\gamma_{n,s}}{2}\int_{\Omega^c}\frac{|u(x)-u(y)|^2}{|x-y|^{n+2s}}\,\de y\Big)u^i(x) \in L^1(\Omega)\,,
%\end{equation}
to derive from \eqref{dopethrone} and \eqref{dopethrone2} the following reformulation of equation \eqref{ELeqsgrad}. 

\begin{lemma}\label{rewritingsharmeq}
Let $\Omega\subset\R^n$ be a bounded open set.  A map $u\in \widehat H^s(\Omega;\mathbb{S}^{d-1})$ is weakly $s$-harmonic in $\Omega$ if and only if
\begin{equation}\label{rewriteEL}
(-\Delta)^su^i=\Big(\sum_{j=1}^d\boldsymbol{\Omega}^{ij}\odot{\rm d}_su^j \Big) +T^i \quad\text{in $\mathscr{D}^\prime(\Omega)$}
\end{equation}
for every $i=1,\ldots,d$, where $\boldsymbol{\Omega}^{ij}$ and $T^i$ are given by \eqref{defOmeij} and \eqref{defTi}, respectively. 
\end{lemma}

\begin{remark}
The presence of the extra term $T^i$ in \eqref{rewriteEL}, compared  the classical harmonic map equation (see \cite{Hel1}), is essentially due to the fact that the $s$-gradient ${\rm d}_su$ is not tangent to the target sphere. 
\end{remark}

The fundamental observation made in \cite[Lemma 3.1]{MazSchi} for $\Omega=\R$ and $s=1/2$ is a characterization of  the $1/2$-harmonic map equation in terms of nonlocal conservation laws satisfied by the $\boldsymbol{\Omega}^{ij}$'s (thus extending \cite{Shat} to the fractional setting). In the following proposition, we slightly generalize this result to a domain of arbitrary dimension and $s\in(0,1)$. The proof remains essentially the same, and we provide it for the reader's convenience. 

\begin{proposition}\label{propconservlaws}
Let $\Omega\subset\R^n$ be a bounded open set with Lipschitz boundary.  A map $u\in \widehat H^s(\Omega;\mathbb{S}^{d-1})$ is weakly $s$-harmonic in $\Omega$ if and only if
\begin{equation}\label{conservlaw}
{\rm div}_s\,\boldsymbol{\Omega}^{ij}=0\quad\text{in $H^{-s}(\Omega)$}
\end{equation}
for each $i,j\in\{1,\ldots,d\}$, where $\boldsymbol{\Omega}^{ij}$ is given by \eqref{defOmeij}. 
\end{proposition}

\begin{proof}
{\it Step 1.} Assume that $u$ is a  weakly $s$-harmonic map in $\Omega$, and let us compute ${\rm div}_s\,\boldsymbol{\Omega}^{ij}$. For $\varphi\in\mathscr{D}(\Omega)$, we have 
\begin{multline*}
\int_{\R^n} \boldsymbol{\Omega}^{ij}\odot {\rm d}_s\varphi\,\de x=\\
\iint_{(\R^n\times\R^n)\setminus(\Omega^c\times\Omega^c)}\big(u^i(x){\rm d}_su^j(x,y){\rm d}_s\varphi(x,y)-u^j(x){\rm d}_su^i(x,y){\rm d}_s\varphi(x,y)\big)\,\frac{\de x\de y}{|x-y|^n}\,. 
\end{multline*}
An elementary computation shows  
$$\begin{cases} 
u^i(x){\rm d}_s\varphi(x,y)={\rm d}_s(u^i\varphi)(x,y) - \varphi(y){\rm d}_su^i(x,y)\\
u^j(x){\rm d}_s\varphi(x,y)={\rm d}_s(u^j\varphi)(x,y) - \varphi(y){\rm d}_su^j(x,y)
\end{cases},$$
so that 
$$\int_{\R^n} \boldsymbol{\Omega}^{ij}\odot {\rm d}_s\varphi\,\de x=  \int_{\R^n}   {\rm d}_su^j \odot{\rm d}_s(u^i\varphi) \,\de x- \int_{\R^n}   {\rm d}_su^i \odot{\rm d}_s(u^j\varphi) \,\de x\,.$$
%\begin{align*}
%\int_{\Omega} \boldsymbol{\Omega}^{ij}\odot {\rm d}_s\varphi\,\de x=&
%\begin{multlined}[t]
%\bigg( \int_\Omega   {\rm d}_su^j \odot{\rm d}_s(u^i\varphi) \,\de x\\
%\qquad+\frac{\gamma_{n,s}}{2}\iint_{\Omega\times\Omega^c}\frac{\big(u^j(x)-u^j(y)\big)\big((u^i\varphi)(x)-(u^i\varphi)(y)\big)}{|x-y|^{n+2s}}\,\de x\de y\bigg)
%\end{multlined} \\
%&\begin{multlined}[t]
%-\bigg(\int_\Omega {\rm d}_su^i \odot  {\rm d}_s(u^j\varphi)\,\de x\\
%\qquad+\frac{\gamma_{n,s}}{2}\iint_{\Omega\times\Omega^c}\frac{\big(u^i(x)-u^i(y)\big)\big((u^j\varphi)(x)-(u^j\varphi)(y)\big)}{|x-y|^{n+2s}}\,\de x\de y\bigg)\,.
%\end{multlined}
%\end{align*}
Since $u^j\varphi$ and $u^i\varphi$ belong to $H^s_{00}(\Omega)$, we infer from Proposition \ref{fracintegbypart} and equation \eqref{ELeqsgrad} that 
\begin{align}
\label{eyehategod}\int_{\R^n} \boldsymbol{\Omega}^{ij}\odot {\rm d}_s\varphi\,\de x&=\big\langle (-\Delta)^s u^j, u^i\varphi\big\rangle_\Omega- \big\langle (-\Delta)^s u^i, u^j\varphi\big\rangle_\Omega\\
\nonumber &=\int_{\Omega}|{\rm d}_su|^2u^ju^i\varphi\,\de x-\int_{\Omega}|{\rm d}_su|^2u^iu^j\varphi\,\de x=0\,.
\end{align}
Therefore ${\rm div}_s\,\boldsymbol{\Omega}^{ij}=0$ in $\mathscr{D}^\prime(\Omega)$, and by approximation  also in $H^{-s}(\Omega)$ (see \eqref{densitysmoothH1/200}). 
\vskip5pt

\noindent{\it Step 2.} We assume that \eqref{conservlaw} holds, and we aim to prove that \eqref{ELeqsgrad} holds. We fix $\varphi\in\mathscr{D}(\Omega;\R^d)$, and we set $\psi:=\varphi-(u\cdot\varphi)u\in H^s_{00}(\Omega;\R^d)$, which satisfies $\psi\cdot u=0$ a.e. in $\R^n$. As in the proof of Proposition \ref{ELeqprop}, proving \eqref{ELeqsgrad} reduces to show that  
$$\big\langle(-\Delta)^su,\psi\big\rangle_\Omega=0\,. $$
Using $|u|^2=1$, we first observe that 
$$\big\langle(-\Delta)^su,\psi\big\rangle_\Omega=\sum_{i=1}^d\big\langle(-\Delta)^su^i,\psi^i\big\rangle_\Omega= \sum_{i,j=1}^d\big\langle(-\Delta)^su^i,(\psi^iu^j)u^j\big\rangle_\Omega\,.$$
Since $\psi^iu^j\in H^s_{00}(\Omega)$, we obtain as in \eqref{eyehategod}, 
\begin{multline*}
\big\langle(-\Delta)^su^i,(\psi^iu^j)u^j\big\rangle_\Omega=\big\langle(-\Delta)^su^j,(\psi^iu^j)u^i\big\rangle_\Omega-\int_{\R^n} \boldsymbol{\Omega}^{ij}\odot {\rm d}_s(\psi^iu^j)\,\de x\\
=\big\langle(-\Delta)^su^j,(\psi^iu^j)u^i\big\rangle_\Omega
\end{multline*}
for every $i,j\in\{1,\ldots,d\}$, thanks to \eqref{conservlaw}. Therefore, 
$$\big\langle(-\Delta)^su,\psi\big\rangle_\Omega= \sum_{i,j=1}^d\big\langle(-\Delta)^su^j,(\psi^iu^j)u^i\big\rangle_\Omega=\sum_{j=1}^d\big\langle(-\Delta)^su^j,(\psi\cdot u)u^j\big\rangle_\Omega=0\,,$$
and the proof is complete. 
\end{proof}

%%%%%%%%%%%%%%%%%%%%%%%%%%%%%%%%%%%%%%%%%%%%%%%%%%%%%%%

\subsection{Weighted harmonic maps with free boundary}

%%%%%%%%%%%%%%%%%%%%%%%%%%%%%%%%%%%%%%%%%%%%%%%%%%%%%%%

\begin{definition}
Let $G\subset\R^{n+1}_+$ be a bounded admissible open set, and $v\in H^1(G;\R^d,|z|^a\de{\bf x})$ satisfying $v({\bf x})\in\mathbb{S}^{d-1}$ for a.e. ${\bf x}\in\partial^0G$. The map $v$ is said to be {\sl a weighted weakly harmonic map in $G$ with respect to the partially free boundary condition}  $v(\partial^0G)\subset\mathbb{S}^{d-1}$  if 
\begin{equation}\label{vareqext}
\int_Gz^a\nabla v\cdot\nabla\Phi\,\de{\bf x}=0
\end{equation}
for every $\Phi\in H^1(G;\R^d,|z|^a\de{\bf x})$ such that $\Phi=0$ on $\partial^+G$ and $\Phi({\bf x})\in{\rm Tan}(v({\bf x}),\mathbb{S}^{d-1})$ for a.e. ${\bf x}\in\partial^0G$. In short, we shall say that $v$ is a weighted weakly harmonic map with free boundary in~$G$.  
\end{definition}

\begin{remark}
If $v\in H^1(G;\R^d,|z|^a\de{\bf x})$ is a weighted weakly harmonic map with free boundary in~$G$, then \eqref{vareqext} means that $v$ satisfies in the weak sense
\begin{equation}\label{EqHarmmapfreebdry} 
\begin{cases}
{\rm div}(z^a\nabla v)=0 &\text{in $G$}\,,\\[5pt]
\displaystyle z^a\frac{\partial v}{\partial \nu} \perp {\rm Tan}(v,\mathbb{S}^{d-1}) & \text{on $\partial^0G$}\,.
\end{cases}
\end{equation}
In particular, $v$ is smooth in $G$ by standard elliptic regularity. 
\end{remark}

In view of Remark \ref{remsLaplorthog}, equation \eqref{EqHarmmapfreebdry}  above, and Lemma \ref{repnormderfraclap}, it is clear that weighted weakly harmonic maps with free boundary and weakly $s$-harmonic maps are intimately related. This relation is made precise in the following proposition (see \cite[Proposition 4.6]{MS}).

\begin{proposition}\label{equivsharmfreebdry}
Let $\Omega\subset\R^n$ be a bounded open set with Lipschitz boundary. If a map $u\in \widehat H^s(\Omega;\mathbb{S}^{d-1})$ is a weakly $s$-harmonic map in $\Omega$, then its 
extension $u^\e$ given by~\eqref{poisson} is a weighted weakly harmonic map with free boundary in every bounded admissible open set $G\subset\R^{n+1}_+$ satisfying $\overline{\partial^0G}\subset\Omega$. 
\end{proposition}
 
 \begin{proof}
 Let us  assume that $u$ is a weakly $s$-harmonic map in $\Omega$, and let $G\subset\R^{n+1}_+$ be bounded admissible open set such that $\overline{\partial^0G}\subset\Omega$. Let $\Phi\in H^1(G;\R^d,|z|^a\de{\bf x})$ such that $\Phi=0$ on $\partial^+G$, and $\Phi\cdot u=0$ on $\partial^0G$. We extend $\Phi$ by $0$ to the whole half space $\R^{n+1}_+$, and the resulting map, still denoted by $\Phi$, belongs to 
$H^1(\R^{n+1}_+;\R^d,|z|^a\de{\bf x})$. In view of \eqref{normexth1/2}, $\Phi_{|\R^n}\in  H^s_{00}(\Omega;\R^d)$, and ${\rm spt}(\Phi_{|\R^n})\subset\Omega$.  Since $\Phi_{|\R^n}\cdot u=0$, we conclude from  
Lemma \ref{repnormderfraclap} and Proposition \ref{ELeqprop} that 
$$\int_Gz^a\nabla u^\e\cdot\nabla\Phi\,\de{\bf x}=\int_{\R^{n+1}_+}z^a\nabla u^\e\cdot\nabla\Phi\,\de{\bf x}=\frac{1}{\boldsymbol{\delta}_s}\big\langle(-\Delta)^s u,\Phi_{|\R^n} \big\rangle_\Omega=0\,. $$
Hence, $u^\e$ is indeed a weighted weakly harmonic map with free boundary in $G$. 
\end{proof}

 %%%%%%%%%%%%%%%%%%%%%%%%%%%%%%%%%%%%%%%%%%%%%%%%%%%%%%%
%%%%%%%%%%%%%%%%%%%%%%%%%%%%%%%%%%%%%%%%%%%%%%%%%%%%%%%
   								       						%%%%%%%%%%%%%%%%%%%
\section{Small energy H\"older regularity}\label{EpsRegthm} 					%%%%%%%%%
								 						%%%%%%%%%%%%%%%%%%%
%%%%%%%%%%%%%%%%%%%%%%%%%%%%%%%%%%%%%%%%%%%%%%%%%%%%%%%
%%%%%%%%%%%%%%%%%%%%%%%%%%%%%%%%%%%%%%%%%%%%%%%%%%%%%%%

In this section, we present the main epsilon-regularity theorem asserting that under a certain smallness assumption of the energy in a ball,  a weakly $s$-harmonic map satisfying the monotonicity formula is 
H\"older continuous in a smaller ball. H\"older regularity will be improved to Lipschitz regularity in the next section with an explicit control on the Lipschitz norm in terms of the energy.

\begin{theorem}\label{thmepsregholder}
There exist constants $\boldsymbol{\eps}_0=\boldsymbol{\eps}_0(n,s)>0$ and  $\beta_0=\beta_0(n,s)\in(0,1)$ such that the following holds. Let $u\in \widehat H^s(D_R;\mathbb{S}^{d-1})$ be a weakly $s$-harmonic map in $D_R$ such that the function $r\in(0,R-|{\bf x}|)\mapsto \boldsymbol{\Theta}_s(u^\e,{\bf x},r)$ is non decreasing for every ${\bf x}\in\partial^0B_R^+$. If  
\begin{equation}\label{condeps0}
%\frac{1}{R^{n-2s}}\mathcal{E}_s(u,D_R)
\boldsymbol{\theta}_s(u,0,R)\leq\boldsymbol{\eps}_0\,, 
\end{equation}
then $u\in C^{0,\beta_0}(D_{R/2})$ and 
\begin{equation}\label{controlholdepsnonloc}
R^{2\beta_0}[u]^2_{C^{0,\beta_0}(D_{R/2})}\leq C \boldsymbol{\theta}_s(u,0,R)\,,
\end{equation}
for a constant $C=C(n,s)$. 
\end{theorem}
 
For what follows, it is useful to translate the epsilon-regularity theorem above only in terms of the extension. This is the purpose of the following corollary.
 
\begin{corollary}\label{coroepsreghold}
There exist three constants $\boldsymbol{\eps}_1=\boldsymbol{\eps}_1(n,s)>0$, $\boldsymbol{\kappa}_1=\boldsymbol{\kappa}_1(n,s)\in(0,1)$, $\beta_1=\beta_1(n,s)\in(0,1)$  such that the following holds. Let $u\in \widehat H^s(D_{2R};\mathbb{S}^{d-1})$ be a weakly $s$-harmonic map in $D_{2R}$ such that the function $r\in(0,2R-|{\bf x}|)\mapsto \boldsymbol{\Theta}_s(u^\e,{\bf x},r)$ is non decreasing for every ${\bf x}\in\partial^0B_{2R}^+$. If  
\begin{equation}\label{condeps0ext}
%\frac{1}{R^{n-2s}}{\bf E}_s(u^\e,B_R^+)
\boldsymbol{\Theta}_s(u^\e,0,R)
\leq \boldsymbol{\eps}_1\,, 
\end{equation}
then $u^\e\in C^{0,\beta_1}(B^+_{\boldsymbol{\kappa}_1R})$ and  
$$R^{2\beta_1} [u^\e]^2_{C^{0,\beta_1}(B^+_{\boldsymbol{\kappa}_1R})}\leq C\,,$$
for a constant $C=C(n,s)$.
\end{corollary} 
 
\begin{proof}
We consider the constant $\boldsymbol{\eps}_0=\boldsymbol{\eps}_0(n,s)>0$ given by Theorem \ref{thmepsregholder}. Since $|u|\equiv 1$, we obtain from Lemma \ref{compardensities} the existence of  $\boldsymbol{\eps}_1= \boldsymbol{\eps}_1(n,s)>0$ and $\alpha=\alpha(n,s)\in(0,1/4]$ such that the condition $\boldsymbol{\Theta}_s(u^\e,0,R)
\leq \boldsymbol{\eps}_1$ implies $\boldsymbol{\theta}_s(u,0,\alpha R)\leq\boldsymbol{\eps}_0$. In turn, Theorem \ref{thmepsregholder} tells us that $u\in C^{0,\beta_0}(D_{\alpha R/2})$. Then Lemma \ref{HoldTransf} implies that $u^\e\in C^{0,\beta_1}(B^+_{\boldsymbol{\kappa}_1 R})$ with $\beta_1:=\min(\beta_0,s)$ and $\boldsymbol{\kappa}_1:=\alpha/8$. Moreover, combining \eqref{Holdtransfesti} and \eqref{controlholdepsnonloc} leads to 
\begin{multline*}
R^{2\beta_1} [u^\e]^2_{C^{0,\beta_1}(B^+_{\boldsymbol{\kappa}_1R})}\leq C\big(R^{2\beta_1} [u]^2_{C^{0,\beta_1}(D_{\alpha R/2})}+1 \big) 
 \leq C\big(R^{2\beta_0} [u]^2_{C^{0,\beta_0}(D_{\alpha R/2})}+1 \big)\\
\leq C\big( \boldsymbol{\theta}_s(u,0,\alpha R)+1\big)\leq C\,,
\end{multline*}
and the proof is complete. 
\end{proof} 
 
\begin{remark}\label{remarlepsholdregsubcritic}
In the case $n\leq 2s$, the function $r\in(0,R-|{\bf x}|)\mapsto \boldsymbol{\Theta}_s(u^\e,{\bf x},r)$ is nondecreasing for every $u\in \widehat H^s(D_{R};\R^d)$. In other words, in the case $n\leq 2s$, Theorem \ref{thmepsregholder} and Corollary \ref{coroepsreghold} apply to arbitrary weakly $s$-harmonic maps. Moreover,  in the case $n=1$ and $s\in(1/2,1)$ (i.e., $n<2s$), the conclusions of Theorem \ref{thmepsregholder} and Corollary \ref{coroepsreghold} apply even without the smallness assumptions \eqref{condeps0} or \eqref{condeps0ext}, since it follows from the classical imbedding $H^s(\R)\hookrightarrow C^{0,s-1/2}(\R)$. For our purposes, it is convenient to state it suitably. This is the object of the proposition below, whose proof is postponed to the end of Section \ref{subsectprfthmhold}. 
\end{remark} 
 
\begin{proposition}\label{propHoldsubcritic}
Assume that $n=1$ and $s\in(1/2,1)$. If $u\in \widehat H^s(D_R;\R^d)$, then $u\in C^{0,s-1/2}(D_{R/2})$ and 
%$u^\e\in C^{0,s-1/2}(B_{R/8}^+)$. In addition,  
\begin{equation}\label{holdestisubcriticcase1}
R^{2s-1}[u]^2_{C^{0,s-1/2}(D_{R/2})}\leq C\boldsymbol{\theta}_s(u,0,R)\,, 
\end{equation}
%and 
%\begin{equation}\label{holdestisubcriticcase2}
%R^{2s-1}[u^\e]^2_{C^{0,s-1/2}(B^+_{R/8})}\leq C \big(\boldsymbol{\theta}_s(u,0,R)+\|u\|^2_{L^\infty(\R^n)} \big)\,,
%\end{equation}
for a constant $C=C(s)$. 
\end{proposition}

\subsection{Proof of Theorem \ref{thmepsregholder} and Proposition \ref{propHoldsubcritic}}\label{subsectprfthmhold}
 
The key point to prove Theorem \ref{thmepsregholder} is to obtain a geometric decay of the energy in small balls. Then H\"older continuity follows classically from Campanato's criterion. The purpose of the next proposition, very much inspired from \cite[Proposition 3.1]{Evans}, is exactly to show such decay.

\begin{proposition}\label{energimprovprop}
Assume that $n\geq 2s$. There exist two constants $\boldsymbol{\varepsilon}_*=\boldsymbol{\varepsilon}_*(n,s)>0$ and $\boldsymbol{\tau}=\boldsymbol{\tau}(n,s)\in(0,1/4)$ such that the following holds. 
Let $u\in \widehat H^s(D_1;\mathbb{S}^{d-1})$ be a weakly $s$-harmonic map in $D_1$ such that the function $r\in(0,1-|{\bf x}|)\mapsto \boldsymbol{\Theta}_s(u^\e,{\bf x},r)$ is non decreasing for every ${\bf x}\in\partial^0B_1^+$. If  
$$\mathcal{E}_s(u,D_1)\leq  \boldsymbol{\varepsilon}_*\,,$$
then 
$$\frac{1}{\boldsymbol{\tau}^{n-2s}}\mathcal{E}_s(u,D_{\boldsymbol{\tau}}) \leq \frac{1}{2}\mathcal{E}_s(u,D_1)\,.$$
\end{proposition}

\begin{proof}
We fix the constant $\boldsymbol{\tau}\in(0,1/4)$ that will be specified later on. We proceed by contradiction assuming that there exists a sequence $\{u_k\}$ of stationary weakly $s$-harmonic maps in $D_1$
satisfying 
$$\varepsilon^2_k:=\mathcal{E}_s(u_k,D_1) \mathop{\longrightarrow}\limits_{k\to\infty} 0\,,$$
and
\begin{equation}\label{hypcontrepsregprop}
\frac{1}{\boldsymbol{\tau}^{n-2s}}\mathcal{E}_s(u_k,D_{\boldsymbol{\tau}}) >  \frac{1}{2}\mathcal{E}_s(u_k,D_1)\,.
\end{equation}
(Note that this later condition ensures that $\eps_k>0$.) Then we consider the (expanded) map
$$w_k:=\frac{u_k-(u_k)_{0,1}}{\eps_k}\in  \widehat H^s(D_1;\R^{d})\cap L^\infty(\R^n)\,,$$
which satisfies
$$\dashint_{D_1}w_k\,\de x=0\quad\text{and}\quad \mathcal{E}_s(w_k,D_1)=1 \,.$$
Assumption \eqref{hypcontrepsregprop} also rewrites
\begin{equation}\label{newhypcontr}
 \frac{1}{\boldsymbol{\tau}^{n-2s}}\mathcal{E}_s(w_k,D_{\boldsymbol{\tau}}) >  \frac{1}{2}\,.
 \end{equation}
% For the rest of proof, we shall now  denote by $C$ a (positive) generic constant depending only on $n$ and $s$ without mentioning this dependence. 
%\vskip3pt
By Poincar\'e's inequality in $H^s(D_1)$, we have 
$$\|w_k\|^2_{L^2(D_1)}\leq C   \mathcal{E}_s(w_k,D_1)\leq C\,.$$
Therefore $\{w_k\}$ is bounded in  $\widehat H^s(D_1;\R^{d})$, so that we can find a (not relabeled) subsequence and $w\in  \widehat H^s(D_1;\R^{d})$ such that $w_k\rightharpoonup w$ weakly in $\widehat H^s(D_1)$ and $w_k\to w$ strongly in $L^2(D_1)$ (see Remark \ref{remweakcvHhat}). In particular, $\|w\|_{L^2(D_1)}\leq C$. By lower semicontinuity of the energy $\mathcal{E}_s(\cdot,D_1)$, we also have $\mathcal{E}_s(w,D_1)\leq 1$ (see again Remark \ref{remweakcvHhat}).

Recalling that $u_k$ satisfies
$$\big\langle (-\Delta)^su_k,\varphi\big\rangle_{D_1}= \int_{D_1} |{\rm d}_su_k|^2u_k\cdot\varphi\,\de x \qquad \forall \varphi\in\mathscr{D}(D_1;\R^d)\,,$$
we obtain in terms of $w_k$, 
\begin{equation}\label{eqwk}
\big\langle (-\Delta)^sw_k,\varphi\big\rangle_{D_1}= \eps_k\int_{D_1} |{\rm d}_sw_k|^2u_k\cdot\varphi\,\de x \qquad \forall \varphi\in\mathscr{D}(D_1;\R^d)\,.
\end{equation}
Since $|u_k|\equiv 1$, it leads to 
\begin{multline*}
\Big|\big\langle (-\Delta)^sw_k,\varphi\big\rangle_{D_1}\Big| \leq \eps_k\big\| |{\rm d}_s w_k|^2 \big\|_{L^1(D_1)}\|\varphi\|_{L^\infty(D_1)}\\
\leq 2\eps_k\mathcal{E}_s(w_k,D_1)\|\varphi\|_{L^\infty(D_1)}=
2\eps_k\|\varphi\|_{L^\infty(D_1)}\mathop{\longrightarrow}\limits_{k\to\infty}0
\end{multline*}
for every $\varphi\in\mathscr{D}(D_1;\R^d)$. On the other hand, the weak convergence in $\widehat H^s(D_1)$ of $w_k$ towards $w$ implies that 
$$\big\langle (-\Delta)^sw_k,\varphi\big\rangle_{D_1}\mathop{\longrightarrow}\limits_{k\to\infty} \big\langle (-\Delta)^sw,\varphi\big\rangle_{D_1}\qquad \forall \varphi\in\mathscr{D}(D_1;\R^d)\,.$$
As a consequence, $w$ satisfies
\begin{equation}\label{eqwlim}
 (-\Delta)^sw=0\quad\text{in $H^{-s}(D_1)$}\,.
 \end{equation}
By Lemma \ref{lipestsharmfctlem} in Appendix \ref{AppSharmfct}, $w$ is (locally) smooth in $D_1$, and we have the estimate  
\begin{equation}\label{lipestiexpandmap}
\|w\|^2_{L^\infty(D_{1/2})}+\|\nabla w\|^2_{L^\infty(D_{1/2})} \leq C\big(\mathcal{E}_s(w,D_1)+\|w\|^2_{L^2(D_1)}\big)\leq C\,.
\end{equation}
In view of \eqref{lipestiexpandmap}, we have 
\begin{equation}\label{jeudsoirwhs2211}
 \iint_{D_{\boldsymbol{\tau}}\times D_{\boldsymbol{\tau}}}\frac{|w(x)-w(y)|^2}{|x-y|^{n+2s}}\,\de x\de y \leq C \iint_{D_{\boldsymbol{\tau}}\times D_{\boldsymbol{\tau}}} \frac{\de x\de y}{|x-y|^{n+2s-2}}\leq C\boldsymbol{\tau}^{n+2-2s}\,.
\end{equation}
Then, writing  
\begin{multline}\label{jeudsoirwhs2211bis}
 \iint_{D_{\boldsymbol{\tau}}\times D^c_{\boldsymbol{\tau}}}\frac{|w(x)-w(y)|^2}{|x-y|^{n+2s}}\,\de x\de y= \iint_{D_{\boldsymbol{\tau}}\times (D_{1/2}\setminus D_{\boldsymbol{\tau}})}\frac{|w(x)-w(y)|^2}{|x-y|^{n+2s}}\,\de x\de y\\
+\iint_{D_{\boldsymbol{\tau}}\times D^c_{1/2}}\frac{|w(x)-w(y)|^2}{|x-y|^{n+2s}}\,\de x\de y\,,
\end{multline}
we first estimate, using \eqref{lipestiexpandmap},
\begin{equation}\label{jeudsoirwhs2211bisbis}
\iint_{D_{\boldsymbol{\tau}}\times (D_{1/2}\setminus D_{\boldsymbol{\tau}})}\frac{|w(x)-w(y)|^2}{|x-y|^{n+2s}}\,\de x\de y
\leq C \iint_{D_{\boldsymbol{\tau}}\times (D_{1/2}\setminus D_{\boldsymbol{\tau}})}\frac{\de x\de y }{|x-y|^{n+2s-2}} \leq C\boldsymbol{\tau}^{n}\,.
\end{equation}
Next we infer from Lemma \ref{adminHchap} and \eqref{lipestiexpandmap} that 
\begin{multline}\label{jeudsoirwhs2211bisbisbis}
\iint_{D_{\boldsymbol{\tau}}\times D^c_{1/2}}\frac{|w(x)-w(y)|^2}{|x-y|^{n+2s}}\,\de x\de y\leq 2\iint_{D_{\boldsymbol{\tau}}\times D^c_{1/2}}\frac{|w(x)|^2+|w(y)|^2}{|x-y|^{n+2s}}\,\de x\de y\\
\leq C\Big(\int_{D_{\boldsymbol{\tau}}}|w(x)|^2\,\de x+\boldsymbol{\tau}^n\int_{D^c_{1/2}}\frac{|w(y)|^2}{(|y|+1)^{n+2s}}\,\de y\Big)\leq C\boldsymbol{\tau}^n\,.
\end{multline}
Gathering \eqref{jeudsoirwhs2211}, \eqref{jeudsoirwhs2211bis}, \eqref{jeudsoirwhs2211bisbis}, and \eqref{jeudsoirwhs2211bisbisbis} yields
\begin{equation}\label{energwsmall}
\frac{1}{\boldsymbol{\tau}^{n-2s}}\mathcal{E}_s(w,D_{\boldsymbol{\tau}}) \leq C\boldsymbol{\tau}^{2s}\,.
\end{equation}

By Lemma \ref{compactexpandmap} -- which is postponed at the end of the proof -- there exists a universal constant $\boldsymbol{\sigma}\in(0,1)$  such that 
\begin{equation}\label{strongcvwkw}
w_k\to w \text{ strongly in } H^s(D_{\boldsymbol{\sigma}})\,.
\end{equation}
In view of \eqref{energwsmall}, we can choose $\boldsymbol{\tau}$ (depending only on $n$ and $s$) in such a way that 
\begin{equation}\label{firstchoicetau}
0<\boldsymbol{\tau}<\boldsymbol{\sigma}/2\quad\text{and}\quad  \frac{1}{\boldsymbol{\tau}^{n-2s}}\mathcal{E}_s(w,D_{\boldsymbol{\tau}}) \leq \frac{1}{4}\,. 
\end{equation}
From  \eqref{jeudsoirwhs2211} and the strong convergence in \eqref{strongcvwkw},  we first infer that for $k$ large enough, 
\begin{equation}\label{samcanic1527}
\iint_{D_{\boldsymbol{\tau}}\times D_{\boldsymbol{\tau}}}\frac{|w_k(x)-w_k(y)|^2}{|x-y|^{n+2s}}\,\de x\de y \leq \iint_{D_{\boldsymbol{\tau}}\times D_{\boldsymbol{\tau}}}\frac{|w(x)-w(y)|^2}{|x-y|^{n+2s}}\,\de x\de y+\boldsymbol{\tau}^n\,.
\end{equation}
In the same way, for $k$ large enough, one obtains from \eqref{strongcvwkw}, 
\begin{align}
\nonumber \iint_{D_{\boldsymbol{\tau}}\times D^c_{\boldsymbol{\tau}}}\frac{|w_k(x)-w_k(y)|^2}{|x-y|^{n+2s}}\,\de x\de y=&\iint_{D_{\boldsymbol{\tau}}\times (D_{\boldsymbol{\sigma}}\setminus D_{\boldsymbol{\tau}})}\frac{|w_k(x)-w_k(y)|^2}{|x-y|^{n+2s}}\,\de x\de y\\
\nonumber &\qquad\qquad+ \iint_{D_{\boldsymbol{\tau}}\times D^c_{\boldsymbol{\sigma}}}\frac{|w_k(x)-w_k(y)|^2}{|x-y|^{n+2s}}\,\de x\de y\\[3pt]
\nonumber \leq&\;\boldsymbol{\tau}^n+\iint_{D_{\boldsymbol{\tau}}\times (D_{\boldsymbol{\sigma}}\setminus D_{\boldsymbol{\tau}})}\frac{|w(x)-w(y)|^2}{|x-y|^{n+2s}}\,\de x\de y\\
\label{samcanic1528} &\qquad\qquad+ \iint_{D_{\boldsymbol{\tau}}\times D^c_{\boldsymbol{\sigma}}}\frac{|w_k(x)-w_k(y)|^2}{|x-y|^{n+2s}}\,\de x\de y\,.
\end{align}
Then we estimate by means of Lemma \ref{adminHchap}, 
\begin{multline*}
\iint_{D_{\boldsymbol{\tau}}\times D^c_{\boldsymbol{\sigma}}}\frac{|w_k(x)-w_k(y)|^2}{|x-y|^{n+2s}}\,\de x\de y\leq 2\iint_{D_{\boldsymbol{\tau}}\times D^c_{\boldsymbol{\sigma}}}\frac{|w_k(x)|^2+|w_k(y)|^2}{|x-y|^{n+2s}}\,\de x\de y\\
\leq C\Big(\int_{D_{\boldsymbol{\tau}}}|w_k(x)|^2\,\de x+\boldsymbol{\tau}^n\int_{D^c_{\boldsymbol{\sigma}}}\frac{|w_k(y)|^2}{(|y|+1)^{n+2s}}\,\de y\Big)\leq C\Big(\int_{D_{\boldsymbol{\tau}}}|w_k(x)|^2\,\de x+\boldsymbol{\tau}^n\Big)\,.
\end{multline*}
Since $w_k\to w$ strongly in $L^2(D_1)$ and in view of \eqref{lipestiexpandmap}, we deduce that for $k$ large enough,
\begin{equation}\label{samcanic1525}
\iint_{D_{\boldsymbol{\tau}}\times D^c_{\boldsymbol{\sigma}}}\frac{|w_k(x)-w_k(y)|^2}{|x-y|^{n+2s}}\,\de x\de y\leq C\Big(\int_{D_{\boldsymbol{\tau}}}|w(x)|^2\,\de x+\boldsymbol{\tau}^n\Big)\leq C\boldsymbol{\tau}^n\,.
\end{equation} 
Combining \eqref{samcanic1527}, \eqref{samcanic1528}, and \eqref{samcanic1525}  together with \eqref{firstchoicetau},  we conclude that for $k$ large enough, 
$$\frac{1}{\boldsymbol{\tau}^{n-2s}}\mathcal{E}_s(w_k,D_{\boldsymbol{\tau}})\leq \frac{1}{\boldsymbol{\tau}^{n-2s}}\mathcal{E}_s(w,D_{\boldsymbol{\tau}})+C\boldsymbol{\tau}^{2s}\leq \frac{1}{4}+ C\boldsymbol{\tau}^{2s}\,.$$
Hence, we can choose $\boldsymbol{\tau}\in (0,1/4)$ small enough (depending only on $n$ and $s$) in such a way that $\frac{1}{\boldsymbol{\tau}^{n-2s}}\mathcal{E}_s(w_k,D_{\boldsymbol{\tau}})\leq 1/2$ whenever $k$ is large enough, contradicting \eqref{newhypcontr}. 
\end{proof}

As it is transparent from the proof above, Proposition \ref{energimprovprop} crucially rests on the strong convergence stated in \eqref{strongcvwkw} that we now prove. 

\begin{lemma}\label{compactexpandmap}
There exists a universal constant $\boldsymbol{\sigma}\in(0,1)$  such that the weakly converging subsequence $\{w_k\}$ (towards $w$) actually converges strongly in $H^s(D_{\boldsymbol{\sigma}})$. 
\end{lemma}

\begin{proof}
We choose the constant $\boldsymbol{\sigma}$ as follows: 
$$\boldsymbol{\sigma}:=\min\Big\{\frac{4}{5\Lambda},\frac{1}{32}\Big\}\,,$$
where $\Lambda>1$ is the universal constant given by Theorem \ref{divcurlthm}. 
%We divide the proof in several steps, and all along, $C$ still denotes a generic constant depending only on $n$ and $s$. 
\vskip3pt

\noindent{\it Step 1.} Subtracting  \eqref{eqwlim} from equation \eqref{eqwk} leads to
\begin{equation}\label{eqdiffwkw}
\big\langle (-\Delta)^s(w_k-w),\varphi\big\rangle_{D_1}= \eps_k\int_{D_1} |{\rm d}_sw_k|^2u_k\cdot\varphi\,\de x \qquad \forall \varphi\in\mathscr{D}(D_1;\R^d)\,.
\end{equation}
By approximation (see \eqref{densitysmoothH1/200}), this equation also holds for every $\varphi\in H^s_{00}(D_1;\R^d)\cap L^\infty(D_1)$ compactly supported in $D_1$. Let us now fix a smooth cut-off function $\zeta\in\mathscr{D}(D_{5\boldsymbol{\sigma}/4})$ 
such that $0\leq \zeta\leq 1$, $\zeta=1$ in $D_{\boldsymbol{\sigma}}$. Using the test function $\varphi_k:=\zeta(w_k-w)\in H^s_{00}(D_1;\R^d)\cap L^\infty(D_1)$ in \eqref{eqdiffwkw} yields 
\begin{equation}\label{identLk=Rk}
\big\langle (-\Delta)^s(w_k-w),\varphi_k \big\rangle_{D_1}= \eps_k\int_{D_1}|{\rm d}_sw_k|^2u_k\cdot\varphi_k\,\de x \,.
\end{equation}
Setting 
$$L_k:=\big\langle (-\Delta)^s(w_k-w),\zeta(w_k-w)\big\rangle_{D_1}\quad\text{and}\quad R_k:= \eps_k\int_{D_1} |{\rm d}_sw_k|^2u_k\cdot \varphi_k\,\de x\,, $$
we claim that 
\begin{equation}\label{claim1Lk}
 L_k\geq [w_k-w]^2_{H^s(D_{\boldsymbol{\sigma}})}+o(1)\quad\text{as $k\to\infty$}\,,
 \end{equation}
and 
\begin{equation}\label{claim2Rk}
\lim_{k\to\infty} R_k=0\,. 
\end{equation}
Identity  \eqref{identLk=Rk} rewrites $L_k=R_k$, and the two claims above will imply that $[w_k-w]^2_{H^s(D_{\boldsymbol{\sigma}})}\to 0$ as $k\to\infty$,  whence the conclusion. 
\vskip3pt

\noindent{\it Step 2.} This step is devoted to the proof of \eqref{claim1Lk}. For simplicity, let us  denote  
$$\triangle_k:=w_k-w\,.$$ 
Since $\zeta= 1$ in $D_{\boldsymbol{\sigma}}$, and $\zeta=0$ in $D^c_{2\boldsymbol{\sigma}}$, we have
\begin{equation}\label{decompLk}
L_k=  [\triangle_k]^2_{H^s(D_{\boldsymbol{\sigma}})}+\frac{\gamma_{n,s}}{2}\big(L_k^{(1)}+L_k^{(2)}+L_k^{(3)}\big)\,,
\end{equation}
with
$$ L_k^{(1)}:=\iint_{(D_1\setminus D_{\boldsymbol{\sigma}} )\times(D_1\setminus D_{\boldsymbol{\sigma}})}\frac{(\triangle_k(x)-\triangle_k(y))\cdot(\zeta(x)\triangle_k(x)-\zeta(y)\triangle_k(y))}{|x-y|^{n+2s}}\,\de x\de y\,,$$
$$ L_k^{(2)}:=2\iint_{D_{\boldsymbol{\sigma}} \times(D_1\setminus D_{\boldsymbol{\sigma}})}\frac{(\triangle_k(x)-\triangle_k(y))\cdot(\zeta(x)\triangle_k(x)-\zeta(y)\triangle_k(y))}{|x-y|^{n+2s}}\,\de x\de y\,,$$
and
$$L_k^{(3)}:=2\iint_{D_{2\boldsymbol{\sigma}} \times D^c_1}\frac{(\triangle_k(x)-\triangle_k(y))\cdot \triangle_k(x)}{|x-y|^{n+2s}}\, \zeta(x)\,\de x\de y\,,$$
Concerning $L_k^{(1)}$, we first rewrite 
\begin{align*}
L_k^{(1)}  = & \; \iint_{(D_1\setminus D_{\boldsymbol{\sigma}} )\times(D_1\setminus D_{\boldsymbol{\sigma}})}\frac{\big((\triangle_k(x)-\triangle_k(y))\cdot \triangle_k(x)\big) (\zeta(x)-\zeta(y))}{|x-y|^{n+2s}}\,\de x\de y\\
&\qquad\qquad +  \iint_{(D_1\setminus D_{\boldsymbol{\sigma}} )\times(D_1\setminus D_{\boldsymbol{\sigma}})}\frac{|\triangle_k(x)-\triangle_k(y)|^2}{|x-y|^{n+2s}}\,\zeta(y)\,\de x\de y\\[3pt]
\geq &   \; \iint_{(D_1\setminus D_{\boldsymbol{\sigma}} )\times(D_1\setminus D_{\boldsymbol{\sigma}})}\frac{\big((\triangle_k(x)-\triangle_k(y))\cdot \triangle_k(x)\big) (\zeta(x)-\zeta(y))}{|x-y|^{n+2s}}\,\de x\de y\,.
\end{align*}
Recalling that 
$$\mathcal{E}_s(\triangle_k,D_1)\leq 2\mathcal{E}_s(w_k,D_1)+2\mathcal{E}_s(w,D_1)\leq 4\,,$$
we estimate by means of H\"older's inequality, 
\begin{multline*}
\left|\iint_{(D_1\setminus D_{\boldsymbol{\sigma}} )\times(D_1\setminus D_{\boldsymbol{\sigma}})}\frac{\big((\triangle_k(x)-\triangle_k(y))\cdot \triangle_k(x)\big) (\zeta(x)-\zeta(y))}{|x-y|^{n+2s}}\,\de x\de y\right| \\
\leq \sqrt{\mathcal{E}_s(\triangle_k,D_1)}\left( \iint_{(D_1\setminus D_{\boldsymbol{\sigma}} )\times(D_1\setminus D_{\boldsymbol{\sigma}})}\frac{|\triangle_k(x)|^2|\zeta(x)-\zeta(y)|^2}{|x-y|^{n+2s}}\,\de x\de y\right)^{1/2}\\
\leq C\left( \iint_{(D_1\setminus D_{\boldsymbol{\sigma}} )\times(D_1\setminus D_{\boldsymbol{\sigma}})}\frac{|\triangle_k(x)|^2}{|x-y|^{n+2s-2}}\,\de x\de y\right)^{1/2}
\leq C\|\triangle_k\|_{L^2(D_1)}\,.
\end{multline*}
Since $\|\triangle_k\|_{L^2(D_1)}\to 0$, we conclude that 
\begin{equation}\label{lowvanL1k}
L^{(1)}_k\geq o(1)\quad\text{as $k\to\infty$}\,.
\end{equation}
Exactly in the same way, one derives 
\begin{equation}\label{lowvanL2k}
L^{(2)}_k\geq o(1)\quad\text{as $k\to\infty$}\,.
\end{equation}
For the last term $L_k^{(3)}$, we use again H\"older's inequality to derive 
\begin{equation}\label{lowvanL3k}
\big|L_k^{(3)}\big|\leq 2 \sqrt{\mathcal{E}_s(\triangle_k,D_1)} \left(\iint_{D_{2\boldsymbol{\sigma}} \times D^c_1}\frac{ |\triangle_k(x)|^2\zeta^2(x)}{|x-y|^{n+2s}}\,\de x\de y \right)^{1/2}\leq C\|\triangle_k\|_{L^2(D_1)}=o(1) 
\end{equation}
as $k\to\infty$. Gathering now \eqref{decompLk} with \eqref{lowvanL1k}, \eqref{lowvanL2k}, and \eqref{lowvanL3k} leads to  \eqref{claim1Lk}. 
\vskip3pt

\noindent{\it Step 3.} In order to prove \eqref{claim2Rk}, we need to rewrite $R_k$ in a suitable form. First, we rewrite 
$$R_k=\frac{1}{\eps_k}\int_{D_1}|{\rm d}_su_k|^2u_k\cdot \varphi_k\,\de x\,,$$
and we recall from  Lemma \ref{rewritingsharmeq} that for each $i=1,\ldots,d$, 
$$
 |{\rm d}_su_k|^2u^i_k =\Big(\sum^n_{j=1}\boldsymbol{\Omega}^{ij}_k\odot{\rm d}_su_k^j\Big)+T_k^i= \eps_k \Big(\sum^n_{j=1}\boldsymbol{\Omega}^{ij}_k\odot{\rm d}_sw_k^j\Big)+T_k^i\,,
$$
where $\boldsymbol{\Omega}^{ij}_k\in L^2_{\rm od}(D_1)$ is given by 
$$\boldsymbol{\Omega}^{ij}_k(x,y):=u_k^i(x){\rm d}_su_k^j(x,y)-u_k^j(x){\rm d}_su_k^i(x,y)\,,$$
and 
\begin{align*}
T_k^i(x):=&\,\frac{\gamma_{n,s}}{4}\int_{\R^n}\frac{|u_k(x)-u_k(y)|^2}{|x-y|^{n+2s}} \big(u_k^i(x)-u_k^i(y)\big)\,\de y\\
=&\, \frac{\gamma_{n,s}\eps_k^3}{4}\int_{\R^n}\frac{|w_k(x)-w_k(y)|^2}{|x-y|^{n+2s}} \big(w_k^i(x)-w_k^i(y)\big)\,\de y\,.
\end{align*}
%and 
%%\begin{align*}
%$$F^i_k(x):=\Big(\frac{\gamma_{n,s}}{2}\int_{D_1^c}\frac{|u_k(x)-u_k(y)|^2}{|x-y|^{n+2s}}\,\de y\Big)u_k^i(x)
%=\eps_k^2\Big(\frac{\gamma_{n,s}}{2}\int_{D_1^c}\frac{|w_k(x)-w_k(y)|^2}{|x-y|^{n+2s}}\,\de y\Big)u_k^i(x)\,.$$
%\end{align*}
Hence, 
%\begin{multline*}
$$R_k=\Big(\sum_{i,j=1}^n\int_{D_1} \big( \Omega^{ij}_k\odot {\rm d}_sw_k^j\big)\varphi_k^i\,\de x\Big) + \eps_k^2\int_{D_1}\widetilde T_k\cdot \varphi_k\,\de x
%-\eps_k\int_{D_1}\widetilde F_k\cdot\varphi_k\,\de x\\
=:R_k^{(1)}+R_k^{(2)}\,,$$
%\end{multline*}
where we have set 
$$\widetilde T_k(x):= \frac{\gamma_{n,s}}{4}\int_{\R^n}\frac{|w_k(x)-w_k(y)|^2}{|x-y|^{n+2s}} \big(w_k(x)-w_k(y)\big)\,\de y\,.$$
%and 
%$$\widetilde F_k(x):=\Big(\frac{\gamma_{n,s}}{2}\int_{D_1^c}\frac{|w_k(x)-w_k(y)|^2}{|x-y|^{n+2s}}\,\de y\Big)u_k^i(x)\,.  $$
\vskip3pt

\noindent{\it Step 4.} We shall now prove that 
\begin{equation}\label{vanishofR1kjul}
\lim_{k\to\infty}R_k^{(1)}=0\,.
\end{equation} 
%First, notice that by stationarity ... {\bf complete}, $u_k^\e$ satisfies the monotonicity formula in {\bf Lemma ?? complete}. 
First, notice that formula \eqref{poisson} shows that $u_k^\e=\eps_k w_k^\e+ (u_k)_{0,1}$, which implies that 
$$\boldsymbol{\Theta}_s(u_k^\e,{\bf x},r)=\eps_k^2\boldsymbol{\Theta}_s(w_k^\e,{\bf x},r)\quad\text{for every ${\bf x}\in\partial^0B_1^+$ and $r\in(0,1-|{\bf x}|)$}\,.$$
As a consequence, our assumption on $\boldsymbol{\Theta}_s(u_k^\e,{\bf x},r)$ tells us that  $r\in(0,1-|{\bf x}|)\mapsto \boldsymbol{\Theta}_s(w_k^\e,{\bf x},r)$ is non decreasing for every ${\bf x}\in\partial^0B_1^+$. 

Applying Corollary \ref{coroBMO} (with $R=2\boldsymbol{\sigma}$), we deduce that 
$$[\zeta w_k]_{{\rm BMO}(\R^n)}\leq C\Big(\mathcal{E}_s(w_k,4\boldsymbol{\sigma})+ \|w_k\|^2_{L^2(D_{4\boldsymbol{\sigma}})}\Big)^{1/2}\leq C\,,  $$
for some constant $C$ depending only on $n$, $s$, and $\zeta$. Since $w_k\to w$ strongly in $L^2(D_1)$ and $\zeta$ is supported in $D_{5\boldsymbol{\sigma}/4}$, we have $\zeta w_k\to \zeta w$ strongly in $L^1(\R^n)$ (in other words, $\|\varphi_k\|_{L^1(\R^n)}\to 0$). By lower semi-continuity of the BMO-seminorm with respect to the $L^1$-convergence, we deduce that $\zeta w\in {\rm BMO}(\R^n)$, and then  (remember that $\varphi_k:=\zeta(w_k-w)$)
$$[\varphi_k]_{{\rm BMO}(\R^n)}\leq C \,.$$
Next, we recall from Proposition \ref{propconservlaws} that $u_k$ being weakly $s$-harmonic in $D_1$ yields 
$${\rm div}_s\,\boldsymbol{\Omega}^{ij}_k=0\quad\text{in $H^{-s}(D_1)$}\,, $$
for each $i,j\in \{1,\ldots, d\}$. Applying Theorem \ref{divcurlthm} (with $x_0=0$ and $r=5\boldsymbol{\sigma}/4$), we infer that 
\begin{align*}
\left|\int_{D_1} \big(\boldsymbol{\Omega}^{ij}_k\odot {\rm d}_sw_k^j\big)\varphi_k^i\,\de x\right| &\leq C\|\boldsymbol{\Omega}_k^{ij}\|_{L^2_{\rm od}(D_1)}\sqrt{\mathcal{E}_s(w_k^j,D_1)} \Big([\varphi^i_k]_{{\rm BMO}(\R^n)}+\|\varphi^i_k\|_{L^1(\R^n)}\Big)\\
&\leq C\|\boldsymbol{\Omega}_k^{ij}\|_{L^2_{\rm od}(D_1)}\,.
\end{align*}
Since $|u_k|\equiv 1$, we have the pointwise estimate $|\boldsymbol{\Omega}_k^{ij}(x,y)|\leq |{\rm d}_su_k^j(x,y)| + |{\rm d}_su_k^i(x,y)|$ which leads to  
$\|\boldsymbol{\Omega}_k^{ij}\|^2_{L^2_{\rm od}(D_1)}\leq C\mathcal{E}_s(u_k,D_1)=O(\eps^2_k)$ for each $i,j\in \{1,\ldots, d\}$. Consequently,
$$R^{(1)}_k=O(\eps_k)\,, $$
and \eqref{vanishofR1kjul} is proved. 
\vskip5pt

\noindent{\it Step 5.} We  complete the proof of \eqref{claim2Rk} showing now that 
\begin{equation}\label{vanishingofR2k}
\lim_{k\to\infty}R_k^{(2)}=0\,.
\end{equation} 
Using the fact that $\varphi_k$ is supported in $D_{5\boldsymbol{\sigma}/4}\subset D_{1/20}\subset D_{1/16}$, we first write 
\begin{equation}\label{decompofR2k}
R_k^{(2)}=\eps_k^2\int_{D_1}\widetilde T_k\cdot \varphi_k\,\de x =\frac{\gamma_{n,s}}{4} \eps_k^2\big(I_k+II_k\big)\,,
\end{equation}
with
\begin{align*}
I_k&:=\iint_{D_{1/16}\times D_{1/16}}\frac{|w_k(x)-w_k(y)|^2}{|x-y|^{n+2s}} \big(w_k(x)-w_k(y)\big)\cdot\varphi_k(x)\,\de x\de y \\
&= \frac{1}{2}\iint_{D_{1/16}\times D_{1/16}}\frac{|w_k(x)-w_k(y)|^2}{|x-y|^{n+2s}} \big(w_k(x)-w_k(y)\big)\cdot\big(\varphi_k(x)-\varphi_k(y)\big)\,\de x\de y\,,
\end{align*}
and 
\begin{equation}\label{defofIIkinR2k}
II_k:=\iint_{D_{1/20}\times D^c_{1/16}} \frac{|w_k(x)-w_k(y)|^2}{|x-y|^{n+2s}} \big(w_k(x)-w_k(y)\big)\cdot\varphi_k(x)\,\de x\de y\,.
\end{equation}
We shall estimate separately the two terms $I_k$ and $II_k$. 

Concerning $I_k$, we apply H\"older's inequality to reach 
\begin{align}
\nonumber |I_k| & \leq \frac{1}{2} \iint_{D_{1/16}\times D_{1/16}}\frac{|w_k(x)-w_k(y)|^3|\varphi_k(x)-\varphi_k(y)|}{|x-y|^{n+2s}}\,\de x\de y\\[3pt]
\label{proutauchocolat}&\leq C [w_k]^3_{W^{s/3,6}(D_{1/16})}[\varphi_k]_{H^s(D_{1/16})}\,,
\end{align}
where $[\cdot]_{W^{s/3,6}(D_{1/16})}$ denotes the  $W^{s/3,6}(D_{1/16})$-seminorm (i.e., of the Sobolev-Slobodeckij space, see \eqref{defWspseminorm}). Recalling our 
notation $\triangle_k:=w_k-w$ and the fact that $0\leq \zeta\leq 1$, we have
\begin{align}
\nonumber [\varphi_k]^2_{H^s(D_{1/16})} &\leq C\left([\triangle_k]^2_{H^s(D_{1/16})}+ \iint_{D_{1/16}\times D_{1/16}}\frac{|\zeta(x)-\zeta(y)|^2|\triangle_k(x)|^2}{|x-y|^{n+2s}}\,\de x\de y\right)\\
\nonumber &\leq C\left([\triangle_k]^2_{H^s(D_{1/16})}+ \iint_{D_{1/16}\times D_{1/16}}\frac{|\triangle_k(x)|^2}{|x-y|^{n+2s-2}}\,\de x\de y\right)\\
\label{nocluphikhs} &\leq C\Big( \mathcal{E}_s(\triangle_k,D_{1/16})+\|\triangle_k\|^2_{L^2(D_{1/16})}\Big)\leq C\,.
\end{align}
To estimate $ [w_k]_{W^{s/3,6}(D_{1/16})}$, we proceed as follows. First, we fix a further cut-off function $\eta\in \mathscr{D}(D_{1/8})$ satisfying 
$0\leq \eta\leq 1$, $\eta\equiv 1$ in $D_{1/16}$, and $|\nabla \eta|\leq C$. Then we apply Corollary  \ref{coroinjQspaces} (in Appendix \ref{appQspaces}) to $\eta w_k$ to derive 
\begin{equation}\label{mardcpamorreystuff}
[w_k]^2_{W^{s/3,6}(D_{1/16})}= [\eta w_k]^2_{W^{s/3,6}(D_{1/16})} \leq C\Big(\sup_{D_r(\bar x)\subset\R^n}\frac{1}{r^{n-2s}}\,[\eta w_k]^2_{H^s(D_r(\bar x))}\Big)\,,
\end{equation}
and it remains to estimate the right hand side of \eqref{mardcpamorreystuff}. To this purpose, we need to distinguish different types of balls: 
\vskip3pt

{\sl Case 1:} $\bar x\in D_{3/16}$ and $0< r \leq 1/16$. Arguing as in \eqref{nocluphikhs}, we obtain 
\begin{align*}
 [\eta w_k]^2_{H^s(D_{r}(\bar x))} & \leq C\left([w_k]^2_{H^s(D_{r}(\bar x))} + \iint_{D_{r}(\bar x)\times D_{r}(\bar x)}\frac{|w_k(x)|^2}{|x-y|^{n+2s-2}}\,\de x\de y \right)\\
%  & \leq C\left([w_k]^2_{H^s(D_{r}(z))} + \int_{D_{r}(z)}\Big( \int_{D_{2r}(x)}\frac{\de y}{|x-y|^{n+2s-2}}\Big) |w_k(x)|^2   \,\de x \right)\\
  & \leq C\Big([w_k]^2_{H^s(D_{r}(\bar x))} + r^{2-2s}\|w_k\|^2_{L^2(D_r(\bar x))} \Big)\,.
 \end{align*}
Applying H\"older's inequality in the case $n\geq 3$, we obtain 
\begin{equation}\label{mercrbefiep1559}
 [\eta w_k]^2_{H^s(D_{r}(\bar x))} \leq 
\begin{cases}
C\Big([w_k]^2_{H^s(D_{r}(\bar x))} + r^{n-2s}\|w_k\|^{2}_{L^n(D_r(\bar x))} \Big) & \text{if $n\geq 3$}\\[5pt]
C\Big([w_k]^2_{H^s(D_{r}(\bar x))} + r^{2-2s}\|w_k\|^2_{L^2(D_r(\bar x))} \Big) & \text{if $n\leq 2$}\,.
\end{cases}
\end{equation}
Let us now recall that $r\mapsto\boldsymbol{\Theta}_s(w_k^\e, {\bf x},r)$ is non decreasing for every ${\bf x}\in B^+_{1}$  (see Step 4). By the proof of Lemma \ref{cutoffbmo1}, Step 1 (applied to $w_k^\e$),  we have 
\begin{equation}\label{bmowkjuly}
[w_k]_{{\rm BMO}(D_{7/16})}\leq C\sqrt{\mathbf{E}_s(w^\e_k,B^+_{1/2})}\leq C\sqrt{\mathcal{E}_s(w_k,D_1)}\leq C \,,
\end{equation}
where we have used Lemma \ref{hatH1/2toH1} in the last inequality. In case $n\geq 3$, we apply the John-Nirenberg inequality  in Lemma \ref{JohnNir} and use the fact that $D_r(\bar x)\subset D_{7/16}$,  to derive 
\begin{multline}\label{argbdlebwklundjul}
\|w_k\|_{L^n(D_r(\bar x))}\leq \|w_k\|_{L^n(D_{7/16})} \leq \big\|w_k-(w_k)_{0,7/16}\big\|_{L^n(D_{7/16})}+C\|w_k\|_{L^1(D_{7/16})}\\
\leq C \big([w_k]_{{\rm BMO}(D_{7/16})} + \|w_k\|_{L^2(D_{7/16})}\big)\leq C\,.
\end{multline}
 Back to \eqref{mercrbefiep1559} and in view of Lemma \ref{HsregtraceH1weight}, we have thus proved that
  \begin{multline*}
  [\eta w_k]^2_{H^s(D_{r}(\bar x))} \leq C\big( [w_k]^2_{H^s(D_{r}(\bar x))} + r^{n-2s}\big)\\
  \leq  C\Big( \mathbf{E}_s\big(w^\e_k,B^+_{2r}(\bar {\bf x})\big) + r^{n-2s}\Big)\leq Cr^{n-2s}\big(\boldsymbol{\Theta}_s(w_k^\e,\bar {\bf x},2r)+1\big)\,,
  \end{multline*}
with $\bar {\bf x}:=(\bar x,0)$.  Then the monotonicity of $r\mapsto \boldsymbol{\Theta}_s(w_k^\e,\bar {\bf x},2r)$ together with Lemma \ref{hatH1/2toH1} yields 
 \begin{multline*}
 \frac{1}{r^{n-2s}} [\eta w_k]^2_{H^s(D_{r}(\bar x))} \leq C \big(\boldsymbol{\Theta}_s(w_k^\e,\bar {\bf x},1/8)+1\big)\\
 \leq C \big({\bf E}_s(w_k^\e,B^+_{1/2})+1\big)\leq C \big(\mathcal{E}_s(w_k,D_1)+1\big)\leq C\,.
 \end{multline*}
\vskip3pt

{\sl Case 2:} $\bar x\not\in D_{3/16}$ and $0< r \leq 1/16$. This case is trivial since $\eta w_k \equiv 0$ in $D_r(\bar x)$. 
\vskip3pt

{\sl Case 3:} $\bar x\in \R^n$ and $r>1/16$. Since $\eta w_k$ is supported in $D_{1/8}$ and $0\leq \eta\leq 1$, we have (recall that $n-2s\geq 0$)
\begin{align*}
\frac{1}{r^{n-2s}}  [\eta w_k]^2_{H^s(D_{r}(\bar x))} &\leq 16^{2s-n}[\eta w_k]^2_{H^s(\R^n)}\\
&\leq C\Big( [\eta w_k]^2_{H^s(D_{1/4})}+\iint_{D_{1/8}\times D^c_{1/4}}\frac{|\eta(x)w_k(x)|^2}{|x-y|^{n+2s}}\,\de x\de y\Big)\\
&\leq C\big( [\eta w_k]^2_{H^s(D_{1/4})}+\|w_k\|^2_{L^2(D_{1/8})}\big)\,.
\end{align*}
Arguing as in \eqref{nocluphikhs}, we obtain 
$$  [\eta w_k]^2_{H^s(D_{1/4})}\leq C\big(\mathcal{E}_s(w_k,D_{1/4})+\|w_k\|^2_{L^2(D_{1/4})}\big)\,,$$
and thus
\begin{equation}\label{zutbiblioberlin}
\frac{1}{r^{n-2s}}  [\eta w_k]^2_{H^s(D_{r}(\bar x))} \leq C\big(\mathcal{E}_s(w_k,D_{1})+\|w_k\|^2_{L^2(D_{1})}\big)\leq C\,.
\end{equation}
\vskip5pt

Gathering Cases 1, 2, and 3 above, we have proved that  the right hand side of \eqref{mardcpamorreystuff} remains bounded independently of $k$. 
We can now conclude from \eqref{mardcpamorreystuff} that $[w_k]_{W^{s/3,6}(D_{1/16})}\leq C$. In view of \eqref{proutauchocolat} and \eqref{nocluphikhs}, we have thus obtained that 
\begin{equation}\label{firstpieceR2k}
|I_k|\leq C\,,
\end{equation}
and it only remains to estimate the term $II_k$ (defined in \eqref{defofIIkinR2k}). 

First, we trivially have 
\begin{align}
\nonumber |II_k|&\leq \iint_{D_{1/20}\times D^c_{1/16}} \frac{|w_k(x)-w_k(y)|^3}{|x-y|^{n+2s}}|\triangle_k(x)|\,\de x\de y \\
\nonumber &\leq 4\iint_{D_{1/20}\times D^c_{1/16}} \frac{|w_k(x)|^3}{|x-y|^{n+2s}}|\triangle_k(x)|\,\de x\de y \\
\label{almfinlundjul1} &\qquad\qquad\qquad\qquad\qquad\qquad\qquad +4\iint_{D_{1/20}\times D^c_{1/16}} \frac{|w_k(y)|^3}{|x-y|^{n+2s}}|\triangle_k(x)|\,\de x\de y\,.
\end{align}
On the other hand,
\begin{multline*}
\iint_{D_{1/20}\times D^c_{1/16}} \frac{|w_k(x)|^3}{|x-y|^{n+2s}}|\triangle_k(x)|\,\de x\de y\leq C\int_{D_{1/20}} |w_k(x)|^3|\triangle_k(x)|\,\de x\\
\leq C\|w_k\|^3_{L^6(D_{1/20})}\|\triangle_k\|_{L^2(D_{1})}\,.
\end{multline*}
Recalling from \eqref{bmowkjuly} that $\{w_k\}$ is bounded in ${\rm BMO}(D_{7/16})$, we can argue as in \eqref{argbdlebwklundjul} to infer that $\{w_k\}$ is bounded in $L^6(D_{1/20})$. Hence, 
\begin{equation}\label{almfinlundjul2}
\iint_{D_{1/20}\times D^c_{1/16}} \frac{|w_k(x)|^3}{|x-y|^{n+2s}}|\triangle_k(x)|\,\de x\de y\leq C \|\triangle_k\|_{L^2(D_{1})}\,.
\end{equation}
Since $|u_k|\equiv1$, we have $|w_k|\leq 2/\eps_k$, and consequently
\begin{align}
\nonumber \iint_{D_{1/20}\times D^c_{1/16}} \frac{|w_k(y)|^3}{|x-y|^{n+2s}}|\triangle_k(x)|\,\de x\de y& \leq  \frac{2}{\eps_k}\int_{D_{1/20}}\left(\int_{D^c_{1/16}}\frac{|w_k(y)|^2}{|x-y|^{n+2s}}\,\de y\right)|\triangle_k(x)|\,\de x\\
\nonumber &\leq   \frac{C}{\eps_k}\int_{D_{1/20}}\left(\int_{\R^n}\frac{|w_k(y)|^2}{(|y|+1)^{n+2s}}\,\de y\right)|\triangle_k(x)|\,\de x\\
\label{almfinlundjul3} & \leq \frac{C}{\eps_k}\Big(\mathcal{E}_s(w_k,D_1)+\|w_k\|^2_{L^2(D_1)}\Big)\|\triangle_k\|_{L^2(D_{1})}\,,
\end{align}
where we have used Lemma \ref{adminHchap} in the last inequality.  Combining \eqref{almfinlundjul1}, \eqref{almfinlundjul2}, and \eqref{almfinlundjul3}, we obtain the estimate 
\begin{equation}\label{secondpieceR2k}
|II_k|\leq C\eps_k^{-1}\|\triangle_k\|_{L^2(D_{1})}= o(\eps_k^{-1})\,. 
\end{equation}
In view of \eqref{decompofR2k}, \eqref{firstpieceR2k}, and \eqref{secondpieceR2k}, we have thus proved that 
$$R_k^{(2)}=o(\eps_k)\,, $$
and thus \eqref{vanishingofR2k} holds, which completes the whole proof. 
%\vskip5pt
%
%\noindent{\it Step 6.} We  finally complete the proof \eqref{claim2Rk} showing that 
%$$\lim_{k\to\infty} R^{(3)}_k=0\,. $$
%Since $|u_k|=1$, we have for $x\in D_{1/20}$, 
%\begin{multline*}
%|\widetilde F_k(x)|\leq C\Big(|w_k(x)|^2+\int_{D_1^c}\frac{|w_k(y)|^2}{|x-y|^{n+2s}}\,\de y\Big)\\
%\leq C\big(|w_k(x)|^2+\mathcal{E}_s(w_k,D_1)+\|w_k\|^2_{L^2(D_1)}\big)\leq C\big(|w_k(x)|^2+1\big)\,,
%\end{multline*}
% thanks again to Lemma \ref{adminHchap}. Since ${\rm spt}(\varphi_k)\subset D_{1/20}$, we deduce from H\"older's inequality that 
% $$\Big|\int_{D_1}\widetilde F_k\cdot\varphi_k\,\de x\Big|\leq C\int_{D_{1/20}}\big(|w_k|^2+1\big)|\triangle_k|\,\de x\leq C\big(\|w_k\|^2_{L^4(D_{1/20})}+1\big)\|\triangle_k\|_{L^2(D_1)} \,.$$
%We have already shown (in the previous step) that $\{w_k\}$ is bounded in $L^6(D_{1/20})$, and thus it is also bounded in $L^4(D_{1/20})$. Consequently, 
%$$  R^{(3)}_k=\eps_k\int_{D_1}\widetilde F_k\cdot\varphi_k\,\de x=o(\eps_k)\,,$$
%which completes the proof of \eqref{claim2Rk}, and thus the whole proof.
\end{proof}

\begin{proof}[Proof of Theorem \ref{thmepsregholder}] 
Rescaling variables, we can assume that $R=2$. We need to distinguish the two cases $n\geq 2s$, and $n=1$ with $s\in(1/2,1)$. 
\vskip3pt

\noindent{\it Case 1: $n\geq 2s$.} We choose $\boldsymbol{\eps}_0:=2^{2s-n}\boldsymbol{\eps}_*$ where  $\boldsymbol{\eps}_*=\boldsymbol{\eps}_*(n,s)>0$ is the constant provided by Proposition ~\ref{energimprovprop}. We fix an arbitrary point $x_0\in D_1$, and we observe that condition \eqref{condeps0} implies 
$$\mathcal{E}_s\big(u,D_1(x_0)\big)\leq  \mathcal{E}_s(u,D_2)=2^{n-2s}\boldsymbol{\theta}_s(u,0,2)\leq \boldsymbol{\eps}_*\,.$$
Setting ${\bf e}:= \mathcal{E}_s(u,D_2)$, 
%
%\vskip30pt
%Then condition \eqref{condeps0} together with 
Proposition  \ref{energimprovprop} then leads to 
\begin{equation}\label{merjul171648}
\frac{1}{\boldsymbol{\tau}^{n-2s}}\mathcal{E}_s\big(u,D_{\boldsymbol{\tau}}(x_0)\big)\leq \frac{1}{2} \mathcal{E}_s\big(u,D_1(x_0)\big)\leq \frac{1}{2}{\bf e}\,,
\end{equation}
where $\boldsymbol{\tau}=\boldsymbol{\tau}(n,s)\in(0,1/4)$. Considering the rescaled map $u_{\boldsymbol{\tau}}(x):=u(\boldsymbol{\tau}x+x_0)$, one realizes from \eqref{merjul171648} that $u_{\boldsymbol{\tau}}$ satisfies $\mathcal{E}_s(u_{\boldsymbol{\tau}},D_1)\leq\frac{1}{2}\boldsymbol{\eps}_* $, and thus  Proposition  \ref{energimprovprop} applies. Unscaling variables, it yields 
$$\frac{1}{(\boldsymbol{\tau}^{n-2s})^2}\mathcal{E}_s(u,D_{\boldsymbol{\tau}^2}(x_0)\big)=\frac{1}{\boldsymbol{\tau}^{n-2s}}\mathcal{E}_s(u_{\boldsymbol{\tau}},D_{\boldsymbol{\tau}})\leq \frac{1}{2} \mathcal{E}_s(u_{\boldsymbol{\tau}},D_1)=\frac{1}{2\boldsymbol{\tau}^{n-2s}}\mathcal{E}_s\big(u,D_{\boldsymbol{\tau}}(x_0)\big) \leq \frac{1}{4}{\bf e}\,.$$
Arguing by induction, we infer that  
\begin{equation}\label{merjul171648bis}
\mathcal{E}_s\big(u,D_{\boldsymbol{\tau}^k}(x_0)\big)\leq\frac{\boldsymbol{\tau}^{k(n-2s)}}{2^k}{\bf e} \quad\text{for each $k=0,1,2,3,\ldots$}\,.
\end{equation}
Let us now fix an arbitrary  $r\in(0,1)$, and consider the integer $k$ such that $\boldsymbol{\tau}^{k+1}<r\leq\boldsymbol{\tau}^k$. From \eqref{merjul171648bis}, we deduce that 
$$\frac{1}{r^{n-2s}}\mathcal{E}_s\big(u,D_r(x_0)\big)\leq \frac{1}{r^{n-2s}}\mathcal{E}_s\big(u,D_{\boldsymbol{\tau}^k}(x_0)\big)\leq \frac{\boldsymbol{\tau}^{2s-n}}{2^k}{\bf e}\leq 2\boldsymbol{\tau}^{2s-n}{\bf e}\, r^{2\beta_0} \,,$$
with $2\beta_0:=\log(2)/\log(1/\boldsymbol{\tau})$. By the Poincar\'e inequality in $H^s(D_r(x_0))$, it yields 
$$\frac{1}{r^n}\int_{D_r(x_0)}\big|u-(u)_{x_0,r}\big|^2\,\de x\leq \frac{C}{r^{n-2s}}[u]^2_{H^s(D_r(x_0))}\leq  \frac{C}{r^{n-2s}}\mathcal{E}_s\big(u,D_r(x_0)\big)\leq C{\bf e}\,r^{2\beta_0}\,.$$
In view of the arbitrariness of $r$ and $x_0$, we can apply Campanato's criterion (see e.g. \cite[Theorem I.6.1]{Maggi}), and it yields $u\in C^{0,\beta_0}(D_1)$ with 
$$|u(x)-u(y)|\leq C\sqrt{{\bf e}}\,|x-y|^{\beta_0} \qquad\forall x,y\in D_1\,,$$
which completes the proof. 
\vskip3pt

\noindent{\it Case 2: $n=1$ and $s\in (1/2,1)$.} In this case, we simply choose $\boldsymbol{\eps}_0:=1$, and we invoke Proposition \ref{propHoldsubcritic} whose proof is given below. 
%one can choose $\boldsymbol{\eps}_0=1$, and the conclusion follows from the embedding $H^s(D_2)\hookrightarrow C^{0,s-1/2}(D_1)$ {\bf (FIND REF)}. 
\end{proof}

\begin{proof}[Proof of Proposition \ref{propHoldsubcritic}]
Rescaling variables, we can assume that $R=1$. Without loss of generality, we can also assume that $u$ has a vanishing average over $D_1$. We consider a given cut-off function $\zeta\in\mathscr{D}(D_{3/4})$ such that $0\leq \zeta\leq 1$ and $\zeta=1$ in $D_{1/2}$. Arguing as \eqref{zutbiblioberlin}, we obtain that $\zeta u\in H^{s}(\R;\R^d)$ with 
\begin{equation}\label{prfholdembed1}
[\zeta u]^2_{H^s(\R)}\leq C \big(\mathcal{E}_s(u,D_1)+\|u\|^2_{L^2(D_1)}\big) \,.
%\leq C \mathcal{E}_s(u,D_1) \,, 
\end{equation}
On the other hand, by the continuous embedding $H^s(\R^n)\hookrightarrow C^{0,s-1/2}(\R^n)$ (see e.g. \cite[Theorem~1.4.4.1]{G}), we have  
\begin{equation}\label{prfholdembed2}
 [\zeta u]^2_{C^{0,s-1/2}(\R)}\leq C\big([\zeta u]^2_{H^s(\R)} + \|\zeta u\|^2_{L^2(\R)}\big)
 \leq C\big([\zeta u]^2_{H^s(\R)} + \|u\|^2_{L^2(D_1)}\big)\,.
\end{equation}
Combining \eqref{prfholdembed2} with  \eqref{prfholdembed1} and applying Poincar\'e's inequality in $H^s(D_1)$, we derive that 
$$[u]^2_{C^{0,s-1/2}(D_{1/2})}\leq  [\zeta u]^2_{C^{0,s-1/2}(\R)} \leq  C \big(\mathcal{E}_s(u,D_1)+\|u\|^2_{L^2(D_1)}\big)\leq C \mathcal{E}_s(u,D_1)\,,$$
which completes the proof of  \eqref{holdestisubcriticcase1}. 
%Then, we infer \eqref{holdestisubcriticcase2} directly from Lemma \ref{HoldTransf}.  
\end{proof}

%%%%%%%%%%%%%%%%%%%%%%%%%%%%%%%%%%%%%%%%%%%%%%%%%%%%%%%
%%%%%%%%%%%%%%%%%%%%%%%%%%%%%%%%%%%%%%%%%%%%%%%%%%%%%%%
   								       						%%%%%%%%%%%%%%%%%%%
\section{Small energy Lipschitz regularity}\label{Highordreg} 					%%%%%%%%%
								 						%%%%%%%%%%%%%%%%%%%
%%%%%%%%%%%%%%%%%%%%%%%%%%%%%%%%%%%%%%%%%%%%%%%%%%%%%%%
%%%%%%%%%%%%%%%%%%%%%%%%%%%%%%%%%%%%%%%%%%%%%%%%%%%%%%%

In this section, our goal is to improve the conclusion of Theorem \ref{thmepsregholder} to Lipschitz continuity, as stated in the following theorem. Higher order regularity will be the object of the next section.

\begin{theorem}\label{thmepsregLip}
Let  $\boldsymbol{\eps}_1=\boldsymbol{\eps}_1(n,s)>0$  be  the constant given by Corollary \ref{coroepsreghold}. There exists a constant $\boldsymbol{\kappa}_2=\boldsymbol{\kappa}_2(n,s)\in(0,1)$ such that the following holds. Let $u\in \widehat H^s(D_{2R};\mathbb{S}^{d-1})$ be a weakly $s$-harmonic map in $D_{2R}$ such that the function $r\in(0,2R-|{\bf x}|)\mapsto \boldsymbol{\Theta}_s(u^\e,{\bf x},r)$ is nondecreasing for every ${\bf x}\in\partial^0B_{2R}^+$. If  
\begin{equation}\label{condeps0Lip}
%\frac{1}{R^{n-2s}}\mathcal{E}_s(u,D_R)
\boldsymbol{\Theta}_s(u^\e,0,R)\leq\boldsymbol{\eps}_1\,, 
\end{equation}
then $u\in C^{0,1}(D_{\boldsymbol{\kappa}_2R})$ and 
$$R^2\|\nabla u\|^2_{L^\infty(D_{\boldsymbol{\kappa}_2R})}\leq C \boldsymbol{\Theta}_s(u^\e,0,R)\,, $$
for a constant $C=C(n,s)$. 
\end{theorem}

The proof of Theorem \ref{thmepsregLip} consists in considering the system satisfied by the $\mathbb{S}^{d-1}$-valued map $u^\e/|u^\e|$. By Corollary \ref{coroepsreghold}, $u^\e$ is H\"older continuous, and therefore $|u^\e|\geq 1/2$ in a smaller half ball $B_r^+$.  In particular, $v:=u^\e/|u^\e|$ is well defined and H\"older continuous in $B_r^+$. We shall see that it satisfies in the weak sense the degenerate system with {\sl homogeneous} Neumann boundary condition
\begin{equation}\label{strongfromsystumodu}
\begin{cases}
-{\rm div}\big(z^a\rho^2\nabla v\big)= z^a\rho^2|\nabla v|^2 v & \text{in $B^+_r$}\,,\\[5pt]
\displaystyle z^a\rho^2\frac{\partial v}{\partial\nu}=0 & \text{on $\partial^0B^+_r$}\,,
\end{cases}
\end{equation}
with H\"older continuous weight $\rho^2:=|u^\e|^2$.  Up to the extra weight term $\rho^2$, this system fits into the class of degenerate harmonic map systems with free boundary considered in \cite{Rob}. Adjusting the arguments in \cite{Rob} to take care of the extra weight  $\rho^2$, we shall prove that $v$ is Lipschitz continuous in an even smaller half ball. Since $u^\e=v$ on $\partial^0B^+_r$, the conclusion will follow straight away.

\subsection{Proof of Theorem \ref{thmepsregLip}}

The aforementioned Lipschitz estimate on the map $u^\e/|u^\e|$ is the object of the following proposition.

\begin{proposition}\label{themholdimplLip}
Let $u\in \widehat H^s(D_{2R};\mathbb{S}^{d-1})$ be a weakly $s$-harmonic map in $D_{2R}$. Assume that $u^\e\in C^{0,\beta}(B^+_R)$ for some exponent $\beta\in(0,1)$, and that $|u^\e|\geq 1/2$ in $B_R^+$.  Setting $\eta:=R^\beta[u^\e]_{C^{0,\beta}(B^+_R)}$, the map $u^\e/|u^\e|$  is Lipschitz continuous in $\overline B^+_{R/3}$, and 
$$R^2 \|\nabla \big(u^\e/|u^\e|\big)\|^2_{L^\infty(B^+_{R/3})}\leq C_{\eta,\beta} \boldsymbol{\Theta}_s(u^\e,0,R)\,,$$
for a constant $C_{\eta,\beta}=C_{\eta,\beta}(\eta,\beta,n,s)$. 
\end{proposition}

Before proving this proposition, we need to show that $u^\e/|u^\e|$ satisfies system \eqref{strongfromsystumodu} in the weak sense.

\begin{lemma}\label{lemmodphase}
Let $u\in \widehat H^s(D_{2R};\mathbb{S}^{d-1})$ be a weakly $s$-harmonic map in $D_{2R}$. Assume that $\rho:=|u^\e|$ satisfies $\rho\geq 1/2$ a.e. in $B_R^+$. Then the map 
$v:=u^\e/\rho$ belongs to $H^1(B_R^+;\R^d,|z|^a\de{\bf x})$ and it satisfies
$$\int_{B_R^+}z^a\rho^2\nabla v\cdot\nabla \phi \,\de {\bf x}=\int_{B_R^+}z^a\rho^2|\nabla v|^2v\cdot\phi\,\de{\bf x}$$
for every $\phi\in H^1(B_R^+;\R^d,|z|^a\de{\bf x})\cap L^\infty(B_R^+)$ such that $\phi=0$ on $\partial^+B_R$. 
\end{lemma}

\begin{proof}
First recall from \eqref{bdlinftyext} and Lemma \ref{hatH1/2toH1} that $u^\e\in H^1(B_R^+;\R^d,|z|^a\de{\bf x})\cap L^\infty(\R_+^{n+1})$, and consequently, $\rho\in H^1(B_R^+,|z|^a\de{\bf x})\cap L^\infty(\R_+^{n+1})$. By assumption $\rho\geq 1/2$, so that  $1/\rho\in H^1(B_R^+,|z|^a\de{\bf x})\cap L^\infty(\R_+^{n+1})$. The space $H^1(B_R^+,|z|^a\de{\bf x})\cap L^\infty(\R_+^{n+1})$ being an algebra, it follows that 
$v\in H^1(B_R^+;\R^d,|z|^a\de{\bf x})$, and by definition $|v|=1$ a.e. in $B_R^+$. 

Let us now fix $\Phi\in H^1(B_R^+;\R^d,|z|^a\de{\bf x})\cap L^\infty(B_R^+)$ such that  $\Phi=0$ on $\partial^+B_R$. Again, $H^1(B_R^+;\R^d,|z|^a\de{\bf x})\cap L^\infty(B_R^+)$ being an algebra, 
 $\psi:=\Phi-(\Phi\cdot v)v\in H^1(B_R^+;\R^d,|z|^a\de{\bf x})\cap L^\infty(B_R^+)$. It also satisfies $\psi=0$ on $\partial^+B_R$, and by construction, we have $v\cdot\psi=0$ a.e. in $B_R^+$. Now we consider  $\xi:=\rho\psi\in H^1(B_R^+;\R^d,|z|^a\de{\bf x})\cap L^\infty(\R^n)$, which still satisfies $\xi=0$ on $\partial^+B_R$, and $u^\e\cdot\xi=0$  in $B_R^+$. In particular, $u\cdot \xi=0$  on $\partial^0B_R^+$. 
 
 By Proposition \ref{equivsharmfreebdry}, the map $u^\e$ is a weighted weakly harmonic map with free boundary in the half ball $B_R^+$, i.e., it satisfies \eqref{vareqext}. Hence, 
\begin{equation}\label{2334vend19}
\int_{B_R^+}z^a\nabla u^\e\cdot\nabla\xi\,\de {\bf x}=0\,. 
\end{equation}
On the other hand, $\partial_i u^\e=\partial_i\rho v + \rho\partial_i v$ and $\partial_i \xi= \partial_i\rho \psi+\rho\partial_i \psi$ in $B_R^+$ for $i=1,\ldots,n+1$. Then we notice that $v\cdot\psi=0$ implies $v\cdot\partial_i\psi=-\partial_iv\cdot\psi$ in $B_R^+$ for $i=1,\ldots,n+1$. In the same way, the fact that $|v|^2=1$ leads to $v\cdot\partial_iv=0$ in $B_R^+$ for $i=1,\ldots,n+1$. As a consequence, 
$$\partial_iu^\e\cdot\partial_i\xi=\big(\partial_i\rho v+\rho\partial_iv\big)\cdot\big(\partial_i\rho\psi+\rho\partial_i\psi\big) =\rho^2\partial_i v\cdot\partial_i\psi\quad\text{a.e. in $B_R^+$}\,,$$
for $i=1,\ldots,n+1$. Inserting this identity in \eqref{2334vend19} yields 
\begin{equation}\label{preeqphasesam20}
\int_{B_R^+}z^a\rho^2\nabla v\cdot\nabla\psi\,\de {\bf x}=0\,. 
\end{equation}
To conclude, we notice that 
$$\partial_iv\cdot\partial_i\psi=\partial_iv\cdot\big(\partial_i\Phi-(v\cdot\Phi)\partial_iv-(\partial_i v\cdot\Phi+v\cdot\partial_i\Phi)v\big)=\partial_iv\cdot\partial_i\Phi -|\partial_iv|^2v\cdot\Phi\quad\text{a.e. in $B_R^+$}\,,$$
for $i=1,\ldots,n+1$. Using this last identity in \eqref{preeqphasesam20} leads to the announced conclusion. 
\end{proof}

As usual, to deal with homogeneous Neumann condition, we extend the equation to the whole ball by symmetry.  In this way, proving estimates up to the boundary reduces to  prove interior estimates. 

\begin{corollary}\label{eqsymtrizedphase}
Let $u\in \widehat H^s(D_{2R};\mathbb{S}^{d-1})$ be a weakly $s$-harmonic map in $D_{2R}$. Assume that $|u^\e|\geq 1/2$ a.e. in $B_R^+$. Then the function $\rho$ and the map $v$ defined by 
\begin{equation}\label{defmodrhosym}
\rho({\bf x}):=\begin{cases}
|u^\e(x,z)| & \text{if ${\bf x}=(x,z)\in B_R^+$}\\
|u^\e(x,-z)| & \text{if ${\bf x}=(x,z)\in B_R^-$}
\end{cases}
\end{equation}
 and
 \begin{equation}\label{defphsevsym}
 v({\bf x}):=
\begin{cases}
u^\e(x,z)/\rho({\bf x}) & \text{if ${\bf x}=(x,z)\in B_R^+$}\\
u^\e(x,-z)/\rho({\bf x}) & \text{if ${\bf x}=(x,z)\in B_R^-$}
\end{cases}
\end{equation}
belong to  $H^1(B_R,|z|^a\de{\bf x})\cap L^\infty(B_R)$ and  $H^1(B_R;\R^d,|z|^a\de{\bf x})\cap L^\infty(B_R)$ respectively, and 
\begin{equation}\label{eqsymmetrizedphse}
\int_{B_R}|z|^a\rho^2\nabla  v\cdot\nabla\Phi\,\de{\bf x}=\int_{B_R}|z|^a\rho^2|\nabla  v|^2v\cdot\Phi\,\de{\bf x} 
\end{equation}
holds for every $\Phi\in H^1(B_R;\R^d,|z|^a\de{\bf x})\cap L^\infty(B_R)$ such that $\Phi=0$ on $\partial B_R$.
\end{corollary}

\begin{proof}
The fact that $\rho$ and $v$ belong to $H^1(B_R,|z|^a\de{\bf x})\cap L^\infty(B_R)$ and  $H^1(B_R;\R^d,|z|^a\de{\bf x})\cap L^\infty(B_R)$ respectively follows from Lemma \ref{lemmodphase} together with the symmetry with respect to the hyperplane $\{z=0\}$. 

We now consider an arbitrary $\Phi\in H^1(B_R;\R^d,|z|^a\de{\bf x})\cap L^\infty(B_R)$ satisfying $\Phi=0$ on $\partial B_R$. We split $\Phi$ into its symmetric and anti-symmetric parts defined by 
$$\Phi^s(x,z):=\frac{\Phi(x,z)+\Phi(x,-z)}{2}\quad\text{and}\quad \Phi^a(x,z):=\frac{\Phi(x,z)-\Phi(x,-z)}{2}\,. $$
Clearly, $\Phi^s,\Phi^a\in H^1(B_R;\R^d,|z|^a\de{\bf x})\cap L^\infty(B_R)$ and $\Phi^s=\Phi^a=0$ on $\partial B_R$. By construction, we have $\Phi^s(x,-z)=\Phi^s(x,z)$ and $\Phi^a(x,-z)=-\Phi^a(x,z)$, so that 
$\partial_z \Phi^s(x,z)=-\partial_z \Phi^s(x,-z)$ and $\partial_z \Phi^a(x,z)=\partial_z \Phi^a(x,-z)$. The map $v$ being symmetric with respect to $\{z=0\}$, it also satisfies $\partial_z v(x,z)=-\partial_zv(x,-z)$. Therefore, 
$$(\nabla v\cdot\nabla\Phi^s)(x,z)=(\nabla v\cdot\nabla\Phi^s)(x,-z) \quad\text{and}\quad (\nabla v\cdot\nabla\Phi^a)(x,z)=-(\nabla v\cdot\nabla\Phi^a)(x,-z) \,.$$
As a first consequence, 
\begin{equation}\label{antisym1}
\int_{B_R}|z|^a\rho^2\nabla  v\cdot\nabla\Phi^a\,\de{\bf x}=0\,.
\end{equation}
Since $(v\cdot\Phi^a)(x,-z)=-(v\cdot\Phi^a)(x,z)$, we also have 
\begin{equation}\label{antisym2}
\int_{B_R}|z|^a\rho^2|\nabla  v|^2v\cdot\Phi^a\,\de{\bf x} =0\,.
\end{equation}
Then we infer from Lemma  \ref{lemmodphase} that 
\begin{multline}\label{symparteq}
\int_{B_R}|z|^a\rho^2\nabla  v\cdot\nabla\Phi^s\,\de{\bf x}=2\int_{B^+_R}z^a\rho^2\nabla  v\cdot\nabla\Phi^s\,\de{\bf x}\\
= 2\int_{B^+_R}z^a\rho^2|\nabla  v|^2v\cdot\Phi^s\,\de{\bf x} =\int_{B_R}|z|^a\rho^2|\nabla  v|^2v\cdot\Phi^s\,\de{\bf x}\,. 
\end{multline}
Gathering \eqref{antisym1}, \eqref{antisym2}, and \eqref{symparteq} leads to \eqref{eqsymmetrizedphse}, and the proof is complete. 
\end{proof}

\begin{proof}[Proof of Proposition \ref{themholdimplLip}] Rescaling variables, we can assume without loss of generality that $R=1$. 
Throughout the proof, we shall write for a measurable set $A\subset\R^{n+1}$, 
$$|A|_a:=\int_{A}|z|^a\,\de{\bf x}\,,$$
and we notice that for ${\bf y}\in\R^n\times\{0\}$, 
\begin{equation}\label{weightvolball}
|B_r({\bf y})|_a=|B_r|_a=|B_1|_ar^{n+2-2s}\,. 
\end{equation}
We start by applying Corollary \ref{eqsymtrizedphase} to consider the (symmetrized) modulus function $\rho$ and the (symmetrized) phase map $v$ defined by \eqref{defmodrhosym} and \eqref{defphsevsym}, respectively. Since $u^\e$ belongs to $C^{0,\beta}(B_R^+)$ and $|u^\e|\geq 1/2$ in $B_R^+$, it follows that $v\in C^{0,\beta}(B_R)$, and $\rho\in C^{0,\beta}(B_R)$ with $\rho\geq 1/2$ in $B_R$. 
By Corollary \ref{eqsymtrizedphase}, $v$ satisfies \eqref{eqsymmetrizedphse}, and from this equation we shall obtain that $v\in C^{0,1}(B_{R/3})$. We proceed in several steps. 
\vskip3pt

\noindent{\it Step 1.} Let us fix ${\bf y}\in D_{1/2}\times\{0\}$ and $r\in(0,1/2]$. We consider  the unique weak solution $w\in H^1(B_r({\bf y});\R^d,|z|^a\de{\bf x})$ of 
\begin{equation}\label{eqwlipregproof}
\begin{cases}
{\rm div}(|z|^a\nabla w)=0 &\text{in $B_r({\bf y})$}\,,\\
w=v & \text{on $\partial B_r({\bf y})$}\,, 
\end{cases}
\end{equation}
see Appendix \ref{appendweightharm}. The map $v$ being continuous in $\overline B_r({\bf y})$, it follows from Lemma \ref{maxprincip} that $w\in C^0(\overline B_r({\bf y}))$. 
Moreover, since $v$ is symmetric with respect to the hyperplane $\{z=0\}$, Lemma \ref{symmharmw} tells us that $w$ is also symmetric with respect to $\{z=0\}$. 
\vskip3pt

Now we estimate through Minkowski's inequality, 
\begin{multline}\label{ppffpasideedutou}
\left(\frac{1}{|B_{r/2}|_a}\int_{B_{r/2}({\bf y})}|z|^a\rho^2|\nabla v|^2\,\de{\bf x}\right)^{1/2} \leq \left(\frac{1}{|B_{r/2}|_a}\int_{B_{r/2}({\bf y})}|z|^a\rho^2|\nabla w|^2\,\de{\bf x}\right)^{1/2} \\
+ C\left(\frac{1}{|B_{r}|_a}\int_{B_{r}({\bf y})}|z|^a\rho^2|\nabla (v-w)|^2\,\de{\bf x}\right)^{1/2} \,,
\end{multline}
and our first aim is to estimate the two terms in the right hand side of this inequality. 

From the definition of $\eta$ and the fact that $0\leq \rho\leq 1$, we have 
\begin{equation}\label{holdestrhosq}
|\rho^2({\bf x})-\rho^2({\bf y})|\leq 2\eta|{\bf x}-{\bf y}|^\beta\leq C\eta r^\beta\quad\forall{\bf x}\in B_r({\bf y})\,. 
\end{equation}
Consequently,
\begin{align}
\nonumber \int_{B_{r/2}({\bf y})}|z|^a\rho^2|\nabla w|^2\,\de{\bf x}&\leq \rho^2({\bf y})\int_{B_{r/2}({\bf y})}|z|^a|\nabla w|^2\,\de{\bf x}+\int_{B_{r/2}({\bf y})}|z|^a|\rho^2-\rho^2({\bf y})||\nabla w|^2\,\de{\bf x}\\
\label{estiwcomm2steps} &\leq (1+C\eta r^\beta)\int_{B_{r/2}({\bf y})}|z|^a|\nabla w|^2\,\de{\bf x}\,.
\end{align}
Since $w$ is symmetric with respect to $\{z=0\}$, we infer from Lemma \ref{monotIharmreplac} and \eqref{weightvolball} that the function 
$$t\in(0,r]\mapsto  \frac{1}{|B_{t}|_a}\int_{B_{t}({\bf y})}|z|^a|\nabla w|^2\,\de{\bf x}$$
is nondecreasing. Hence,  
\begin{multline*}
\frac{1}{|B_{r/2}|_a}\int_{B_{r/2}({\bf y})}|z|^a\rho^2|\nabla w|^2\,\de{\bf x}\leq \frac{(1+C\eta r^\beta)}{|B_{r}|_a}\int_{B_{r}({\bf y})}|z|^a|\nabla w|^2\,\de{\bf x}\\
\leq \frac{(1+C\eta r^\beta)}{|B_{r}|_a}\int_{B_{r}({\bf y})}|z|^a|\nabla v|^2\,\de{\bf x}\,,
\end{multline*}
where we have used the minimality of $w$ stated in Lemma \ref{minimalityharmonw} in the last inequality. Using $\rho({\bf y})=1$ and $\rho\geq 1/2$, we now estimate as above, 
\begin{align*}
\int_{B_{r}({\bf y})}|z|^a|\nabla v|^2\,\de{\bf x} & \leq \int_{B_{r}({\bf y})}|z|^a\rho^2|\nabla v|^2\,\de{\bf x}+\int_{B_{r}({\bf y})}|z|^a|\rho^2-\rho^2({\bf y})||\nabla v|^2\,\de{\bf x}\\
&\leq  (1+C\eta r^\beta) \int_{B_{r}({\bf y})}|z|^a\rho^2|\nabla v|^2\,\de{\bf x}\,,
\end{align*}
to reach 
\begin{equation}\label{firstpiecelipregbdry}
 \frac{1}{|B_{r/2}|_a}\int_{B_{r/2}({\bf y})}|z|^a\rho^2|\nabla w|^2\,\de{\bf x}\leq  \frac{(1+C\eta r^\beta)^2}{|B_{r}|_a}\int_{B_{r}({\bf y})}|z|^a\rho^2|\nabla v|^2\,\de{\bf x}\,.
\end{equation}

Next, we recall that $v-w\in H^1(B_r({\bf y});\R^d,|z|^a\de{\bf x})$ satisfies $v-w=0$ on $\partial B_r({\bf y})$. Hence, we can apply Corollary \ref{eqsymtrizedphase} to deduce that
\begin{align}
\nonumber \int_{B_{r}({\bf y})}|z|^a\rho^2&|\nabla (v-w)|^2\,\de{\bf x}=\int_{B_{r}({\bf y})}|z|^a\rho^2\nabla v\cdot \nabla (v-w)\,\de{\bf x}-\int_{B_{r}({\bf y})}|z|^a\rho^2\nabla w\cdot \nabla (v-w)\,\de{\bf x}\\
\label{dim21071936} &=\int_{B_{r}({\bf y})}|z|^a\rho^2|\nabla v|^2 v\cdot (v-w)\,\de{\bf x}-\int_{B_{r}({\bf y})}|z|^a\rho^2\nabla w\cdot \nabla (v-w)\,\de{\bf x}\,.
\end{align}
On the other hand, the equation \eqref{eqwlipregproof} satisfied by $w$ yields 
\begin{align}
\nonumber\int_{B_{r}({\bf y})}|z|^a\rho^2\nabla w\cdot \nabla (v-w)\,\de{\bf x}&=\rho^2({\bf y})\int_{B_{r}({\bf y})}|z|^a\nabla w\cdot \nabla (v-w)\,\de{\bf x}\\
\nonumber &\qquad \qquad\qquad  +  \int_{B_{r}({\bf y})}|z|^a\big(\rho^2-\rho^2({\bf y})\big)\nabla w\cdot \nabla (v-w)\,\de{\bf x}\\
 \label{ctjrspasdim07}&=  \int_{B_{r}({\bf y})}|z|^a\big(\rho^2-\rho^2({\bf y})\big)\nabla w\cdot \nabla (v-w)\,\de{\bf x}\,.
 \end{align}
 By \eqref{holdestrhosq} and the minimality of $w$, we have 
\begin{align}
\nonumber\left| \int_{B_{r}({\bf y})}|z|^a\big(\rho^2-\rho^2({\bf y})\big)\nabla w\cdot \nabla (v-w)\,\de{\bf x}\right| & \leq C\eta r^\beta \int_{B_{r}({\bf y})}|z|^a|\nabla w||\nabla (v-w)|\,\de{\bf x}\\
\nonumber&\leq C\eta r^\beta \int_{B_{r}({\bf y})}|z|^a\big(|\nabla w|^2 +|\nabla v|^2\big)\,\de{\bf x}\\
\nonumber &\leq C\eta r^\beta \int_{B_{r}({\bf y})}|z|^a|\nabla v|^2\,\de{\bf x}\\
\label{pasideedim07}&\leq C\eta r^\beta \int_{B_{r}({\bf y})}|z|^a\rho^2|\nabla v|^2\,\de{\bf x}\,,
\end{align}
where we have used that $\rho\geq 1/2$ in the last inequality. 
 Combining \eqref{dim21071936}, \eqref{ctjrspasdim07}, \eqref{pasideedim07}, and using that $|v|=1$, we infer that 
\begin{equation}\label{ahbahcrottejsp}
\int_{B_{r}({\bf y})}|z|^a\rho^2|\nabla (v-w)|^2\,\de{\bf x} \leq \big(\|v-w\|_{L^\infty(B_r({\bf y}))}+C\eta r^\beta\big)\int_{B_{r}({\bf y})}|z|^a\rho^2|\nabla v|^2\,\de{\bf x}\,.
\end{equation}
Let us now bound $\|v-w\|_{L^\infty(B_r({\bf y}))}$. First, notice that for ${\bf x}\in B_r({\bf y})$, 
\begin{equation}\label{pasunpoildideedim1631}
|v({\bf x})-w({\bf x})|\leq |v({\bf x})-v({\bf y})|+|w({\bf x})-v({\bf y})|\leq C\eta r^\beta+ |w({\bf x})-v({\bf y})|\,.
\end{equation}
Next we observe that for each $i=1,\ldots,d$, the scalar function $w^i-v^i({\bf y})\in H^1(B_r({\bf y}),|z|^a\de{\bf x})$ satisfies in the weak sense
$$\begin{cases}
{\rm div}\big(|z|^a\nabla(w^i-v^i({\bf y}))\big)= 0 & \text{in $B_r({\bf y})$}\,,\\
w^i-v^i({\bf y})= v^i-v^i({\bf y}) & \text{on $\partial B_r({\bf y})$}\,.
\end{cases}$$
It then follows from Lemma \ref{maxprincip} that for each $i=1,\ldots,d$, 
$$\|w^i-v^i({\bf y})\|_{L^\infty(B_r({\bf y}))}\leq \|v^i-v^i({\bf y})\|_{L^\infty(\partial B_r({\bf y}))}\leq \|v-v({\bf y})\|_{L^\infty(\partial B_r({\bf y}))}\leq C\eta r^\beta\,. $$
Back to \eqref{pasunpoildideedim1631}, we have thus obtained
$$\|w-v\|_{L^\infty(B_r({\bf y}))} \leq C\eta r^\beta\,.$$
Using this estimate in \eqref{ahbahcrottejsp}, we derive that 
\begin{equation}\label{secondpieceestireglipbdry}
\int_{B_{r}({\bf y})}|z|^a\rho^2|\nabla (v-w)|^2\,\de{\bf x} \leq C\eta r^\beta\int_{B_{r}({\bf y})}|z|^a\rho^2|\nabla v|^2\,\de{\bf x}\,.
\end{equation}

Now, inserting estimates \eqref{firstpiecelipregbdry} and \eqref{secondpieceestireglipbdry} in \eqref{ppffpasideedutou}, and then squaring both sides of the resulting inequality, we are led to 
$$ \frac{1}{|B_{r/2}|_a}\int_{B_{r/2}({\bf y})}|z|^a\rho^2|\nabla v|^2\,\de{\bf x}\leq \frac{(1+C_\eta r^{\beta/2}) }{|B_{r}|_a}\int_{B_{r}({\bf y})}|z|^a\rho^2|\nabla v|^2\,\de{\bf x} \,,$$ 
for a constant $C_\eta=C_\eta(\eta,n,s)$.  Iterating this inequality along dyadic radii $r_k:=2^{-k}$ with $k\geq 1$, we deduce that 
\begin{multline}\label{mard33july}
 \frac{1}{|B_{r_{k+1}}|_a}\int_{B_{r_{k+1}}({\bf y})}|z|^a\rho^2|\nabla v|^2\,\de{\bf x}  \leq \Big(\prod_{j=1}^k(1+C_\eta 2^{-j\beta/2})\Big) \frac{1}{|B_{1/2}|_a}\int_{B_{1/2}({\bf y})}|z|^a\rho^2|\nabla v|^2\,\de{\bf x} \\
 \leq C_{\eta,\beta} \int_{B_{1}}|z|^a\rho^2|\nabla v|^2\,\de{\bf x} \,,
 \end{multline}
 for a constant $C_{\eta,\beta}=C_{\eta,\beta}(\eta,\beta,n,s)$. Next, for an arbitrary radius $r\in(0,1/2]$, we consider the integer $k\geq 1$ satisfying $r_{k+1} <r\leq r_k$, and estimate 
$$ \frac{1}{|B_{r}|_a}\int_{B_{r}({\bf y})}|z|^a\rho^2|\nabla v|^2\,\de{\bf x}  \leq \frac{2^{n+2-2s}}{|B_{r_k}|_a}\int_{B_{r_k}({\bf y})}|z|^a\rho^2|\nabla v|^2\,\de{\bf x}  \,,$$
to conclude from \eqref{mard33july} and the symmetry of $v$ and $\rho$ with respect to $\{z=0\}$ that 
$$\frac{1}{|B_{r}|_a}\int_{B_{r}({\bf y})}|z|^a\rho^2|\nabla v|^2\,\de{\bf x}  \leq C_{\eta,\beta} \int_{B^+_{1}}|z|^a\rho^2|\nabla v|^2\,\de{\bf x}\quad\forall r\in(0,1/2]\,.$$
Noticing that $|\nabla u^\e|^2=|\nabla \rho|^2+\rho^2|\nabla v|^2$, and in view of the arbitrariness of ${\bf y}$, we have thus proved that 
\begin{equation}\label{gfaimadonfpoilonez}
\frac{1}{|B_{r}|_a}\int_{B_{r}({\bf y})}|z|^a\rho^2|\nabla v|^2\,\de{\bf x}  \leq C_{\eta,\beta} \int_{B^+_{1}}|z|^a|\nabla u^\e|^2\,\de{\bf x}\quad\forall {\bf y}\in D_{1/2}\times\{0\}\,,\;\forall r\in(0,1/2]\,.
\end{equation}

\vskip5pt

\noindent{\it Step 2.} Our main goal in this step is to obtain an estimate similar to \eqref{gfaimadonfpoilonez} for balls which are not centered at points of $\{z=0\}$. By symmetry of $v$ and $\rho$ with respect to $\{z=0\}$, it is enough to consider balls centered at points of $\R^{n+1}_+$.  

Let us fix an arbitrary point ${\bf y}=(y,t)\in B^+_{1/3}$, and notice that $\overline B_{t/2}({\bf y})\subset B_1^+$ . We also consider an arbitrary radius $r\in(0,t/2]$ (so that $\overline B_r({\bf y})\subset B_1^+$). As in Step 1, we introduce the (weak) solution $w\in H^1(B_r({\bf y});\R^d,|z|^a\de{\bf x})$ of \eqref{eqwlipregproof}. Exactly as in \eqref{ppffpasideedutou}, we have 
\begin{multline}\label{ppffpasideedutouStep2}
\left(\Big(\frac{2}{r}\Big)^{n+1}\int_{B_{r/2}({\bf y})}|z|^a\rho^2|\nabla v|^2\,\de{\bf x}\right)^{1/2} \leq \left(\Big(\frac{2}{r}\Big)^{n+1}\int_{B_{r/2}({\bf y})}|z|^a\rho^2|\nabla w|^2\,\de{\bf x}\right)^{1/2} \\
+ C\left(\frac{1}{r^{n+1}}\int_{B_{r}({\bf y})}|z|^a\rho^2|\nabla (v-w)|^2\,\de{\bf x}\right)^{1/2} \,.
\end{multline}
Arguing precisely as in Step 1, we derive that \eqref{secondpieceestireglipbdry} still holds. Then, we estimate as in \eqref{estiwcomm2steps}, 
\begin{equation}\label{prejecpa1}
 \int_{B_{r/2}({\bf y})}|z|^a\rho^2|\nabla w|^2\,\de{\bf x}\leq (\rho^2({\bf y})+C\eta r^\beta)\int_{B_{r/2}({\bf y})}|z|^a|\nabla w|^2\,\de{\bf x}\,.
 \end{equation}
Applying Lemma \ref{monotIharmreplac2} with $\theta=t/r$ and then the minimality of $w$, we obtain   
\begin{multline}\label{prejecpa2}
\Big(\frac{2}{r}\Big)^{n+1} \int_{B_{r/2}({\bf y})}|z|^a|\nabla w|^2\,\de{\bf x}\leq \Big(1+\frac{Cr}{t}\Big)\frac{1}{r^{n+1}} \int_{B_{r}({\bf y})}|z|^a|\nabla w|^2\,\de{\bf x}\\
\leq  \Big(1+\frac{Cr}{t}\Big)\frac{1}{r^{n+1}} \int_{B_{r}({\bf y})}|z|^a|\nabla v|^2\,\de{\bf x}\,.
\end{multline}
Combining \eqref{prejecpa1} with \eqref{prejecpa2}, and using again the H\"older continuity of $\rho^2$ (as in \eqref{holdestrhosq}) together with $1/2\leq \rho\leq 1$, we deduce that 
\begin{align}\label{presqfinimarjul}
\Big(\frac{2}{r}\Big)^{n+1}  \int_{B_{r/2}({\bf y})}|z|^a\rho^2|\nabla w|^2\,\de{\bf x}  \leq \Big(1+C(\eta r^\beta +r/t) \Big) \frac{1}{r^{n+1}}  \int_{B_{r}({\bf y})}|z|^a\rho^2|\nabla v|^2\,\de{\bf x}\,.
% &\leq  \Big(1+C(\eta+1/t)r^\beta \Big) \frac{1}{r^{n+1}}  \int_{B_{r}({\bf y})}|z|^a\rho^2|\nabla v|^2\,\de{\bf x}\,.
 \end{align}
 Inserting \eqref{secondpieceestireglipbdry} and \eqref{presqfinimarjul} in \eqref{ppffpasideedutouStep2}, we infer that 
 $$\frac{1}{|B_{r/2}({\bf y})|}\int_{B_{r/2}({\bf y})}|z|^a\rho^2|\nabla v|^2\,\de{\bf x}  \leq \frac{1+C_{\eta}(r^{\beta/2}+r/t)}{|B_r({\bf y})|} \int_{B_{r}({\bf y})}|z|^a\rho^2|\nabla v|^2\,\de{\bf x} \,,$$
 for a constant $C_{\eta}=C_{\eta}(\eta,n,s)$. Arguing as Step 1 (using the dyadic radii $r_k:=2^{-k}t$), the arbitrariness of $r\in(0,t/2]$ in this latter estimate implies that 
 \begin{equation}\label{mercrecanic}
  \frac{1}{|B_r({\bf y})|}\int_{B_{r}({\bf y})}|z|^a\rho^2|\nabla v|^2\,\de{\bf x}  \leq \frac{C_{\eta,\beta}}{|B_{t/2}({\bf y})|}   \int_{B_{t/2}({\bf y})}|z|^a\rho^2|\nabla v|^2\,\de{\bf x}\quad\forall r\in(0,t/2]\,,
  \end{equation}
 for a constant $C_{\eta,\beta}=C_{\eta,\beta}(\eta,\beta,n,s)$. Then, we notice that for every radius $r\in(0,t/2]$, 
 $$ |B_r({\bf y})|_a\leq 
 \begin{cases}
t^a(1+r/t)^a|B_r({\bf y})| & \text{if $s\leq 1/2$}\,,\\
t^a(1-r/t)^a|B_r({\bf y})| & \text{if $s >1/2$}\,,  
 \end{cases}$$
 and 
$$ |B_r({\bf y})|_a\geq 
 \begin{cases}
t^a(1-r/t)^a|B_r({\bf y})| & \text{if $s\leq 1/2$}\,,\\
t^a(1+r/t)^a|B_r({\bf y})| & \text{if $s >1/2$}\,.  
 \end{cases}$$
Consequently, dividing \eqref{mercrecanic} by $t^a$, we obtain 
\begin{equation}\label{mercrecanic2}
\frac{1}{|B_r({\bf y})|_a}\int_{B_{r}({\bf y})}|z|^a\rho^2|\nabla v|^2\,\de{\bf x}  \leq \frac{C_{\eta,\beta}}{|B_{t/2}({\bf y})|_a}   \int_{B_{t/2}({\bf y})}|z|^a\rho^2|\nabla v|^2\,\de{\bf x}\quad\forall r\in(0,t/2]\,.
\end{equation}
Setting $\widetilde {\bf y}:=(y,0)\in D_{1/3}\times\{0\}$, we now observe that $B_{t/2}({\bf y})\subset B^+_{3t/2}(\widetilde{\bf y})$ and $3t/2\leq 1/2$. 
Using the symmetry of $v$ and $\rho$ with respect to $\{z=0\}$ and \eqref{gfaimadonfpoilonez}, we deduce that
\begin{align}\label{mercrecanic3}
\nonumber \frac{1}{|B_{t/2}({\bf y})|_a}   \int_{B_{t/2}({\bf y})}|z|^a\rho^2|\nabla v|^2\,\de{\bf x} & \leq \frac{C}{|B^+_{3t/2}(\widetilde {\bf y})|_a}   \int_{B^+_{3t/2}(\widetilde {\bf y})}|z|^a\rho^2|\nabla v|^2\,\de{\bf x}\\
\nonumber & \leq \frac{C}{|B_{3t/2}(\widetilde {\bf y})|_a}   \int_{B_{3t/2}(\widetilde {\bf y})}|z|^a\rho^2|\nabla v|^2\,\de{\bf x} \\
 &\leq C_{\eta,\beta} \int_{B^+_{1}}|z|^a|\nabla u^\e|^2\,\de{\bf x}\,.
 \end{align}
 Combining \eqref{mercrecanic2} and \eqref{mercrecanic3}, and in view of the arbitrariness of ${\bf y}$, we infer that
$$ \frac{1}{|B_r({\bf y})|_a}\int_{B_{r}({\bf y})}|z|^a\rho^2|\nabla v|^2\,\de{\bf x}  \leq C_{\eta,\beta} \int_{B^+_{1}}|z|^a|\nabla u^\e|^2\,\de{\bf x}\quad\forall {\bf y}=(y,t)\in B_{1/3}^+\,,\;\forall r\in(0,t/2]\,.$$
Still by symmetry of $v$ and $\rho$, this estimate actually holds for every ${\bf y}=(y,t)\in  B_{1/3}\setminus\{z=0\}$ and $r\in(0,|t|/2)$. By Lebesgue's differentiation theorem, we have thus proved that 
$$\rho^2|\nabla v|^2\leq C_{\eta,\beta} \int_{B^+_{1}}|z|^a|\nabla u^\e|^2\,\de{\bf x} \quad\text{a.e. in $B_{1/3}$}\,,$$
and the conclusion follows from the fact that $\rho\geq 1/2$. 
 \end{proof}

\begin{proof}[Proof of Theorem \ref{thmepsregLip}]
Once again, rescaling variables, we can assume that $R=1$. Under condition \eqref{condeps0Lip}, Corollary \ref{coroepsreghold} says that $u^\e\in C^{0,\beta_1}(B^+_{\boldsymbol{\kappa}_1})$ and  $[u^\e]_{C^{0,\beta_1}(B^+_{\boldsymbol{\kappa}_1})}$ is bounded by a constant depending only on $n$ and $s$. Since $|u^\e|=|u|=1$ on $\partial^0 B^+_{\boldsymbol{\kappa}_1}$, we can thus find a constant $\boldsymbol{\kappa}_2=\boldsymbol{\kappa}_2(n,s)\in(0,1)$ such that $6\boldsymbol{\kappa}_2\leq \boldsymbol{\kappa}_1$ and $|u^\e|\geq 1/2$ in $B^+_{3\boldsymbol{\kappa}_2}$. Since $\beta_1=\beta_1(n,s)$,  and $(3\boldsymbol{\kappa}_2)^{\beta_1}[u^\e]_{C^{0,\beta_1}(B^+_{3\boldsymbol{\kappa}_2})}$ is bounded by a constant depending only on $n$ and $s$, Proposition \ref{themholdimplLip} implies that  $v:=u^\e/|u^\e|$ is Lipschitz continuous in $\overline B^+_{\boldsymbol{\kappa}_2}$ with 
$$|v({\bf x})-v({\bf y})|\leq C \boldsymbol{\Theta}_s(u^\e,0,\boldsymbol{\kappa}_2)|{\bf x}-{\bf y}|\leq C\boldsymbol{\Theta}_s(u^\e,0,1)|{\bf x}-{\bf y}|\quad\forall {\bf x},{\bf y}\in \overline B^+_{\boldsymbol{\kappa}_2}\,,$$
for a constant $C=C(n,s)$. Since $v({\bf x})=u(x)$ for every ${\bf x}=(x,0)\in \partial^0B^+_{\boldsymbol{\kappa}_2}$, the conclusion follows. 
\end{proof}

%
%\begin{proof}[Proof of Theorem \ref{Lipregsubcritic}]
%We proceed as in the proof above, assuming that $R=1$. Under condition \eqref{condlipregsubcritic}, Proposition \ref{propHoldsubcritic} implies that $u\in C^{0,s-1/2}(D_{R/4})$ 
%with $[u^\e]^2_{C^{0,s-1/2}(B^+_{R/4})}\leq C(M+1)$ for a constant depending only on $s$.  Since $|u^\e|=1$ on $\partial^0B^+_{R/4}$, we can find a constant $\boldsymbol{\kappa}_M\in(0,1/12)$ depending only on $M$ and $s$ such that  $|u^\e|\geq 1/2$ in $B^+_{3\boldsymbol{\kappa}_M}$.  Then the conclusion follows from Proposition \ref{themholdimplLip} as above. 
%\end{proof}
%

%%%%%%%%%%%%%%%%%%%%%%%%%%%%%%%%%%%%%%%%%%%%%%%%%%%%%%%
%%%%%%%%%%%%%%%%%%%%%%%%%%%%%%%%%%%%%%%%%%%%%%%%%%%%%%%
   								       						%%%%%%%%%%%%%%%%%%%
\section{Higher order regularity}\label{highordreg} 					%%%%%%%%%
								 						%%%%%%%%%%%%%%%%%%%
%%%%%%%%%%%%%%%%%%%%%%%%%%%%%%%%%%%%%%%%%%%%%%%%%%%%%%%
%%%%%%%%%%%%%%%%%%%%%%%%%%%%%%%%%%%%%%%%%%%%%%%%%%%%%%%

We have now reached the final stage of our small energy regularity result where it only remains to prove that a Lipschitz continuous  $s$-harmonic map is of class $C^\infty$. To achieve this result, we shall apply  (local) Schauder type estimates for $(-\Delta)^s$. We only refer to \cite{RosSer} for those estimates as it is best suited to our presentation (see also \cite{Sil}).

\begin{theorem}\label{highordthm}
Let $u\in \widehat H^s(D_{1};\mathbb{S}^{d-1})$ be a weakly $s$-harmonic map in $D_{1}$. If $u$ is Lipschitz continuous in $D_1$, then $u\in C^\infty(D_{1/2})$.  
\end{theorem}

\begin{proof}
The proof of Theorem \ref{highordthm} follows from a bootstrap procedure. The initiation of the induction consists in passing from Lipschitz regularity to $C^{1,\alpha}$-regularity, and it is the object of Proposition \ref{C1alphareg}  in the following subsection. Then we shall prove in Proposition \ref{Ckalphareg} that $C^{k,\alpha}$-regularity upgrades to $C^{k+1,\alpha}$-regularity for every integer $k\geq 1$. In applying this bootstrap argument, we first fix an arbitrary point $x_0\in D_{1/2}$ and an integer $k\geq 1$. We translate variables by $x_0$ and rescale suitably  in order to apply Proposition \ref{C1alphareg} and Proposition \ref{Ckalphareg}, and then conclude that $u$ is $C^{k,\alpha}$ in a neighborhood of $x_0$. 
\end{proof}

\subsection{H\"older continuity of first order derivatives}

\begin{proposition}\label{C1alphareg}
Let $u\in \widehat H^s(D_{3};\mathbb{S}^{d-1})$ be a weakly $s$-harmonic map in $D_{3}$. If $u$ is Lipschitz continuous in $D_3$, then $u\in C^{1,\alpha}(D_{r_*})$ for every $\alpha\in(0,1)$ and some  $r_*=r_*(n,s)\in(0,1/2)$. 
%\begin{enumerate}
%\item for $s\in(1/2,1)$, $u\in C^{1,\alpha}(D_{1/2})$ for every $\alpha<2s-1$; 
%\item for $s\in(0,1/2)$, $u\in C^{1,\alpha}(D_{1/2})$ for every $\alpha<2s$. 
%\end{enumerate}
\end{proposition}

One of the main ingredients to obtain an improved regularity is the following elementary lemma. 

\begin{lemma}\label{keybootstraplemma}
Let $f:D_3\to \R^d$ be a Lipschitz continuous function, $g:D_3\to\R^d$ an H\"older continuous function,  
 and $\zeta:D_1\to [0,1]$ a measurable function. Assume that one of the following items holds: 
\begin{enumerate}
\item[(i)] $s\in(0,1/2)$ and $g\in C^{0,\alpha}(D_3)$ for some $\alpha\in(2s,1]$;
\item[(ii)] $s\in(0,1/2)$ and $g\in C^{0,\alpha}(D_3)$ for every $\alpha\in(0,2s)$;  
\item[(iii)] $s\in[1/2,1)$ and $g\in C^{0,\alpha}(D_3)$ for every $\alpha\in(0,1)$. 
\end{enumerate} 
 Then the function 
\begin{equation}\label{defGfunctionreg}
 G:x\in D_{1}\mapsto \int_{D_1}\frac{\big(f(x+y)-f(x)\big)\cdot\big(g(x+y)-g(x)\big)}{|y|^{n+2s}}\zeta(y)\,\de y 
 \end{equation}
belongs to 
\begin{enumerate}
\item $C^{0,\alpha}(D_1)$ in case (i);
\item $C^{0,\alpha^\prime}(D_1)$ for every $\alpha^\prime\in(0,2s)$ in case (ii);
\item $C^{0,\alpha^\prime}(D_1)$ for every $\alpha^\prime\in(0,2-2s)$ in case (iii). 
\end{enumerate} 
\end{lemma}

\begin{proof}
{\it Step 1.} We first claim that $G$ is well defined in all cases. To simplify the notation, we write 
\begin{equation}\label{notgamGfct}
\Gamma(x,y):= \big(f(x+y)-f(x)\big)\cdot\big(g(x+y)-g(x)\big)\,.
\end{equation}
Observe that in all cases, we have $1+\alpha>2s$ (it holds for every $\alpha\in(0,2s)$ in case (ii), and we can choose such $\alpha\in(0,1)$ in case (iii)). Since $|\Gamma(x,y)|\leq C_{f,g,\alpha}|y|^{1+\alpha}$, we have  
$$\int_{D_1}\frac{|\Gamma(x,y)|}{|y|^{n+2s}}\,\de y\leq C_{f,g,\alpha}\int_{D_1}\frac{\de y}{|y|^{n+2s-(1+\alpha)}}\leq C_{f,g,\alpha}\quad\forall x\in D_1\,, $$
for a constant $C_{f,g,\alpha}$ depending only on $f$, $g$, $\alpha$, $n$, and $s$. 
\vskip5pt

\noindent{\it Step 2, case (i).} Fix arbitrary points $x,h\in D_1$. Since
\begin{equation}\label{debilestiapriorbisbis}
\big|\Gamma(x+h,y)- \Gamma(x,y)\big| 
\leq C_{f,g,\alpha}|h|^\alpha|y|^\alpha \quad\forall y\in D_1\,,
\end{equation}
we have 
$$|G(x+h)-G(x)|\leq C_{f,g,\alpha}|h|^\alpha \int_{D_1}\frac{1}{|y|^{n+2s-\alpha}}\,\de y \leq C_{f,g,\alpha}|h|^\alpha\,,  $$
for a constant $C_{f,g,\alpha}$ depending only on $f$, $g$, $\alpha$, $n$, and $s$. 
\vskip5pt

\noindent{\it Step 3, case (ii).} Let us fix an arbitrary $\varepsilon\in(0,s)$. We set $\alpha:=2s-\eps$ and $\beta:=1-2\varepsilon$. Since
$$\big|\Gamma(x+h,y)- \Gamma(x,y)\big|\leq \big|\Gamma(x+h,y)\big|+\big|\Gamma(x,y)\big| \leq C_{f,g,\eps}|y|^{1+\alpha}\,,$$
we can use \eqref{debilestiapriorbisbis} to obtain 
\begin{equation}\label{vendre02aout1}
\big|\Gamma(x+h,y)- \Gamma(x,y)\big|\leq  C_{f,g,\eps}|y|^{(1+\alpha)(1-\beta)}|h|^{\alpha\beta}|y|^{\alpha\beta}=C_{f,g,\eps}|y|^{2s+\eps}|h|^{\alpha\beta} \quad\forall y\in D_1\,.
\end{equation}
Hence, 
\begin{equation}\label{vendre02aout2}
|G(x+h)-G(x)|\leq C_{f,g,\eps}|h|^{\alpha\beta}\int_{D_1}\frac{1}{|y|^{n-\eps}}\,\de y \leq C_{f,g,\eps}|h|^{\alpha\beta}\,, 
\end{equation}
for a constant $C_{f,g,\eps}>0$ depending only on $f$, $g$, $\eps$, $n$, and $s$. 
\vskip5pt

\noindent{\it Step 4, case (iii).} Now we fix an arbitrary $\varepsilon\in(0,1-s)$, and we set $\alpha:=1-\eps$ and $\beta:=2-2s-2\varepsilon$. Then \eqref{vendre02aout1} still holds, and consequently also \eqref{vendre02aout2}. 
\end{proof}

\begin{proof}[Proof of Proposition \ref{C1alphareg}]
{\it Step 1.} We start by fixing a radial cut-off function $\zeta\in \mathscr{D}(\R^n)$ such that $0\leq \zeta\leq 1$, $\zeta=1$ in $D_{1/2}$, and $\zeta=0$ in $\R^n\setminus D_{3/4}$. With $\zeta$ in hands, we rewrite for $x\in D_1$, 
\begin{multline}\label{rewriterhs1}  
\int_{\R^n}\frac{|u(x)-u(y)|^2}{|x-y|^{n+2s}}\,\de y 
=\int_{\R^n}\frac{|u(x+y)-u(x)|^2}{|y|^{n+2s}}\,\de y\\
= \int_{D_1}\frac{|u(x+y)-u(x)|^2}{|y|^{n+2s}}\zeta(y)\,\de y + \int_{D^c_{1/2}}\frac{|u(x+y)-u(x)|^2}{|y|^{n+2s}}(1-\zeta(y))\,\de y \,,
 \end{multline}
 and we set 
 \begin{equation}\label{deffctGu}
 G_u(x):= \int_{D_1}\frac{|u(x+y)-u(x)|^2}{|y|^{n+2s}}\zeta(y)\,\de y \,. 
 \end{equation}
 By Lemma \ref{keybootstraplemma} (applied to $f=g=u$), the function $G_u$ is Lipschitz continuous in $D_1$ for $s\in(0,1/2)$, and it belongs to $C^{0,\alpha}(D_1)$ for every $\alpha\in(0,2-2s)$ for $s\in[1/2,1)$. 
 
 Concerning the second term in the right hand side of \eqref{rewriterhs1}, we use the identity $|u|^2=1$ to rewrite it as 
\begin{multline} \label{rewriterhs2}
\int_{D^c_{1/2}}\frac{|u(x+y)-u(x)|^2}{|y|^{n+2s}}(1-\zeta(y))\,\de y =\int_{\R^n}\frac{2(1-\zeta(y))}{|y|^{n+2s}}\,\de y\\
-\left(\int_{\R^n}\frac{2(1-\zeta(y))}{|y|^{n+2s}}u(x+y)\,\de y\right)\cdot u(x)\,.
\end{multline}
In view of \eqref{rewriterhs2}, it is convenient to introduce the constant $L_\zeta>0$ and the function $Z\in C^\infty(\R^n)$ given by 
$$L_\zeta:= \int_{\R^n}\frac{2(1-\zeta(y))}{|y|^{n+2s}}\,\de y\quad\text{and}\quad Z(x):= \frac{2}{L_\zeta} \frac{(1-\zeta(x))}{|x|^{n+2s}}\,.$$
In this way, the right-hand side of \eqref{rewriterhs2} can be written as 
 \begin{equation}\label{deffctHu}
H_u(x):=L_\zeta\big(1-Z*u(x)\cdot u(x)\big) \quad\text{for $x\in D_1$}\,.
 \end{equation}
Notice that $Z*u\in C^{\infty}(\R^n)$, so that $H_u$ is Lipschitz continuous in $D_1$.  

Summarizing our manipulations in \eqref{rewriterhs1} and  \eqref{rewriterhs2}, we have obtained 
$$ \int_{\R^n}\frac{|u(x)-u(y)|^2}{|x-y|^{n+2s}}\,\de y=G_u(x)+H_u(x)\qquad\forall x\in D_1\,.$$
Now we introduce the map $F_u:D_1\to \R^d$ given by 
\begin{equation}\label{deffctFu}
F_u(x):=\frac{\gamma_{n,s}}{2}\big(G_u(x)+H_u(x)\big)u(x)\,.
 \end{equation}
Then $F_u\in C^{0,1}(D_1)$ for $s\in(0,1/2)$, and $F_u\in C^{0,\alpha}(D_1)$ for every $\alpha\in(0,2-2s)$ for $s\in[1/2,1)$. 
\vskip5pt

\noindent{\it Step 2.} We consider the map $u_0:\R^n\to \R^d$ given by $u_0:=\zeta u$. Then $u_0 \in C^{0,1}(\R^n)$ and $u_0=0$ in $\R^n\setminus D_1$. In particular, $u_0\in H^s_{00}(D_1;\R^d)$. A lengthy but  straightforward computation shows that 
$$(-\Delta)^su_0= \zeta(-\Delta)^su+\big((-\Delta)^s\zeta\big)u-\gamma_{n,s}\int_{\R^n}\frac{(\zeta(x)-\zeta(y))(u(x)-u(y))}{|x-y|^{n+2s}}\,\de y\quad\text{in $H^{-s}(D_1;\R^d)$}\,,$$
i.e., in the sense of \eqref{deffraclap}. Since $u$ is a weakly $s$-harmonic map in $D_3$, it satisfies equation \eqref{ELeqorig}. In view of Step 1, we thus have 
\begin{equation}\label{eqofprodzetau}
(-\Delta)^su_0= \zeta F_u+\big((-\Delta)^s\zeta\big)u-\gamma_{n,s}\int_{\R^n}\frac{(\zeta(x)-\zeta(y))(u(x)-u(y))}{|x-y|^{n+2s}}\,\de y\quad\text{in $H^{-s}(D_1;\R^d)$}\,.
\end{equation}
The function $(-\Delta)^s\zeta$ being smooth over $\R^n$, we infer from Step 1 that $\zeta F_u+\big((-\Delta)^s\zeta\big)u$ belongs to $C^{0,1}(D_1)$ for $s\in(0,1/2)$, and to $C^{0,\alpha}(D_1)$ for every $\alpha\in(0, 2-2s)$ for $s\in[1/2,1)$.  We now  determine the regularity of the last term in the right-hand side of \eqref{eqofprodzetau} arguing as in Step 1. We write it as 
$$\int_{\R^n}\frac{(\zeta(x)-\zeta(y))(u(x)-u(y))}{|x-y|^{n+2s}}\,\de y=:I(x)+II(x)\,, $$
with 
$$I(x):=\int_{D_1}\frac{(\zeta(x+y)-\zeta(x))(u(x+y)-u(x))}{|y|^{n+2s}}\zeta(y)\,\de y\,, $$
and 
\begin{align*}
II(x)& := \int_{\R^n}\frac{(\zeta(x+y)-\zeta(x))(u(x+y)-u(x))}{|y|^{n+2s}}(1-\zeta(y))\,\de y\\
&= \int_{\R^n}\frac{(\zeta(x)-\zeta(y))(u(x)-u(y))}{|x-y|^{n+2s}}(1-\zeta(x-y))\,\de y\,.
\end{align*}
By Lemma \ref{keybootstraplemma}, the term $I$ belongs to $C^{0,1}(D_1)$ for $s\in(0,1/2)$, and to $C^{0,\alpha}(D_1)$ for every $\alpha\in(0,2-2s)$ for $s\in[1/2,1)$. On the other hand, the function $\zeta$ being smooth and equal to $1$ in $D_{1/2}$, the term $II$ has clearly the regularity of $u$ in $D_1$, that is $C^{0,1}(D_1)$. Summarizing these considerations, we have shown that $u_0\in H^s(\R^n;\R^d)\cap L^\infty(\R^n)$ is a weak solution of 
$$ 
\begin{cases}
(-\Delta)^s u_0=F_0 & \text{in $D_1$}\,,\\
u_0=0 & \text{in $\R^n\setminus D_1$}\,,
\end{cases}
$$
for a right-hand side $F_0$ which belongs to $C^{0,1}(D_1)$ for $s\in(0,1/2)$, and to $C^{0,\alpha}(D_1)$ for every $\alpha\in(0,2-2s)$ for $s\in[1/2,1)$. From well-known (by now) regularity estimates for this equation (see e.g. \cite[Section 2]{RosSer}), the map $u_0$ belongs to $C^{1,\alpha}(D_{1/2})$ for every $\alpha\in(0,2s)$ for $s\in(0,1/2)$, and to $C^{1,\alpha}(D_{1/2})$ for every $\alpha\in(0,1)$ for $s\in[1/2,1)$. 
Since $u_0=u$ in $D_{1/2}$, the proposition is proved  in the case $s\in[1/2,1)$, and we obtained $u\in C^{1,\alpha}(D_{1/2})$ for every $\alpha\in(0,2s)$ for $s\in (0,1/2)$. 

\vskip5pt

\noindent{\it Step 3.} We now assume that $s\in (0,1/2)$, and it remains to prove that $u$ actually belongs to $C^{1,\alpha}(D_{r_*})$ for every $\alpha\in (0,1)$ and a radius $r_*\in(0,1/2)$ depending only on $s$. 
To this purpose, we rescale $u$ by setting $\widetilde u(x):=u(x/6)$, and from Step 3, we infer that $\widetilde u\in C^{1,\alpha}(D_3)$ for every $\alpha\in(0,2s)$. We shall  now make use of the following lemma. 

\begin{lemma}\label{regC1alphafctGreg}
Assume that $s\in(0,1/2)$. Let $f:D_3\to \R^d$ and $g:D_3\to \R^d$  be two $C^1$-functions,   
 and $\zeta:D_1\to [0,1]$ a measurable function. Assume that one of the following items holds: 
 \begin{enumerate}
 \item[(i)] $f,g \in C^{1,\alpha}(D_3)$ for every  $\alpha\in(0,2s)$;
  \item[(ii)] $f,g \in C^{1,\alpha}(D_3)$ for some  $\alpha\in(2s,1)$;
 \end{enumerate}
  Then the function $G:D_1\to\R$ given by \eqref{defGfunctionreg} belongs to 
\begin{enumerate}
\item $C^{1,\alpha^\prime}(D_1)$ for every $\alpha^\prime\in(0,2s)$ in case (i);
\item $C^{1,\alpha}(D_1)$  in case (ii);
\end{enumerate}
and for $x\in D_1$, 
\begin{multline}\label{formpartialderGfctreg}
\partial_i G(x)= \int_{D_1}\frac{\big(\partial_i f(x+y)-\partial_i f(x)\big)\cdot\big(g(x+y)-g(x)\big)}{|y|^{n+2s}}\zeta(y)\,\de y\\
+\int_{D_1}\frac{\big(f(x+y)-f(x)\big)\cdot\big(\partial_i g(x+y)-\partial_i g(x)\big)}{|y|^{n+2s}}\zeta(y)\,\de y\,,
\end{multline}
for $i=1,\ldots,n$. 
\end{lemma}

\begin{proof}
We keep using notation \eqref{notgamGfct}. First we fix an arbitrary point $x\in D_1$ and we claim that $G$ admits a partial derivative $\partial_i G$ at $x$. Indeed, for $t>0$ small enough, we have 
$$\big|\Gamma(x+t e_i,y)- \Gamma(x,y)\big| 
\leq C_{f,g} |y| t \quad\forall y\in D_1\,,$$
since $f$ and $g$ are $C^1$  over $D_3$. Hence, 
$$\frac{|\Gamma(x+t e_i,y)- \Gamma(x,y)|}{|y|^{n+2s}t}\leq C_{f,g} |y|^{1-2s-n} \in L^1(D_1)\,,$$
and it follows from the dominated convergence theorem that $G$ admits a partial derivative $\partial_i G$ at $x$ given by formula \eqref{formpartialderGfctreg}. 

Next we apply Lemma \ref{keybootstraplemma} to the right-hand side of \eqref{formpartialderGfctreg} to deduce that $\partial_i G$ is H\"older continuous, and the conclusion follows. 
\end{proof}

\noindent{\it Proof of Proposition \ref{C1alphareg} completed.}
We consider the function $G_{\widetilde u}:D_1\to\R$ as defined in \eqref{deffctGu} with $\tilde u$ in place of $u$. By Lemma \ref{regC1alphafctGreg} (applied to $f=g=\widetilde u$), $G_{\widetilde u}\in C^{1,\alpha}(D_1)$ for every $\alpha\in(0,2s)$. On the other hand, the function $H_{\widetilde u}:D_1\to \R$ as defined in \eqref{deffctHu} clearly belongs to $C^{1,\alpha}(D_1)$ for every $\alpha\in(0,2s)$. Consequently, the map $F_{\widetilde u}:D_1\to\R^d$ as defined in \eqref{deffctFu} also belongs to $C^{1,\alpha}(D_1)$ for every $\alpha\in(0,2s)$. Since $\widetilde u$ is a rescaling of $u$, it is also $s$-harmonic in $D_1$, and thus $(-\Delta)^s\widetilde u=F_{\widetilde u}$ in $\mathscr{D}^\prime(D_1)$. Next, we keep arguing as in Step 2, and 
 we consider the bounded map $\widetilde u_0:=\zeta \widetilde u$. Applying  Lemma \ref{regC1alphafctGreg} again, we argue as in Step 2 to infer that   $(-\Delta)^s\widetilde u_0=\widetilde F_{0}$ in $H^{-s}(D_1;\R^d)$, for a right-hand side $\widetilde F_0\in C^{1,\alpha}(D_1)$ for every $\alpha\in(0,2s)$. By the results in \cite{RosSer}, we have $\widetilde u_0\in C^{1,\alpha}(D_{1/2})$ for every $\alpha\in(0,4s)$ if $4s<1$, and $\widetilde u_0\in C^{1,\alpha}(D_{1/2})$ for every $\alpha\in(0,1)$ if $4s\geq 1$. Once again, since $\widetilde u_0=\widetilde u$ in $D_{1/2}$,  
 %$(-\Delta)^s(\widetilde u -\widetilde u_0)=0$ in $\mathscr{D}^\prime(D_1)$, so that $\widetilde u-\widetilde u_0$ is smooth in $D_{1/2}$. 
 we have $\widetilde u\in C^{1,\alpha}(D_{1/2})$ for every $\alpha\in(0,4s)$ if $4s<1$, and $\widetilde u\in C^{1,\alpha}(D_{1/2})$ for every $\alpha\in(0,1)$ if $4s\geq 1$.

In the case $s\in[ 1/4,1/2)$, we have thus proved that $u\in C^{1,\alpha}(D_{1/12})$ for every $\alpha\in(0,1)$. Hence it remains to consider the case $s<1/4$. In that case, we repeat the preceding argument considering the rescaling $\widehat u(x):=\widetilde u(x/6)$. Following the same notation as above, Lemma \ref{regC1alphafctGreg}  tells us that $G_{\widehat u}$ belongs to  $C^{1,\alpha}(D_1)$ for every $\alpha\in(0,4s)$, and hence also $F_{\widehat u}$. Then, applying the results of \cite{RosSer} to $\widehat u_0$, we conclude that  $\widehat u\in C^{1,\alpha}(D_{1/2})$ for every $\alpha\in(0,6s)$ if $6s<1$, and $\widehat u\in C^{1,\alpha}(D_{1/12})$ for every $\alpha\in(0,1)$ if $6s\geq 1$. Therefore, if $s\geq 1/6$, then $u\in C^{1,\alpha}(D_{1/72})$ for every $\alpha\in(0,1)$, which is the announced regularity. On the other hand, if $s\in (0,1/6)$, then we repeat the argument. It is now clear that repeating a finite number $\ell$ of times this argument, one reaches the conclusion that $u\in C^{1,\alpha}(D_{(6)^{-\ell}/2})$ for every $\alpha\in(0,1)$, and $\ell$ is essentially the integer part of $1/2s$. 
\end{proof}

Before closing this subsection, we provide an analogue of Lemma \ref{regC1alphafctGreg} in the case $s\in[1/2,1)$. 

\begin{lemma}\label{regC1alphafctGregbis}
Assume that $s\in[1/2,1)$. Let $f:D_3\to \R^d$ and $g:D_3\to \R^d$  be two $C^1$-functions,   
 and $\zeta:D_1\to [0,1]$ a measurable function. If $f$ and $g$ belongs to $C^{1,\alpha}(D_3)$ for every $\alpha\in(0,1)$, then the function $G:D_1\to\R$ given by \eqref{defGfunctionreg} belongs to 
$C^{1,\alpha^\prime}(D_1)$ for every $\alpha^\prime\in(0,2-2s)$, and \eqref{formpartialderGfctreg} holds. 
\end{lemma}

\begin{proof}
We proceed as in the proof of Lemma \ref{regC1alphafctGreg} using notation \eqref{notgamGfct}. We fix an arbitrary point $x\in D_1$ and we want to show that $G$ admits a partial derivative $\partial_i G$ at $x$. 
For $t>0$ small, we have
\begin{multline*}
\Gamma(x+t e_i,y)- \Gamma(x,y)=\Big(\int_0^t\big(\partial_if(x+y+\rho e_i)-\partial_if(x+\rho e_i)\big)\,\de\rho
%\big(f(x+y+te_i)-f(x+y)\big)-\big(f(x+te_i)-f(x)\big)
\Big)\cdot\big(g(x+y+te_i)-g(x+te_i)\big)\\
+\big(f(x+y)-f(x)\big)\cdot \Big(\int_0^t\big(\partial_ig(x+y+\rho e_i)-\partial_i g(x+\rho e_i)\big)\,\de\rho
\Big)
\end{multline*}
for every $y\in D_1$. Fixing an exponent $\alpha\in(2s-1,1)$, we deduce that 
$$\big|\Gamma(x+t e_i,y)- \Gamma(x,y)\big|\leq C_{f,g,\alpha} |y|^{1+\alpha}t \quad\forall y\in D_1\,.$$
Consequently, 
$$\frac{|\Gamma(x+t e_i,y)- \Gamma(x,y)|}{|y|^{n+2s} t}\leq C_{f,g,\alpha} |y|^{n+2s-1-\alpha} \in L^1(D_1)\,.$$
As in the proof of Lemma \ref{regC1alphafctGreg}, it now follows that $G$ admits a partial derivative $\partial_i G$ at $x$ given by \eqref{formpartialderGfctreg}, and the H\"older continuity of the partial  derivatives of $G$  is a consequence of Lemma~\ref{keybootstraplemma}. 
\end{proof}

\subsection{H\"older continuity of higher order derivatives}

\begin{proposition}\label{Ckalphareg}
Let $u\in \widehat H^s(D_{3};\mathbb{S}^{d-1})$ be a weakly $s$-harmonic map in $D_{3}$. If $u\in C^{k,\alpha}(D_3)$ for some integer $k\geq 1$ and every $\alpha\in(0,1)$, then $u\in C^{k+1,\alpha}(D_{r_*})$ for every $\alpha\in(0,1)$, where the radius $r_*\in(0,1/2)$ is given by Proposition \ref{C1alphareg}. 
%\begin{enumerate}
%\item for $s\in(1/2,1)$, $u\in C^{1,\alpha}(D_{1/2})$ for every $\alpha<2s-1$; 
%\item for $s\in(0,1/2)$, $u\in C^{1,\alpha}(D_{1/2})$ for every $\alpha<2s$. 
%\end{enumerate}
\end{proposition}

\begin{proof}
We proceed as in Step 1 in the proof of Proposition \ref{C1alphareg}, and we consider the function $G_u:D_1\to\R$ given by \eqref{deffctGu}. We claim that $G_u\in C^{k,\alpha}(D_1)$ for every $\alpha\in(0,1)$ if $s\in(0,1/2)$, and that   $G_u\in C^{k,\alpha}(D_1)$ for every $\alpha\in(0,2-2s)$ if $s\in[1/2,1)$, together with the formula 
\begin{equation}\label{multiderivfctGu}
\partial^{\beta}G_u(x)=\sum_{\nu\leq \beta} {\beta\choose \nu}\int_{D_1}\frac{(\partial^\nu u(x+y)-\partial^\nu u(x))\cdot(\partial^{\beta-\nu}u(x+y)-\partial^{\beta-\nu}u(x))}{|y|^{n+2s}}\zeta(y)\,\de y
\end{equation}
for every multi-index $\beta\in\mathbb{N}^n$ of length $|\beta|\leq k$. To prove this claim, we distinguish the case $s\in(0,1/2)$ from the case $s\in[1/2,1)$. 
\vskip3pt

\noindent{\it Case $s\in(0,1/2)$.} We proceed by induction. First notice that the fact that $G_u\in C^{1,\alpha}(D_1)$ for every $\alpha\in(0,1)$ follows from Lemma \ref{regC1alphafctGreg}, as well as 
\eqref{multiderivfctGu} with $|\beta|=1$. Next we assume that $G_u\in C^{\ell,\alpha}(D_1)$ for every $\alpha\in(0,1)$ for some integer $\ell<k$, and that \eqref{multiderivfctGu} holds for every multi-index $\beta$ satisfying $|\beta|=\ell$.  Applying Lemma \ref{regC1alphafctGreg} to each term in the right hand side of \eqref{multiderivfctGu}, we infer that $\partial^\beta G_u\in C^{1,\alpha}(D_1)$ for every $\alpha\in(0,1)$ and each $\beta$ satisfying $|\beta|=\ell$, and that \eqref{multiderivfctGu} holds for multi-indices $\beta^\prime$ in place of $\beta$ of length $|\beta^\prime|=|\beta|+1$. The claim is thus proved for $s\in(0,1/2)$. 
\vskip3pt

\noindent{\it Case $s\in[1/2,1)$.} We proceed exactly as in the previous case but using Lemma \ref{regC1alphafctGregbis} instead of Lemma  \ref{regC1alphafctGreg}. 
\vskip3pt

Now we consider the function $H_u:D_1\to\R$ given by \eqref{deffctHu} which clearly belongs to $C^{k,\alpha}(D_1)$ for every $\alpha\in(0,1)$ by our assumption on $u$. Consequently, the map $F_u:D_1\to \R^d$ belongs to $C^{k,\alpha}(D_1)$ for every $\alpha\in(0,1)$ if $s\in(0,1/2)$, and to $C^{k,\alpha}(D_1)$ for every $\alpha\in(0,2-2s)$ if $s\in[1/2,1)$. By the results in \cite{RosSer} (together with \ref{regC1alphafctGreg} and Lemma \ref{regC1alphafctGregbis}), it implies that the map $u_0:=\zeta u$ as defined in  Step 2, proof of Proposition \ref{C1alphareg}, belongs to $C^{k+1,\alpha}(D_{1/2})$
for every $\alpha\in(0,2s)$ if $s\in(0,1/2)$, and to $C^{k+1,\alpha}(D_{1/2})$
for every $\alpha\in(0,1)$ if $s\in[1/2,1)$. Since $u_0=u$ in $D_{1/2}$, 
%As in the proof of Proposition \ref{C1alphareg}, it implies that $u\in C^{k+1,\alpha}(D_{3/8})$ for every $\alpha\in(0,2s)$ if $s\in(0,1/2)$, and  $u\in C^{k+1,\alpha}(D_{3/8})$ for every $\alpha\in(0,1)$ if $s\in[1/2,1)$. 
the proof is thus complete for $s\in[1/2,1)$. In the case $s\in(0,1/2)$, we argue as in the proof of Proposition~\ref{C1alphareg}, Step 3, applying (inductively) Lemma \ref{regC1alphafctGreg} to formula \eqref{multiderivfctGu} with $|\beta|=k$. It leads to the fact that $u\in C^{k+1,\alpha}(D_{r_*})$ for every $\alpha\in(0,1)$, and hence concludes the proof. 
\end{proof}

%%%%%%%%%%%%%%%%%%%%%%%%%%%%%%%%%%%%%%%%%%%%%%%%%%%%%%%
%%%%%%%%%%%%%%%%%%%%%%%%%%%%%%%%%%%%%%%%%%%%%%%%%%%%%%%
   								       						%%%%%%%%%%%%%%%%%%%
\section{Partial regularity for stationary  and minimizing $s$-harmonic maps}\label{partialreg} %%%%%%%%%
								 						%%%%%%%%%%%%%%%%%%%
%%%%%%%%%%%%%%%%%%%%%%%%%%%%%%%%%%%%%%%%%%%%%%%%%%%%%%%
%%%%%%%%%%%%%%%%%%%%%%%%%%%%%%%%%%%%%%%%%%%%%%%%%%%%%%%

In this section, we complete the proof of Theorems \ref{mainthm1}, \ref{mainthm2}, and \ref{mainthm3}. For $n>2s$, we need to prove compactness of stationary / minimizing $s$-harmonic map to apply Federer's dimension reduction principle. This is the object of the first subsection.

\subsection{Compactness properties of $s$-harmonic maps}\label{subsectcompact}

\begin{theorem}\label{maincompactthm}
Assume that $s\in(0,1)\setminus\{1/2\}$ and $n>2s$. Let $\Omega\subset\R^n$ be a bounded open set. Let $\{u_k\}\subset \widehat H^s(\Omega;\mathbb{S}^{d-1})$ be a sequence of stationary weakly $s$-harmonic maps in $\Omega$. Assume that $\sup_k\mathcal{E}_s(u_k,\Omega)<+\infty$, and that $u_k\to u$ a.e. in $\R^n$. Then $u\in \widehat H^s(\Omega;\mathbb{S}^{d-1})$, $u_k\rightharpoonup u$ weakly in $\widehat H^s(\Omega;\R^d)$, and $u$ is a stationary weakly $s$-harmonic map in $\Omega$. In addition, for every open subset $\omega\subset \Omega$ and every  bounded admissible open set $G\subset \R^{n+1}_+$ satisfying $\overline\omega\subset \Omega$ and $\overline{\partial^0G}\subset \Omega$, 
\begin{enumerate}
\item[(i)] $u_k\to u$ strongly in $\widehat H^s(\omega;\R^d)$;
\item[(ii)] $u_k^\e\to u^\e$ strongly in $H^1(G;\R^d,|z^a\de{\bf x})$.
\end{enumerate}
\end{theorem}

\begin{theorem}\label{compactthmmins}
Assume that $s\in(0,1/2)$. In addition to Theorem \ref{maincompactthm}, if each $u_k$ is assumed to be a minimizing $s$-harmonic map in $\Omega$, then the limit $u$ is a minimizing $s$-harmonic map in $\Omega$.  
\end{theorem}

\begin{theorem}\label{compactthmmin1/2}
 Let $\Omega\subset\R^n$ be a bounded open set and $\{u_k\}\subset \widehat H^{1/2}(\Omega;\mathbb{S}^{d-1})$  be a sequence of minimizing $1/2$-harmonic maps in $\Omega$. Assume that $\sup_k\mathcal{E}_{\frac{1}{2}}(u_k,\Omega)<+\infty$, and that $u_k\to u$ a.e. in $\R^n$. Then the conclusion of Theorem \ref{maincompactthm} holds and the limit $u$ is a minimizing $1/2$-harmonic map in $\Omega$. 
\end{theorem}

\begin{remark}
In the case $s\in(1/2,1)$, we do not know if minimality of the sequence $\{u_k\}$ implies  minimality of the limit. We believe this is indeed the case, but we won't need this fact. 
\end{remark}

\begin{remark}
In the case $n=1$ and $s\in(1/2,1)$, sequences of (arbitrary) weakly $s$-harmonic maps with uniformly bounded energy are relatively compact, i.e., the conclusion of Theorem~\ref{maincompactthm} holds. This fact is a consequence of the Lipschitz estimate established in Theorem \ref{thmepsregLip} together with Remark \ref{remarlepsholdregsubcritic}. Since we shall not need this, we leave the details to the reader. 
\end{remark}

\begin{remark}
In the case $s=1/2$, sequences of (stationary or not)  $1/2$-harmonic maps are not compact in general, see e.g. \cite{DaLi1,MPis,MirPis,MilPeg}. The prototypical example  is the following sequence of smooth $1/2$-harmonic maps from $\R^n$ into $\mathbb{S}^1\subset \mathbb{C}$ given by  
$$u_k(x)=u_k(x_1):=\frac{kx_1-i}{kx_1+i}\,,\quad k\in\mathbb{N}\,,$$
%Their harmonic extensions are given by 
%$$u^\e_k({\bf x})=\frac{k(x_1+iz)-i}{k(x_1+iz)+i} \,,$$ 
%where ${\bf x}=(x_1,\ldots,x_n,z)\in\R^{n+1}_+$.  From this explicit formula, one obtains that 
%$${\bf E}_{1/2}(u^\e_k,B_r^+)\leq Cr^{n-1} \quad \forall r>0\,.$$
which is converging  weakly but not strongly to the constant map $1$ in $\widehat H^{1/2}(D_r)$ for every $r>0$. (Recall that $u_k$ being smooth, it is stationary, see Remark \ref{remsmothhimplstat}.) 
\end{remark}

%
%\begin{Example}
%$s=1/2$ and $n=1$ 
%\end{Example}
%
%\begin{Example}
%$s=1/2$ and $n=2$ 
%\end{Example}

\begin{proof}[Proof of Theorem \ref{maincompactthm}]
{\it Step 1.} We fix two arbitrary admissible bounded open sets $G,G^\prime\subset \R^{n+1}_+$ such that $\overline G\subset G^\prime\cup\partial^0G^\prime$ and satisfying $\overline{\partial^0G^\prime}\subset \Omega$. Since $u_k\to u$ a.e. in $\R^n$ and $|u_k|=1$,  we first deduce that $|u|=1$ and $u_k\to u$ strongly in $L^2_{\rm loc}(\R^2;\R^d)$. It then follows from our assumption that $\{u_k\}$ is bounded in  $\widehat H^s(\Omega;\R^d)$. Next we derive from Remark \ref{remweakcvHhat} that $u\in \widehat H^s(\Omega;\mathbb{S}^{d-1})$ and $u_k\rightharpoonup u$ weakly in 
$\widehat H^s(\Omega;\R^d)$. In view of Corollary \ref{contextHsH1}, $u_k^\e\rightharpoonup u^\e$ weakly in $H^1(G^\prime;\R^d,|z|^a\de{\bf x})$. Since $|u_k|\leq 1$, we have $u_k^\e({\bf x})\to u^\e({\bf x})$ for every ${\bf x}\in G^\prime$ by dominated convergence. In turn, we have $|u^\e_k-u^\e|\leq 2$, and it follows by dominated convergence again that $u_k^\e\to u^\e$ strongly in $L^2(G^\prime;\R^d,|z|^a\de{\bf x})$. Recalling that ${\rm div}(z^a\nabla u^\e_k)=0$ in $G^\prime$, we infer from standard elliptic regularity that $u^\e_k\to u^\e$ in $C^1_{\rm loc} (G^\prime)$. In particular, 
\begin{equation}\label{strcvloccmpct}
u^\e_k\to u^\e\quad\text{strongly in $H^1_{\rm loc}(G^\prime;\R^d)$}\,. 
\end{equation}

We aim to show that  $u_k^\e\to u^\e$ strongly in $H^1(G;\R^d,|z|^a\de{\bf x})$. To prove this strong convergence, we consider the finite measures on $G^\prime\cup\partial^0G^\prime$ given by 
$$\mu_k:= \frac{\boldsymbol{\delta}_s}{2}z^a|\nabla u_k^\e|^2\mathscr{L}^{n+1}\res G^\prime\,.$$
Since $\sup_k\mu_k(G^\prime\cup\partial^0G^\prime)<+\infty$, we can find a further (not relabeled) subsequence such that 
\begin{equation}\label{weakcvmuk}
\mu_k\rightharpoonup   \frac{\boldsymbol{\delta}_s}{2}z^a|\nabla u^\e|^2\mathscr{L}^{n+1}\res G^\prime+\mu_{\rm sing}\quad\text{as $k\to\infty$}\,,
\end{equation}
weakly* as Radon measures on $G^\prime\cup\partial^0G^\prime$ for some finite nonnegative measure $\mu_{\rm sing}$. In view of \eqref{strcvloccmpct}, the defect measure $\mu_{\rm sing}$ is supported 
by $\partial^0G^\prime$. 

Since $u_k$ is stationary in $\Omega$, it satisfies the monotonicity formula in Proposition \ref{monotformula}, and thus 
\begin{equation}\label{monotmuk}
\mu_k(B_\rho({\bf x}))\leq  \mu_k(B_r({\bf x}))
\end{equation}
for every ${\bf x}\in\partial^0G^\prime$ and $0<\rho<r<{\rm dist}({\bf x}, \partial^+G^\prime)$. From the weak* convergence of $\mu_k$ towards $\mu$, we then infer that
$$\mu(B_\rho({\bf x}))\leq  \mu(B_r({\bf x}))$$
for every ${\bf x}\in\partial^0G^\prime$ and $0<\rho<r<{\rm dist}({\bf x}, \partial^+G^\prime)$. As a consequence, the $(n-2s)$-dimensional density 
$$\Theta^{n-2s}(\mu,{\bf x}):=\lim_{r\to 0}\frac{\mu(B_r({\bf x}))}{r^{n-2s}}  $$
exists and is finite at every point ${\bf x}\in\partial^0G^\prime$. More precisely, \eqref{monotmuk} implies that 
$$ \Theta^{n-2s}(\mu,{\bf x})\leq \big({\rm dist}({\bf x}, \partial^+G^\prime)\big)^{2s-n}\sup_k{\bf E}_s(u_k,G^\prime)<+\infty\quad \forall {\bf x}\in\partial^0G^\prime\,.$$
We now consider the ``concentration set'' 
$$\Sigma:=\Big\{{\bf x}\in \partial^0G^\prime: \inf_r\big\{\liminf_{k\to\infty} r^{2s-n}\mu_k(B_r({\bf x})) : 0<r<{\rm dist}({\bf x}, \partial^+G^\prime)\big\}\geq \boldsymbol{\eps}_1\Big\}\,, $$
where the constant $\boldsymbol{\eps}_1>0$ is given by Corollary \ref{coroepsreghold}. From the monotonicity of $\mu_k$ and $\mu$ together with \eqref{weakcvmuk}, we deduce that 
\begin{multline*}
\Sigma=\Big\{{\bf x}\in \partial^0G^\prime: \lim_{r\to 0}\liminf_{k\to\infty} r^{2s-n}\mu_k(B_r({\bf x})) \geq  \boldsymbol{\eps}_1 \Big\}\\
= \Big\{{\bf x}\in \partial^0G^\prime: \lim_{r\to 0} r^{2s-n}\mu(B_r({\bf x})) \geq  \boldsymbol{\eps}_1 \Big\}\,,
\end{multline*}
that is 
$$\Sigma= \Big\{{\bf x}\in \partial^0G^\prime:\Theta^{n-2s}(\mu,{\bf x})\geq  \boldsymbol{\eps}_1 \Big\}\,.$$
Observing that ${\bf x}\in \partial^0G^\prime\mapsto \Theta^{n-2s}(\mu,{\bf x})$ is upper semi-continuous, the set $\Sigma$ is a relatively closed subset of $\partial^0G^\prime$. 

We claim that ${\rm spt}(\mu_{\rm sing})\subset\Sigma$. To prove this inclusion, we fix an arbitrary point ${\bf x}_0=(x_0,0)\in\partial^0 G^\prime\setminus\Sigma$. Then we can find a  radius $0<r<{\rm dist}({\bf x}_0, \partial^+G^\prime)$ such that $r^{2s-n}\mu(B_r({\bf x}_0))< \boldsymbol{\eps}_1$ and  $\mu(\partial B_r({\bf x}_0))=0$. By \eqref{weakcvmuk} and our choice of $r$, we have $\lim_k\mu_k(B_r({\bf x}_0))=\mu(B_r({\bf x}_0))$. Therefore,  $r^{2s-n}\mu_k(B_r({\bf x}_0))< \boldsymbol{\eps}_1$ for $k$ large enough, and we derive from Theorem \ref{thmepsregLip} that for $k$ large enough, $u_k$ is bounded in 
$C^{0,1}(D_{\boldsymbol{\kappa}_2r}(x_0))$ (and $u\in C^{0,1}(D_{\boldsymbol{\kappa}_2r}(x_0))$), where the constant $\boldsymbol{\kappa}_2\in(0,1)$ only depends on $n$ and $s$. It then follows by dominated convergence that 
$$[u_k-u]^2_{H^s(D_{\boldsymbol{\kappa}_2r}(x_0))} \mathop{\longrightarrow}\limits_{k\to\infty} 0\,.$$
Setting $w_k:=u_k-u$, we now estimate
\begin{multline*}
\mathcal{E}_s(w_k,D_{2\boldsymbol{\kappa}_2r/3}(x_0))\leq C\Big([u_k-u]^2_{H^s(D_{\boldsymbol{\kappa}_2r}(x_0))}\\+\iint_{D_{2\boldsymbol{\kappa}_2r/3}(x_0)\times D^c_{\boldsymbol{\kappa}_2r}(x_0)}
\frac{|w_k(x)-w_k(y)|^2}{|x-y|^{n+2s}}\,\de x\de y\Big) \,.
\end{multline*}
Since $|w_k|\leq 2$ and $w_k\to 0$ a.e. in $\R^n$, by dominated convergence we have 
\begin{equation}\label{pipicacaprout}
\iint_{D_{2\boldsymbol{\kappa}_2r/3}(x_0)\times D^c_{\boldsymbol{\kappa}_2r}(x_0)}
\frac{|w_k(x)-w_k(y)|^2}{|x-y|^{n+2s}}\,\de x\de y  \mathop{\longrightarrow}\limits_{k\to\infty} 0\,.
\end{equation}
Hence $\mathcal{E}_s(w_k,D_{2\boldsymbol{\kappa}_2r/3}(x_0))\to 0$, and it follows from Lemma \ref{hatH1/2toH1} that 
$${\bf E}_s\big(u_k^\e-u^\e,B^+_{\boldsymbol{\kappa}_2r/3}({\bf x}_0)\big)\leq C  \mathcal{E}_s(u_k-u,D_{2\boldsymbol{\kappa}_2r/3}(x_0))\to 0\,.$$
Hence, $u_k^\e\to u^\e$ strongly in $H^1(B^+_{\boldsymbol{\kappa}_2r/3}({\bf x}_0),|z|^a\de{\bf x})$, and thus $\mu_{\rm sing}(B_{\boldsymbol{\kappa}_2r/3}({\bf x}_0))=0$. This shows that ${\bf x}_0\not\in {\rm spt}(\mu_{\rm sing})$, and the claim is proved. 
\vskip3pt

Next we claim that $\mu(\Sigma)=0$. Indeed, assume by contradiction that $\mu(\Sigma)>0$. Then the density $\Theta^{n-2s}(\mu,{\bf x})$ exists, it is positive (greater than $\boldsymbol{\eps}_1$) and finite, at every point ${\bf x}\in \Sigma$. By Marstrand's theorem (see e.g. \cite[Theorem 14.10]{Matti}), it implies that $n-2s$ is an integer, a contradiction. 

Knowing that $\mu(\Sigma)=0$, we now deduce that $\mu_{\rm sing}(\Sigma)=0$. But $\mu_{\rm sing}$ being supported by $\Sigma$, it implies that $\mu_{\rm sing}\equiv0$. As a  consequence, 
${\bf E}_s(u^\e_k,G)\to  {\bf E}_s(u^\e,G)$, which combined with the weak convergence in $H^1(G;\R^d,|z|^a\de{\bf x})$ implies that ${\bf E}_s(u^\e_k-u^\e,G)\to  0$. We have thus proved that $u_k^\e\to u^\e$ strongly in $H^1(G;\R^d,|z|^a\de{\bf x})$. 
\vskip5pt

\noindent{\it Step 2.} We consider in this step an open subset $\omega\subset\Omega$ such that $\overline\omega\subset \Omega$, and our goal is to prove that $u_k\to u$ 
strongly in $\widehat H^s(\omega;\R^d)$. 
%We fix a further open subset $\omega^\prime\subset \Omega$ such that $\overline\omega\subset\omega^\prime$ and  $\overline{\omega^\prime}\subset\Omega$. 
Set $\delta:=\frac{1}{8}{\rm dist}(\omega,\Omega^c)$, and consider a finite covering of $\omega$ by balls $(D_{\delta}(x_i))_{i\in I}$ with $x_i\in\overline{\omega}$. By Lemma \ref{HsregtraceH1weight} and Step 1, we have for each $i\in I$, 
\begin{equation}\label{pipicacaprout1}
[u_k-u]^2_{H^s(D_{2\delta}(x_i))}\leq C{\bf E}_s(u^\e_k-u^\e,B^+_{4\delta}({\bf x}_i)) \mathop{\longrightarrow}\limits_{k\to\infty}0\,,
\end{equation}
where ${\bf x}_i:=(x_i,0)$. Writing again $w_k:=u_k-u$, we now estimate 
\begin{align}
\nonumber \mathcal{E}_s(w_k,\omega)&\leq C\iint_{\omega\times\R^n}\frac{|w_k(x)-w_k(y)|^2}{|x-y|^{n+2s}}\,\de x\de y\\
\nonumber&\leq C\sum_{i\in I} \iint_{D_{\delta}(x_i)\times\R^n}\frac{|w_k(x)-w_k(y)|^2}{|x-y|^{n+2s}}\,\de x\de y\,\\
\label{pipicacaprout2}&\leq C\sum_{i\in I} \Big( [w_k]^2_{H^s(D_{2\delta}(x_i))}+ \iint_{D_{\delta}(x_i)\times D^c_{2\delta}(x_i)}\frac{|w_k(x)-w_k(y)|^2}{|x-y|^{n+2s}}\,\de x\de y \Big)\,.
\end{align}
As in \eqref{pipicacaprout}, by dominated convergence we have 
\begin{equation}\label{pipicacaprout3}
\iint_{D_{\delta}(x_i)\times D^c_{2\delta}(x_i)}\frac{|w_k(x)-w_k(y)|^2}{|x-y|^{n+2s}}\,\de x\de y  \mathop{\longrightarrow}\limits_{k\to\infty}0 \quad\forall i\in I\,.
\end{equation}
Combining \eqref{pipicacaprout1}, \eqref{pipicacaprout2}, and \eqref{pipicacaprout3} leads to $\mathcal{E}_s(w_k,\omega)\to 0$, and thus $u_k\to u$ 
strongly in $\widehat H^s(\omega;\R^d)$. 
\vskip5pt

\noindent{\it Step 3.} Our aim in this step is to show that $u$ is a weakly $s$-harmonic map in $\Omega$, i.e., $u$ satisfies equation \eqref{ELeqorig}, or equivalently \eqref{ELeqsgrad}, by Proposition \ref{ELeqprop}. To this purpose, we fix an arbitrary $\varphi\in\mathscr{D}(\Omega;\R^d)$, and we choose an open subset $\omega\subset \Omega$ such that ${\rm spt}(\varphi)\subset \omega$ and $\overline\omega\subset\Omega$. Writing again $w_k:=u_k-u$, we have proved in Step 2 that 
$\mathcal{E}_s(w_k,\omega)\to 0$. 

Recalling our notations from Subsection \ref{sectoperandcompcomp}, we observe that 
$$|{\rm d}_s u_k|^2-|{\rm d}_s u|^2=|{\rm d}_sw_k|^2 +2{\rm d}_sw_k\odot{\rm d}_s u\,,$$
and then estimate
\begin{align*}
\big\||{\rm d}_s u_k|^2-|{\rm d}_s u|^2\big\|_{L^1(\omega)}& \leq \big\| |{\rm d}_sw_k|^2\big\|_{L^1(\omega)}+2 \big\| {\rm d}_sw_k\odot{\rm d}_s u\big\|_{L^1(\omega)}\\
&\leq 2 \mathcal{E}_s(w_k,\omega) + 2 \|{\rm d}_sw_k\|_{L^2_{\rm od}(\omega)}\|{\rm d}_su\|_{L^2_{\rm od}(\omega)}\\
&\leq 2 \mathcal{E}_s(w_k,\omega) + 2\sqrt{2} \|{\rm d}_su\|_{L^2_{\rm od}(\omega)}\sqrt{\mathcal{E}_s(w_k,\omega)}\,.
\end{align*}
Therefore $|{\rm d}_s u_k|^2\to |{\rm d}_s u|^2$ in $L^1(\omega)$, and we can find a further (not relabeled) subsequence and $h\in L^1(\omega)$ such that 
$$|{\rm d}_su_k|^2(x)\to |{\rm d}_su|^2(x)\text{ for a.e. $x\in\omega$, and } |{\rm d}_su_k|^2(x)\leq h(x)\text{ for a.e. $x\in\omega$}\,.$$
Since $|u_k|=1$ and $u_k\to u$ a.e. in $\omega$, it follows by dominated convergence that $|{\rm d}_su_k|^2u_k\to |{\rm d}_su|^2u$ in $L^1(\omega)$. Consequently, 
$$\int_\Omega|{\rm d}_su_k|^2u_k\cdot\varphi\,\de x\mathop{\longrightarrow}\limits_{k\to\infty}\int_\Omega|{\rm d}_su|^2u\cdot\varphi\,\de x \,. $$
On the other hand, the weak convergence of $u_k$ to $u$ in $\widehat H^s(\Omega;\R^d)$ implies that $\big\langle (-\Delta)^su_k,\varphi\big\rangle_\Omega$ converges to  
$\big\langle (-\Delta)^su,\varphi\big\rangle_\Omega$. Hence, 
$$ \big\langle (-\Delta)^su,\varphi\big\rangle_\Omega=\lim_{k\to\infty} \big\langle (-\Delta)^su_k,\varphi\big\rangle_\Omega=\lim_{k\to\infty} \int_\Omega|{\rm d}_su_k|^2u_k\cdot\varphi\,\de x=\int_\Omega|{\rm d}_su|^2u\cdot\varphi\,\de x \,,$$
so that $u$ is indeed weakly $s$-harmonic in $\Omega$ (see \eqref{ELeqsgrad}). 
\vskip5pt

\noindent{\it Step 4.} It now only remains to prove that $u$ is stationary  in $\Omega$. This is in fact an easy consequence of the strong convergence of $u^\e$ established in Step 1. Indeed, let us fix an arbitrary vector field 
$X\in C^1(\R^n;\R^n)$ compactly supported in $\Omega$. Combining the strong convergence of $u^\e_k$ established in Step 1 together with the representation of the first variation $\delta\mathcal{E}_s$ stated in Proposition \ref{represfirstvar}, we obtain that $\delta\mathcal{E}_s(u_k,\Omega)[X]\to \delta\mathcal{E}_s(u,\Omega)[X]$, whence $\delta\mathcal{E}_s(u,\Omega)=0$.
\end{proof}

\begin{proof}[Proof of Theorem \ref{compactthmmins}]
In view of Remark \ref{implicminstat} and Theorem \ref{maincompactthm}, it only remains  to prove that the limiting map $u$ is a minimizing $s$-harmonic map in $\Omega$. We follow here the argument in \cite[Theorem 4.1]{MSY}. 
%First notice that Remark \ref{remweakcvHhat} tells us 
%$$\mathcal{E}_s(u,\Omega)\leq \liminf_{k\to\infty} \mathcal{E}_s(u_k,\Omega)\,.$$

Let us now consider an arbitrary $\widetilde u\in \widehat H^s(\Omega;\mathbb{S}^{d-1})$ such that ${\rm spt}(u-\widetilde u)\subset\Omega$. We select an open subset $\omega\subset\Omega$ with Lipschitz boundary such that ${\rm spt}(u-\widetilde u)\subset\omega$ and $\overline\omega\subset\Omega$. Define 
$$\widetilde u_k(x):=\begin{cases}
\widetilde u(x) & \text{if $x\in\omega$}\,,\\
u_k(x) & \text{otherwise}\,.
\end{cases} $$
Since $s\in(0,1/2)$ and $\partial\omega$ is Lipschitz regular, it turns out that $\widetilde u_k\in\widehat H^s(\Omega;\mathbb{S}^{d-1})$ (see e.g. \cite[Section 2.1]{MSK}), and ${\rm spt}(u_k-\widetilde u_k)\subset \Omega$. By minimality of $u_k$, we have $\mathcal{E}_s(u_k,\Omega)\leq \mathcal{E}_s(\widetilde u_k,\Omega)$. Since $\widetilde u_k=u_k$ in $\R^n\setminus\omega$, it reduces to 
$$ \mathcal{E}_s(u_k,\omega)\leq \mathcal{E}_s(\widetilde u_k,\omega)=\frac{\gamma_{n,s}}{4}\iint_{\omega\times\omega}\frac{|\widetilde u(x)-\widetilde u(y)|^2}{|x-y|^{n+2s}}\,\de x\de y+\frac{\gamma_{n,s}}{2}\iint_{\omega\times\omega^c}\frac{|\widetilde u(x)-u_k(y)|^2}{|x-y|^{n+2s}}\,\de x\de y\,.$$
On the other hand, 
$$\frac{|\widetilde u(x)- u_k(y)|^2}{|x-y|^{n+2s}}\leq \frac{4}{|x-y|^{n+2s}}\in L^1(\omega\times\omega^c)\,, $$ 
since $\omega$ has Lipschitz boundary. Hence,  $\mathcal{E}_s(\widetilde u_k,\omega)\to \mathcal{E}_s(\widetilde u,\omega)$ by dominated convergence and the fact that $\widetilde u=u$ in $\R^n\setminus\omega$. By Fatou's Lemma, we have $\liminf_k\mathcal{E}_s(u_k,\omega)\geq \mathcal{E}_s(u,\omega)$, and we reach the conclusion that  $ \mathcal{E}_s(u,\omega)\leq \mathcal{E}_s(\widetilde u,\omega)$. Once again, the fact that $\widetilde u=u$ in $\R^n\setminus\omega$ then implies that $\mathcal{E}_s(u,\Omega)\leq \mathcal{E}_s(\widetilde u,\Omega)$. By arbitrariness of $\widetilde u$, we conclude that $u$ is indeed a minimizing $s$-harmonic map in $\Omega$. 
\end{proof}

We now close this subsection with an easy consequence of Theorem \ref{maincompactthm} and Theorem \ref{compactthmmin1/2} in terms of the pointwise density function $\boldsymbol{\Xi}_s(u,\cdot)$ defined in \eqref{deflimitdens}.

\begin{corollary}\label{uscdensit}
Assume that $n>2s$. In addition to Theorem \ref{maincompactthm} and Theorem \ref{compactthmmin1/2}, if $\{x_k\}\subset\Omega$ is a sequence converging to $x_*\in\Omega$, then 
$$\limsup_{k\to\infty}\,\boldsymbol{\Xi}_s(u_k,x_k)\leq \boldsymbol{\Xi}_s(u,x_*) \,.$$ 
\end{corollary}
 
 \begin{proof}
 Without loss of generality, we can assume that $x_*=0$. Applying Corollary \ref{corolmonotform}, we obtain for $r>0$ small enough and $r_k:=|x_k|$, 
 \begin{equation}\label{concertcesoir}
 \boldsymbol{\Xi}_s(u_k,x_k)\leq \boldsymbol{\Theta}(u^\e_k,{\bf x}_k,r)\leq \frac{1}{r^{n-2s}}{\bf E}_s(u^\e_k,B^+_{r+r_k})\,, 
 \end{equation}
 where ${\bf x}_k:=(x_k,0)$. 
 By Theorem \ref{maincompactthm} (in the case $s\not=1/2$) and Theorem \ref{compactthmmin1/2} (in the case $s=1/2$), $u_k^\e\to u^\e$ strongly in $H^1(B^+_{2r},|z|^a\de{\bf x})$. Since $r_k\to0$, we deduce from \eqref{concertcesoir} that 
 $$\limsup_{k\to\infty} \, \boldsymbol{\Xi}_s(u_k,x_k)\leq \boldsymbol{\Theta}(u^\e,0,r)\,, $$
 and the conclusion follows letting $r\to 0$. 
 \end{proof}

\subsection{Tangent maps}\label{secttangmap}

We assume throughout this subsection that $s\in(0,1)$ and $n>2s$. We consider a bounded open set $\Omega\subset\R^n$ and a map $u\in\widehat H^s(\Omega;\mathbb{S}^{d-1})$ that we assume to be 
\begin{itemize}
\item a stationary weakly $s$-harmonic map in $\Omega$ for $s\not=1/2$; 
\item a minimizing $1/2$-harmonic map in $\Omega$ for $s=1/2$. 
\end{itemize}
We shall apply the results of Subsection \ref{subsectcompact} to define the so-called {\sl tangent maps} of $u$ at a given point. To this purpose, we fix a point of study $x_0\in\Omega$ and a reference radius $\rho_0>0$ such that $D_{2\rho_0}(x_0)\subset\Omega$. We introduce the rescaled function
$$u_{x_0,\rho}(x):=u(x_0+\rho x)\,,$$
and we observe that $(u_{x_0,\rho})^\e({\bf x})=u^\e({\bf x}_0+\rho {\bf x})=u_{x_0,\rho}^\e({\bf x})$ with ${\bf x}_0=(x_0,0)$. Rescaling variables, $u_{x_0,\rho}$ is a stationary weakly $s$-harmonic map in $(\Omega-x_0)/\rho$ for $s\not=1/2$, or a minimizing $1/2$-harmonic map in $(\Omega-x_0)/\rho$ for $s=1/2$. In addition, 
\begin{equation}\label{identrescaldens}
\boldsymbol{\Theta}_s(u^\e_{x_0,\rho},0,r)=\boldsymbol{\Theta}_s(u^\e,{\bf x}_0,\rho r) \quad\forall r\in(0,\rho_0/\rho]\,.
\end{equation}
This identity together with the monotonicity formula in Proposition \ref{monotformula} and  Lemma \ref{hatH1/2toH1} yields 
$$\boldsymbol{\Theta}_s(u^\e_{x_0,\rho},0,r)\leq \boldsymbol{\Theta}_s(u^\e,{\bf x}_0,\rho_0)
%\leq  2^{n-2s}\boldsymbol{\theta}_s(u,x_0,2\rho_0)
\leq C\rho_0^{2s-n}\mathcal{E}_s(u,\Omega) \quad\forall r\in(0,\rho_0/\rho]\,,$$
for a constant $C$ depending only on $n$ and $s$. 
In turn, Lemma \ref{HsregtraceH1weight} implies that 
$$[u_{x_0,\rho}]^2_{H^s(D_{2r})}\leq C\rho_0^{2s-n}r^{n-2s} \mathcal{E}_s(u,\Omega)\quad\forall r\in(0,\rho_0/(4\rho)]\,.$$
Using $|u_{x_0,\rho}|=1$, we can now estimate for $r\in(0,\rho_0/(4\rho)]$, 
$$\mathcal{E}_s(u_{x_0,\rho},D_r)\leq C\Big([u_{x_0,\rho}]^2_{H^s(D_{2r})}+\iint_{D_r\times D^c_{2r}}\frac{\de x\de y}{|x-y|^{n+2s}}\Big) \leq Cr^{n-2s}\big(\rho_0^{2s-n} \mathcal{E}_s(u,\Omega)+1\big)\,.$$
Given a sequence $\rho_k\to 0$, we deduce from the  above estimate that 
$$\limsup_{k\to\infty} \mathcal{E}_s(u_{x_0,\rho_k},D_r) <+\infty\quad \forall r>0\,.$$
Applying Theorem \ref{maincompactthm}, Theorem \ref{compactthmmins}, and Theorem \ref{compactthmmin1/2}, we can now find a subsequence $\{\rho^\prime_k\}$ and $\varphi\in H^s_{\rm loc}(\Rn;\mathbb{S}^{d-1})$ such that 
$$u_{x_0,\rho^\prime_k}\to \varphi\text{ strongly in $\widehat H^s(D_r)$, and }u^\e_{x_0,\rho^\prime_k}\to \varphi^\e\text{ strongly in $H^1(B_r^+,|z|^a\de{\bf x})$ for all $r>0$}\,, $$
where 
\begin{enumerate}
\item[(i)] {\it if $s\not=1/2$:}  $\varphi$ is a stationary weakly $s$-harmonic map in $D_r$ for all $r>0$; 
\item[(ii)] {\it if $s\leq1/2$ and $u$  minimizing:}  $\varphi$ is a minimizing $s$-harmonic map in $D_r$ for all $r>0$. 
\end{enumerate}

\begin{definition}
Every function $\varphi$ obtained by this process will be referred to as  {\sl a tangent map to $u$ at the point $x_0$}. The family of all tangent maps to $u$ at $x_0$ is denoted by $T_{x_0}(u)$. 
\end{definition}

We now  present some classical properties of tangent maps following e.g. \cite{Sim} or \cite[Section~6]{MSK}.

\begin{lemma}
If $\varphi\in T_{x_0}(u)$, then 
$$\boldsymbol{\Theta}_s(\varphi^\e,0,r)= \boldsymbol{\Xi}_s(\varphi,0)=\boldsymbol{\Xi}_s(u,x_0)\quad\forall r>0\,,$$
and $\varphi$ is positively $0$-homogeneous, i.e., $\varphi(\lambda x)=\varphi(x)$ for every $\lambda>0$ and $x\in\R^n$. In particular, 
\begin{equation}\label{homogdenstangmap}
 \boldsymbol{\Xi}_s(\varphi,\lambda x)= \boldsymbol{\Xi}_s(\varphi,x)\quad\text{for every $x\in\R^n\setminus\{0\}$ and $\lambda>0$}\,.
 \end{equation}
\end{lemma}

\begin{proof}
From the strong convergence of $u^\e_{x_0,\rho^\prime_k}$ to $\varphi^\e$ in $H^1(B_r^+,|z|^a\de{\bf x})$ and \eqref{identrescaldens}, we first deduce that 
$$\boldsymbol{\Theta}_s(\varphi^\e,0,r)=\lim_{k\to\infty}\boldsymbol{\Theta}_s(u^\e,{\bf x}_0,\rho_k^\prime r)=\boldsymbol{\Xi}_s(u,x_0)\quad\forall r>0\,.$$
Then, the constancy of $r\mapsto \boldsymbol{\Theta}_s(\varphi^\e,0,r)$ together with the monotonicity formula in Proposition~\ref{monotformula} implies that ${\bf x}\cdot\nabla\varphi^\e({\bf x})=0$ for every 
${\bf x}\in\R^{n+1}_+$. Hence, $\varphi^\e$ is positively $0$-homogeneous, and the homogeneity of $\varphi$ follows. As a consequence, for $x\in\R^n\setminus\{0\}$ and $\lambda>0$, 
$$\boldsymbol{\Theta}_s(\varphi^\e,\lambda {\bf x},r)= \boldsymbol{\Theta}_s(\varphi^\e,{\bf x},r/\lambda)\,,$$
where ${\bf x}:=(x,0)$. Letting now $r\to0$ yields \eqref{homogdenstangmap}. 
\end{proof}

\begin{lemma}\label{defSphi}
If $\varphi\in T_{x_0}(u)$, then 
$$\boldsymbol{\Xi}_s(\varphi,y)\leq \boldsymbol{\Xi}_s(\varphi,0) \quad\forall y\in\R^n\,.$$
In addition, the set
$$ S(\varphi):=\Big\{y\in\R^n: \boldsymbol{\Xi}_s(\varphi,y)= \boldsymbol{\Xi}_s(\varphi,0)\Big\}$$
is a linear subspace of $\R^n$, and $\varphi(x+y)=\varphi(x)$ for every $y\in S(\varphi)$ and every $x\in\R^n$. 
\end{lemma}

\begin{proof}
{\it Step 1.} By Corollary \ref{corolmonotform}, we have have for every $y\in\R^n$ and $\rho>0$, 
\begin{equation}\label{yaourtpres}
\boldsymbol{\Xi}_s(\varphi,y) +\boldsymbol{\delta}_s\int_{B^+_\rho({\bf y})}z^a\frac{|({\bf x}-{\bf y})\cdot\nabla \varphi^\e|^2}{|{\bf x}-{\bf y}|^{n+2-2s}}\,\de {\bf x}=\boldsymbol{\Theta}_s(\varphi^\e,{\bf y},\rho)\,,
\end{equation}
where ${\bf y}=(y,0)$. On the other hand, by homogeneity of $\varphi$, 
$$ \boldsymbol{\Theta}_s(\varphi^\e,{\bf y},\rho)\leq \frac{(\rho+|y|)^{n-2s}}{\rho^{n-2s}} \boldsymbol{\Theta}_s(\varphi^\e,0,\rho+|y|)=\frac{(\rho+|y|)^{n-2s}}{\rho^{n-2s}}\boldsymbol{\Xi}_s(\varphi,0)\,. $$
Combining this inequality with \eqref{yaourtpres} and letting $\rho\to\infty$ yields 
$$\boldsymbol{\Xi}_s(\varphi,y) +\boldsymbol{\delta}_s\int_{\R^{n+1}_+}z^a\frac{|({\bf x}-{\bf y})\cdot\nabla \varphi^\e({\bf x})|^2}{|{\bf x}-{\bf y}|^{n+2-2s}}\,\de {\bf x}\leq \boldsymbol{\Xi}_s(\varphi,0) \,.$$
\vskip5pt

\noindent{\it Step 2.} Next, assume that $\boldsymbol{\Xi}_s(\varphi,y)=\boldsymbol{\Xi}_s(\varphi,0)$ for some $y\not=0$. Then $({\bf x}-{\bf y})\cdot\nabla\varphi^\e({\bf x})=0$ for all ${\bf x}\in\R^{n+1}_+$.  
By $0$-homogeneity of $\varphi^\e$, we then have ${\bf y}\cdot\nabla\varphi^\e({\bf x})=0$ for all ${\bf x}\in\R^{n+1}_+$, and thus 
\begin{equation}\label{invartransSphi}
\varphi(x+y)=\varphi(x) \quad\forall x\in\R^n\,.
\end{equation}
The other way around, if \eqref{invartransSphi} holds for some $y\not=0$, then $({\bf x}-{\bf y})\cdot\nabla\varphi^\e({\bf x})=0$ for all ${\bf x}\in\R^{n+1}_+$ (again by homogeneity). We then infer from \eqref{yaourtpres} and \eqref{invartransSphi} that for $\rho>0$, 
$$\boldsymbol{\Xi}_s(\varphi,y) =\boldsymbol{\Theta}_s(\varphi^\e,{\bf y},\rho)=\boldsymbol{\Theta}_s(\varphi^\e,0,\rho) =\boldsymbol{\Xi}_s(\varphi,0)\,,$$
i.e., $y\in S(\varphi)$. Hence, \eqref{invartransSphi} caracterizes $S(\varphi)$, and the linearity of  $S(\varphi)$ follows. 
\end{proof}

\begin{remark}\label{remdimSphinonconst}
If there exists $\varphi\in T_{x_0}(u)$ such that  ${\rm dim}\,S(\varphi)=n$, then $\varphi$ is clearly constant, and thus $\boldsymbol{\Xi}_s(u,x_0)=\boldsymbol{\Xi}_s(\varphi,0)=0$. By Theorem \ref{thmepsregLip}, $u$ is  continuous  in a neighborhood of $x_0$, so that $\varphi=u(x_0)$. In other words, $ T_{x_0}(u)=\{u(x_0)\}$. 

As a consequence, if on the contrary $\boldsymbol{\Xi}_s(u,x_0)>0$, then all tangent maps $\varphi\in T_{x_0}(u)$ must be non constant, and hence satisfy ${\rm dim}\,S(\varphi)\leq n-1$. 
\end{remark}

\begin{lemma}\label{rigidlemtangmapsbig1/2}
Assume that $s\in[1/2,1)$. If $\varphi\in T_{x_0}(u)$ is not constant, then
$${\rm dim}\,S(\varphi)\leq n-2\,. $$
\end{lemma}

\begin{proof}
We proceed by contradiction assuming that there exists a non constant tangent map $\varphi\in T_{x_0}(u)$ such that ${\rm dim}\,S(\varphi)=n-1$. Rotating coordinates if necessary, we can assume 
that  $S(\varphi)=\{0\}\times\R^{n-1}$. By Lemma \ref{defSphi}, the map $\varphi$ only depends on the $x_1$-variable, that is $\varphi(x)=:\psi(x_1)$ where $\psi\in H^s_{\rm loc}(\R;\mathbb{S}^{d-1})$.  
Since $\varphi$ is positively $0$-homogeneous and non constant, the map $\psi$ is of the form 
\begin{equation}\label{shapetangmap}
\psi(x_1)=\begin{cases}
{\rm a} & \text{if $x_1>0$}\,\\
{\rm b} & \text{if $x_1<0$}\,,
\end{cases}
\end{equation}
for some points ${\rm a}, {\rm b}\in\mathbb{S}^{d-1}$, ${\rm a}\not={\rm b}$. However, the space $H^s_{\rm loc}(\R)$ imbeds into $C^{0,s-1/2}_{\rm loc}(\R)$, which enforces ${\rm a}={\rm b}$, a contradiction.
\end{proof}

\begin{lemma}\label{rigidminsharmtangmap}
Assume that $n\geq 2$, $s\in(0,1/2)$, and that $u$ is a minimizing $s$-harmonic map in $\Omega$. If $\varphi\in T_{x_0}(u)$ is not constant, then
$${\rm dim}\,S(\varphi)\leq n-2\,. $$
\end{lemma}

To prove Lemma \ref{rigidminsharmtangmap}, we shall make use of the following pleasant computation. 

%\begin{lemma}\label{lemcompPoikern}
%Assume that $n\geq 2$. The fractional Poisson kernel ${\bf P}_{n,s}$ satisfies  
%$$\int_{\R^{n-1}}{\bf P}_{n,s}\big((x_1,x^\prime,z)\big)\,\de x^\prime= {\bf P}_{1,s}(x_1,z)\quad\forall (x_1,z)\in\R^2_+\,.$$
%\end{lemma}

\begin{remark}\label{remcomputmasspoisskern}
For $n\geq 2$, we have 
\begin{equation}\label{computalphans}
\alpha_{n,s}:=\int_{\R^{n-1}} \frac{\de x^\prime}{(1+|x^\prime|^2)^{\frac{n+2s}{2}}}=\frac{\gamma_{1,s}}{\gamma_{n,s}}\,.
\end{equation}
Indeed, we easily compute in polar coordinates and  setting $t:=r^2$, 
$$\int_{\R^{n-1}} \frac{\de x^\prime}{(1+|x^\prime|^2)^{\frac{n+2s}{2}}}=|\mathbb{S}^{n-2}|\int_{0}^{+\infty}\frac{r^{n-2}}{(1+r^2)^{\frac{n+2s}{2}}}\,\de r 
=\frac{|\mathbb{S}^{n-2}|}{2}\int_0^{+\infty}\frac{t^{\frac{n-1}{2}-1}}{(1+t)^{\frac{n+2s}{2}}}\,\de t\,.$$
Recalling the value of $\gamma_{n,s}$ given in \eqref{defHsandgammans}, we thus have
\begin{multline}\label{calculpoisskern}
\int_{\R^{n-1}} \frac{\de x^\prime}{(1+|x^\prime|^2)^{\frac{n+2s}{2}}}=\frac{|\mathbb{S}^{n-2}|}{2} {\rm B}\Big(\frac{n-1}{2},\frac{1+2s}{2}\Big)\\
=\frac{|\mathbb{S}^{n-2}|}{2} \frac{\Gamma(\frac{n-1}{2})\Gamma(\frac{1+2s}{2})}{\Gamma(\frac{n+2s}{2})} =\pi^{\frac{n-1}{2}}\frac{\Gamma(\frac{1+2s}{2})}{\Gamma(\frac{n+2s}{2})}=\frac{\gamma_{1,s}}{\gamma_{n,s}}\,,
\end{multline}
where ${\rm B}(\cdot,\cdot)$ denotes the Euler Beta function. 
%Combining \eqref{reduccalcpoisskern} with \eqref{calculpoisskern} and \eqref{defpoisskern}, we obtain
%$$\int_{\R^{n-1}}{\bf P}_{n,s}\big((x_1,x^\prime,z)\big)\,\de x^\prime=\sigma_{1,s}\frac{z^{2s}}{(x_1^2+z^2)^{\frac{1+2s}{2}}}=  {\bf P}_{1,s}(x_1,z)\,,$$
%and the proof is complete.  
\end{remark}

\begin{proof}[Proof of Lemma \ref{rigidminsharmtangmap}]
{\it Step 1.} We proceed again by contradiction assuming that there exists a non constant tangent map $\varphi\in T_{x_0}(u)$ such that ${\rm dim}\,S(\varphi)=n-1$. Rotating coordinates if necessary, we can proceed as in the proof of Lemma \ref{rigidlemtangmapsbig1/2} to infer that $\varphi(x)=:\psi(x_1)$ where $\psi\in H^s_{\rm loc}(\R;\mathbb{S}^{d-1})$ is of the form  \eqref{shapetangmap} for some points ${\rm a}, {\rm b}\in\mathbb{S}^{d-1}$, ${\rm a}\not={\rm b}$.
%
%we can assume 
%that  $S(\varphi)=\{0\}\times\R^{n-1}$. By Lemma \ref{defSphi}, the map $\varphi$ only depends on the $x_1$-variable, that is $\varphi(x)=:\psi(x_1)$ where $\psi\in H^s_{\rm loc}(\R;\mathbb{S}^{d-1})$.  
%Since $\varphi$ is positively $0$-homogeneous and non constant, the map $\psi$ is of the form
%$$\psi(x_1)=\begin{cases}
%{\rm a} & \text{if $x_1>0$}\,\\
%{\rm b} & \text{if $x_1<0$}\,,
%\end{cases}$$
%for some points ${\rm a}, {\rm b}\in\mathbb{S}^{d-1}$, ${\rm a}\not={\rm b}$. 
We claim that $\psi$ is a minimizing $s$-harmonic map in the interval $(-1,1)$. Once the claim is proved (which is the object of the next step), we can infer from the regularity result \cite[Theorem 1.2]{MSY} that $\psi$ is continuous in $(-1,1)$, which again enforces ${\rm a}={\rm b}$, a contradiction.  
\vskip5pt

\noindent{\it Step 2.} We now prove that $\psi$ is a minimizing $s$-harmonic map in $(-1,1)$. To this purpose, we fix an arbitrary competitor $v\in \widehat H^s((-1,1);\mathbb{S}^{d-1})$ such that ${\rm spt}(v-\psi)\subset(-1,1)$. Given $r>1$, we consider the open set $Q_r\subset\R^n$ defined by $Q_r:=(-1,1)\times D^\prime_r$ where $D_r^\prime$ denotes the open ball in $\R^{n-1}$ centered at the origin of radius $r$. 
We define a map $\widetilde v_r\in \widehat H^s(Q_r;\mathbb{S}^{d-1})$ by setting for $x=(x_1,x^\prime)\in\R^n$, 
$$\widetilde v_r(x):=
\begin{cases} 
v(x_1) & \text{if $|x^\prime|<r$}\,,\\
\psi(x_1) & \text{if $|x^\prime|\geq r$}\,.
\end{cases}
$$
Recalling that $u$ is assumed to be minimizing, $\varphi$ is minimizing in every ball. Since ${\rm spt}(\widetilde v_r-\varphi)\subset Q_{r+1}$, we thus have 
$$\mathcal{E}_s(\varphi,Q_{r+1})\leq \mathcal{E}_s(\widetilde v_r,Q_{r+1})\,. $$
Since $\widetilde v_r=\varphi$ in $\R^n\setminus Q_r$, it reduces to 
\begin{equation}\label{mintesttangmap}
\mathcal{E}_s(\varphi,Q_{r})\leq \mathcal{E}_s(\widetilde v_r,Q_{r})\,.
\end{equation}
We claim that 
\begin{equation}\label{asymptreducdimenerg}
\frac{1}{|D^\prime_r|}\,\mathcal{E}_s(\widetilde v_r,Q_{r})\mathop{\longrightarrow}\limits_{r\to\infty} \mathcal{E}_s\big(v,(-1,1)\big)\,,
\end{equation}
where $|D^\prime_r|$ denotes the volume of $D_r^\prime$ in $\R^{n-1}$. Since we could have taken $v$ to be equal to $\psi$, \eqref{asymptreducdimenerg} also holds with $\varphi$ in place of $\widetilde v_r$ and $\psi$ in place of $v$. Therefore, dividing both sides of \eqref{mintesttangmap} by $|D^\prime_r|$ and letting $r\to\infty$ leads to 
$$\mathcal{E}_s\big(\psi,(-1,1)\big)\leq \mathcal{E}_s\big(v,(-1,1)\big) \,,$$
which proves that $\psi$ is indeed minimizing in $(-1,1)$. 

Let us now compute $\mathcal{E}_s(\widetilde v_r,Q_{r})$ to prove \eqref{asymptreducdimenerg}. First, by Fubini's theorem we have 
\begin{multline*}
\iint_{Q_r\times Q_r}\frac{|\widetilde v_r(x)-\widetilde v_r(y)|^2}{|x-y|^{n+2s}}\,\de x\de y\\
=\iint_{(-1,1)^2}|v(x_1)-v(y_1)|^2\Big(\iint_{D_r^\prime\times D_r^\prime}\frac{\de x^\prime\de y^\prime}{(|x_1-y_1|^2+|x^\prime-y^\prime|^2)^{\frac{n+2s}{2}}}\Big)\,\de x_1\de y_1\,.
\end{multline*}
Then we observe that a change of variables yields 
\begin{multline*}
\iint_{D_r^\prime\times D_r^\prime}\frac{\de x^\prime\de y^\prime}{(|x_1-y_1|^2+|x^\prime-y^\prime|^2)^{\frac{n+2s}{2}}}\\
=\iint_{D_r^\prime\times \R^{n-1}}\frac{\de x^\prime\de y^\prime}{(|x_1-y_1|^2+|x^\prime-y^\prime|^2)^{\frac{n+2s}{2}}}-A_r(|x_1-y_1|)\\
=\frac{\alpha_{n,s}|D^\prime_r|}{|x_1-y_1|^{1+2s}}-A_r(|x_1-y_1|)\,,
\end{multline*}
where $\alpha_{n,s}$ is given by \eqref{computalphans}, and $A_r(t)$ is defined for $t>0$ by
$$A_r(t):= \iint_{D_r^\prime\times(D^\prime_r)^c}\frac{\de x^\prime\de y^\prime}{(t^2+|x^\prime-y^\prime|^2)^{\frac{n+2s}{2}}}\,.$$ 
Therefore, 
\begin{multline}\label{calcaoutprout1}
\iint_{Q_r\times Q_r}\frac{|\widetilde v_r(x)-\widetilde v_r(y)|^2}{|x-y|^{n+2s}}\,\de x\de y=\alpha_{n,s}|D^\prime_r|\iint_{(-1,1)^2}\frac{|v(x_1)-v(y_1)|^2}{|x_1-y_1|^{1+2s}}\,\de x_1\de y_1\\
-\iint_{(-1,1)^2}|v(x_1)-v(y_1)|^2A_r(|x_1-y_1|)\,\de x_1\de y_1\,.
\end{multline}
Similarly, we compute 
\begin{multline*}
\iint_{Q_r\times (Q_r)^c}\frac{|\widetilde v_r(x)-\widetilde v_r(y)|^2}{|x-y|^{n+2s}}\,\de x\de y\\
=\iint_{(-1,1)\times(-1,1)^c}|v(x_1)-v(y_1)|^2\Big(\iint_{D_r^\prime\times \R^{n-1}}\frac{\de x^\prime\de y^\prime}{(|x_1-y_1|^2+|x^\prime-y^\prime|^2)^{\frac{n+2s}{2}}}\Big)\,\de x_1\de y_1\\
+\iint_{(-1,1)^2}|v(x_1)-\psi(y_1)|^2A_r(|x_1-y_1|)\,\de x_1\de y_1\,,
\end{multline*}
so that 
\begin{multline}\label{calcaoutprout2}
\iint_{Q_r\times (Q_r)^c}\frac{|\widetilde v_r(x)-\widetilde v_r(y)|^2}{|x-y|^{n+2s}}\,\de x\de y
= \alpha_{n,s}|D^\prime_r|\iint_{(-1,1)\times(-1,1)^c}\frac{|v(x_1)-v(y_1)|^2}{|x_1-y_1|^{1+2s}}\,\de x_1\de y_1\\
+\iint_{(-1,1)^2}|v(x_1)-\psi(y_1)|^2A_r(|x_1-y_1|)\,\de x_1\de y_1\,.
\end{multline}
Combining \eqref{calcaoutprout1} and \eqref{calcaoutprout2} leads to 
%\begin{equation}\label{calcaoutprout3}
$$\frac{1}{|D^\prime_r|}\,\mathcal{E}_s(\widetilde v_r,Q_{r})=  \mathcal{E}_s\big(v,(-1,1)\big)-I_r+II_r\,,$$
%\end{equation}
where 
$$I_r:=\frac{\gamma_{1,s}}{4|D^\prime_r|} \iint_{(-1,1)^2}|v(x_1)-v(y_1)|^2A_r(|x_1-y_1|)\,\de x_1\de y_1\,,$$
and 
$$II_r :=\frac{\gamma_{1,s}}{2|D^\prime_r|}\iint_{(-1,1)^2}|v(x_1)-\psi(y_1)|^2A_r(|x_1-y_1|)\,\de x_1\de y_1\,.$$
Since $|v|=|\psi|=1$, we have 
$$I_r+II_r\leq Cr^{1-n}\iint_{(-1,1)^2}A_r(|x_1-y_1|)\,\de x_1\de y_1\,,$$
and  using Fubini's theorem again, we estimate 
\begin{align*}
\iint_{(-1,1)^2}A_r(|x_1-y_1|)&\;\de x_1\de y_1& \\
&\leq \iint_{D^\prime_r\times(D^\prime_r)^c}\Big(\iint_{(-1,1)\times\R}\frac{\de x_1\de y_1}{(|x_1-y_1|^2+|x^\prime-y^\prime|^2)^{\frac{n+2s}{2}}}\Big)\,\de x^\prime\de y^\prime\\
&\leq C \iint_{D^\prime_r\times(D^\prime_r)^c}\frac{\de x^\prime\de y^\prime}{|x^\prime-y^\prime|^{n-1+2s}}\\
&\leq C r^{n-1-2s}\,.
\end{align*}
%for a constant $C$ depending only on $n$ and $s$. 
Therefore, 
$$\frac{1}{|D^\prime_r|}\,\mathcal{E}_s(\widetilde v_r,Q_{r})=  \mathcal{E}_s\big(v,(-1,1)\big)+O(r^{-2s})\,,$$
and the proof is complete.
%
%
%%By Lemma \ref{lemcompPoikern} above, we have $\varphi^\e(x,z)=\psi^\e(x_1,z)$ for every $(x,z)\in\R^{n+1}_+$ (here $\psi^\e$ corresponds to the convolution product with the $1$-dimensional fractional Poisson kernel ${\bf P}_{1,s}$).
% Recalling that $\varphi$ is a minimizing $s$-harmonic map in $D_r$ for every $r>0$, we infer from {\bf COMPLETE !!} that $\varphi^\e$ is a minimizing weighted harmonic map in the open set $Q\times(0,1)$, where $Q\subset \R^n$ denotes the open cube $(-1,1)^n$. We now claim that $\psi^\e$ is a minimizing weighted harmonic map in the open set $(-1,1)\times(0,1)\subset\R^2_+$. To prove this claim, we fix an arbitrary $w\in H^1((-1,1)\times(0,1),|z|^a\de x_1\de z)$ satisfying $w(x_1,0)\in\mathbb{S}^{d-1}$ for a.e. $x_1\in(-1,1)$, and such that ${\rm spt}(w-\psi^\e)\subset (-1,1)\times[0,1)$. 
%%We choose a function $\zeta\in \mathscr{D}((-1,1)^{n-1})$ satisfying $\|\nabla\zeta\|_{L^2((-1,1)^{n-1})}=1$, and we define $v(x_1,x^\prime,z):=\zeta()$
\end{proof}

\subsection{Proof of Theorem \ref{mainthm1}, Theorem \ref{mainthm2}, and Theorem \ref{mainthm3}}

\begin{proof}[Proof of Theorem \ref{mainthm1}]
Let us fix an arbitrary point $x_0\in \Omega$, and set $r_0:=\frac{1}{2}{\rm dist}(x_0,\Omega^c)$. Without loss of generality, we can assume that $x_0=0$, so that our aim is to show that $u$ is smooth in a neighborhood of $x_0=0$. 
As noticed in Remark \ref{remarlepsholdregsubcritic}, the function $r\in(0,2r_0-|{\bf x}|)\mapsto \boldsymbol{\Theta}_s(u^\e,{\bf x},r)$ is nondecreasing for every ${\bf x}\in\partial^0B^+_{2r_0}$. 
Moreover, since $2s-n=2s-1\geq0$, we have 
$$\lim_{r\to0}\boldsymbol{\theta}_s(u,0,r)=0\,. $$
Then we deduce from  Corollary \ref{corequivvanishdensities} that
$$\lim_{r\to0}\boldsymbol{\Theta}_s(u^\e,0,r)=0\,. $$
As a consequence, we can find $r_1\in(0,r_0)$ such that $\boldsymbol{\Theta}(u^\e,0,r_1)\leq\boldsymbol{\eps}_1$, where  the constant $\boldsymbol{\eps}_1$ is given by Corollary \ref{coroepsreghold}. 
From Theorem \ref{thmepsregLip}, we infer that $u\in C^{0,1}(D_{\boldsymbol{\kappa}_2}r_1)$ for a constant $\boldsymbol{\kappa}_2\in(0,1)$ depending only on $s$. 
In turn, Theorem \ref{highordthm} tells us that $u\in C^\infty(D_{\boldsymbol{\kappa}_2r_1/2})$. 
\end{proof}

\begin{proof}[Proof of Theorem \ref{mainthm2}, case $s=1/2$] Considering  the constant $\boldsymbol{\eps}_1>0$ given by Corollary \ref{coroepsreghold}, we define 
\begin{equation}\label{defSigmaproofmainthms}
\Sigma:=\Big\{x\in\Omega: \boldsymbol{\Xi}_s(u,x)\geq \boldsymbol{\eps}_1 \Big\}\,.
\end{equation}
By Corollary \ref{corolmonotform}, $\Sigma$ is relatively closed subset of $\Omega$. On the other hand, it is well known that $\mathcal{H}^{n-1}(\Sigma)=0$, see e.g. \cite[Corollary 3.2.3]{Ziem}. 

We claim that $u\in C^\infty(\Omega\setminus\Sigma)$. Indeed, if $x_0\in\Omega\setminus\Sigma$, then we can find a radius $r\in(0,\frac{1}{2}{\rm dist}(x_0,\Omega^c))$ such that $\boldsymbol{\Theta}(u^\e,0,r)\leq\boldsymbol{\eps}_1$. Applying Theorem  \ref{thmepsregLip} and Theorem \ref{highordthm}, we conclude that $u\in C^\infty(D_{\boldsymbol{\kappa}_2r/2})$, and the claim is proved. 

Obviously, ${\rm sing}(u)\subset\Sigma$, and it now only remains to show that ${\rm sing}(u)=\Sigma$. This is in fact a direct consequence of the regularity result in \cite[Theorem 4.1]{GJ}. Indeed, assume by contradiction that there is a point $x_0\in \Sigma\setminus{\rm sing}(u)$. Since ${\rm sing}(u)$ is a relatively closed subset of $\Omega$, we can find $r>0$ such that $D_{2r}(x_0)\subset \Omega\setminus{\rm sing}(u)$, i.e., $u$ is continuous in $D_{2r}(x_0)$. Consequently, $u^\e$ is continuous in $B_r^+({\bf x}_0)\cup\partial^0B_r^+({\bf x}_0)$, where ${\bf x}_0=(x_0,0)$. However, by Proposition \ref{equivsharmfreebdry} (with $s=1/2$), $u^\e\in H^1(B_r^+({\bf x}_0);\R^d)$ also solves
$$\int_{B_r(x_0)}\nabla u^\e\cdot\nabla\Phi\,\de{\bf x}=0 $$
for every $\Phi\in H^1(B_r({\bf x}_0);\R^d)$ such that $\Phi=0$ on $\partial^+B_r({\bf x}_0)$ and $u\cdot\Phi=0$ on $\partial^0B_r({\bf x}_0)$. 
Then   \cite[Theorem 4.1]{GJ} tells us that $u^\e\in C^{1,\alpha}(B^+_{r/2}(x_0))$ for every $\alpha\in(0,1)$. Consequently, $\boldsymbol{\Xi}_s(u,x_0)=0$, i.e., $x_0\not\in \Sigma$, a contradiction. 
\end{proof}

\begin{proof}[Proof of Theorem \ref{mainthm2}, case $s\not=1/2$]
We still consider the relatively closed subset $\Sigma$ of $\Omega$ defined in \eqref{defSigmaproofmainthms}. As in the case $s=1/2$, it follows from Theorem  \ref{thmepsregLip} and Theorem \ref{highordthm} that $u\in C^\infty(\Omega\setminus\Sigma)$. In particular, ${\rm sing}(u)\subset \Sigma$. On the other hand, if $u$ is continuous in a neighborhood of a point $x_0\in\Omega$, then $T_{x_0}(u)=\{u(x_0)\}$, and thus $\boldsymbol{\Xi}_s(u,x_0)=0$. Hence, $x_0\not\in\Sigma$, and we conclude that ${\rm sing}(u)= \Sigma$. In view of Remark \ref{remdimSphinonconst} and Lemma \ref{rigidlemtangmapsbig1/2}, we have 
$$\Sigma= 
\begin{cases}
\big\{x\in\Omega: {\rm dim}\,S(\varphi)\leq n-1\;\;\forall\varphi\in T_x(u) \big\} &\text{if $s\in(0,1/2)$}\,;\\[5pt]
\big\{x\in\Omega: {\rm dim}\,S(\varphi)\leq n-2\;\;\forall\varphi\in T_x(u) \big\} &\text{if $s\in(1/2,1)$}\,.
\end{cases}
$$
We can now apply e.g. \cite[Chapter 3.4, proof of Lemma 1]{Sim} (which only relies on the upper semicontinuity of $\boldsymbol{\Xi}_s$ stated in Corollary \ref{uscdensit}, the strong convergence of blow-ups to tangent maps, and the structure results on tangent maps established in Subection \ref{secttangmap}) to conclude that ${\rm dim}_{\mathcal{H}}\Sigma\leq n-1$ for $s\in(0,1/2)$, ${\rm dim}_{\mathcal{H}}\Sigma\leq n-2$ for $s\in(1/2,1)$, and that $\Sigma$ is locally finite in $\Omega$ if $n=1$ with $s\in(0,1/2)$ or $n=2$ with $s\in(1/2,1)$. 
\end{proof}

\begin{proof}[Proof of Theorem \ref{mainthm3}]
For $s\in(1/2,1)$, we simply apply Theorem \ref{mainthm2} (recalling that minimality implies stationarity). We thus assume that $s\in(0,1/2]$. Since $u$ is minimizing in $\Omega$, the results in Subsection  \ref{secttangmap} apply. Hence, we can repeat the proof of Theorem \ref{mainthm2} to derive that $u\in C^\infty(\Omega\setminus\Sigma)$, ${\rm sing}(u)=\Sigma$, where $\Sigma$ is still given by \eqref{defSigmaproofmainthms}. In view of  Lemma \ref{rigidlemtangmapsbig1/2} and Lemma \ref{rigidminsharmtangmap}, we now have 
$$\Sigma=\big\{x\in\Omega: {\rm dim}\,S(\varphi)\leq n-2\;\;\forall\varphi\in T_x(u) \big\}\,.$$
Once again, \cite[Chapter 3.4, proof of Lemma 1]{Sim} shows that ${\rm dim}_{\mathcal{H}}\Sigma\leq n-2$,  and that $\Sigma$ is locally finite in $\Omega$ if $n=2$.  
\end{proof}

%\begin{proof}[Proof of Theorem \ref{mainthm2}, case $s\in(1/2,1)$]
%
%\end{proof}
%

%\subsection{Proof of Theorem \ref{mainthm3}}

%%%%%%%%%%%%%%%%%%%%%%%%%%%%%%%%%%%%%%%%%%%%%%%%%%%%%%%
%%%%%%%%%%%%%%%%%%%%%%%%%%%%%%%%%%%%%%%%%%%%%%%%%%%%%%%
   								       						%%%%%%%%%%%%%%%%%%%
\appendix													%%%%%%%%%
								 						%%%%%%%%%%%%%%%%%%%
%%%%%%%%%%%%%%%%%%%%%%%%%%%%%%%%%%%%%%%%%%%%%%%%%%%%%%%
%%%%%%%%%%%%%%%%%%%%%%%%%%%%%%%%%%%%%%%%%%%%%%%%%%%%%%%

%%%%%%%%%%%%%%%%%%%%%%%%%%%%%%%%%%%%%%%%%%%%%%%%%%%%%%%
%%%%%%%%%%%%%%%%%%%%%%%%%%%%%%%%%%%%%%%%%%%%%%%%%%%%%%%
   								       						%%%%%%%%%%%%%%%%%%%
\section{On the degenerate Laplace equation}\label{appendweightharm} 			 %%%%%%%%%
								 						%%%%%%%%%%%%%%%%%%%
%%%%%%%%%%%%%%%%%%%%%%%%%%%%%%%%%%%%%%%%%%%%%%%%%%%%%%%
%%%%%%%%%%%%%%%%%%%%%%%%%%%%%%%%%%%%%%%%%%%%%%%%%%%%%%%

In this first appendix, our aim is to recall some of the properties satisfied by weak solutions of the (scalar) degenerate linear elliptic equation
\begin{equation}\label{maineqappendA}
{\rm div}(|z|^a\nabla w)= 0 \quad\text{in $B_R({\bf x_0})$}\,,
\end{equation}
with ${\bf x}_0=(x_0,z_0)\in\R^{n+1}$. Those properties are essentially taken from \cite{Rob}, and we reproduce here the statements for convenience of the reader. The notion of weak solution to this equation corresponds to the  variational formulation. In other words, we say that $w\in H^1(B_R({\bf x}_0),|z|^a\de{\bf x})$ is a weak solution of \eqref{maineqappendA} if 
$$\int_{B_R({\bf x}_0)}|z|^a\nabla w\cdot\nabla\Phi\,\de{\bf x}=0$$
for every $\Phi\in H^1(B_R({\bf x}_0),|z|^a\de{\bf x})$ such that $\Phi=0$ on $\partial B_R({\bf x_0})$. 

One may complement \eqref{maineqappendA} with a boundary condition of the form $w=v$ on $\partial B_R({\bf x}_0)$ for a given $v\in H^1(B_R({\bf x}_0),|z|^a\de{\bf x})$. This boundary condition is thus interpreted in the sense of traces. Classically, such a boundary condition uniquely determines the solution of \eqref{maineqappendA} which can be characterized by energy minimality. 

\begin{lemma}\label{minimalityharmonw}
Let $v\in H^1(B_R({\bf x}_0),|z|^a\de{\bf x})$. The equation 
\begin{equation}\label{maineqappendAwithbdry}
\begin{cases}
{\rm div}(|z|^a\nabla w)= 0 &\text{in $B_R({\bf x_0})$}\,,\\
w=v & \text{on $\partial B_R({\bf x}_0)$}\,,
\end{cases} 
\end{equation}
admits a unique weak solution which is characterized by 
$$\int_{B_R({\bf x}_0)}|z|^a|\nabla w|^2\,\de{\bf x}\leq  \int_{B_R({\bf x}_0)}|z|^a|\nabla \Phi|^2\,\de{\bf x}$$
for every $\Phi \in H^1(B_R({\bf x}_0),|z|^a\de{\bf x})$ satisfying $\Phi=v$ on $\partial B_R({\bf x_0})$. 
\end{lemma}

As for the usual Laplace equation, energy minimality can be used to prove that $w$ inherits symmetries from the boundary condition. In our case, we make use of the following lemma. 

\begin{lemma}\label{symmharmw}
Let ${\bf x}_0\in\R^n\times\{0\}$ and $v\in H^1(B_R,|z|^a\de{\bf x})$. If $v$ is symmetric with respect to $\{z=0\}$, then the weak solution $w$ of \eqref{maineqappendAwithbdry}  is also symmetric with respect to $\{z=0\}$. 
\end{lemma}

Concerning interior regularity of weak solutions, the issue is of course near the hyperplane $\{z=0\}$. Indeed, if the ball $B_R({\bf x}_0)$ is away from $\{z=0\}$, then the operator becomes uniformly elliptic with smooth coefficients, and the classical elliptic theory tells us that weak solutions are $C^\infty$ in the interior. For an arbitrary ball, the general results of \cite{FKS} about degenerate elliptic equations  apply, and they  provide at least local H\"older continuity in the interior. Using the invariance of the equation with respect to the $x$-variables, the regularity can be further improved (see e.g. \cite[Corollary~2.13]{Rob}). 
 %but we do not require this fact here.   
 Some boundary regularity and related maximum principles are also known from the general theory in \cite{HKM}.  We reproduce here the statement in \cite[Lemma2.18]{Rob}. 

\begin{lemma}\label{maxprincip}
Let $v\in H^1(B_R({\bf x}_0),|z|^a\de{\bf x})\cap C^0\big(\overline B_R({\bf x}_0)\big)$. The weak solution $w$ of \eqref{maineqappendAwithbdry} belongs to $C^0\big(\overline B_R({\bf x}_0)\big)$. Moreover, 
$$\min_{\overline B_R({\bf x}_0)} w=\min_{\partial B_R({\bf x}_0)} v \quad\text{and}\quad \max_{\overline B_R({\bf x}_0)} w=\max_{\partial B_R({\bf x}_0)} v\,.$$
\end{lemma}

A further fundamental property of weak solutions of \eqref{maineqappendA} is an energy monotonicity in which one has to distinguish balls centered at a point of $\{z=0\}$ from balls lying away from  $\{z=0\}$. The two following lemmas are taken from \cite[Lemma 2.8]{Rob} and \cite[Lemma 2.17]{Rob}, respectively. 

\begin{lemma}\label{monotIharmreplac}
Let ${\bf x}_0\in\R^n\times\{0\}$ and $w\in H^1(B_R({\bf x}_0),|z|^a\de{\bf x})$  a weak solution of \eqref{maineqappendA}. Assume that either $s\geq 1/2$, or that $s<1/2$ and $w$ is symmetric with respect to the hyperplane $\{z=0\}$. Then, 
$$\frac{1}{\rho^{n+2-2s}}\int_{B_\rho({\bf x}_0)}|z|^a|\nabla w|^2\,\de{\bf x}\leq \frac{1}{r^{n+2-2s}}\int_{B_r({\bf x}_0)}|z|^a|\nabla w|^2\,\de{\bf x} $$  
for every $0<\rho\leq r\leq R$.
\end{lemma}

\begin{lemma}\label{monotIharmreplac2}
Let $w\in H^1(B_R({\bf x}_0),|z|^a\de{\bf x})$ be a weak solution of \eqref{maineqappendA}. If ${\bf x}_0=(x_0,z_0)\in \R^{n+1}_+$ and $R>0$ are such that  $B_R({\bf x}_0)\subset \R^{n+1}_+$ and $z_0\geq \theta R$ for some $\theta\geq 2$, then 
$$\Big(\frac{2}{R}\Big)^{n+1}\int_{B_{R/2}({\bf x}_0)} |z|^a|\nabla w|^2\,\de{\bf x}\leq\Big(1+\frac{C}{\theta-1}\Big)\frac{1}{R^{n+1}}\int_{ B_{R} ( {\bf x}_0 ) } |z|^a|\nabla w|^2\,\de{\bf x}\,,$$
for a constant $C=C(n)$. 
\end{lemma}

%%%%%%%%%%%%%%%%%%%%%%%%%%%%%%%%%%%%%%%%%%%%%%%%%%%%%%%
%%%%%%%%%%%%%%%%%%%%%%%%%%%%%%%%%%%%%%%%%%%%%%%%%%%%%%%
   								       						%%%%%%%%%%%%%%%%%%%
\section{A Lipschitz estimate for  $s$-harmonic functions} \label{AppSharmfct}			 %%%%%%%%%
								 						%%%%%%%%%%%%%%%%%%%
%%%%%%%%%%%%%%%%%%%%%%%%%%%%%%%%%%%%%%%%%%%%%%%%%%%%%%%
%%%%%%%%%%%%%%%%%%%%%%%%%%%%%%%%%%%%%%%%%%%%%%%%%%%%%%%

The purpose of this appendix is to provide an interior Lipschitz estimate for weak solutions $w\in \widehat H^s(D_1)$ of the fractional Laplace equation
\begin{equation}\label{sharmfuncteqappend}
(-\Delta)^sw=0 \quad\text{in $H^{-s}(D_1)$}\,. 
\end{equation}
The notion of weak solution is understood here according to the weak formulation of the $s$-Laplacian operator, see \eqref{deffraclap}.
Interior regularity for weak solutions is known, and it tells us that $w$ is locally $C^\infty$ in $D_1$. The following estimate is probably 
also well known, but we give a proof for convenience of the reader. 

\begin{lemma}\label{lipestsharmfctlem}
If $w\in \widehat H^s(D_1)$ is a weak solution of \eqref{sharmfuncteqappend}, then $w\in C^\infty(D_{1/2})$, and 
\begin{equation}\label{lipestiappend}
\|w\|^2_{L^\infty(D_{1/2})}+\|\nabla w\|^2_{L^\infty(D_{1/2})}\leq C\big(\mathcal{E}_s(w,D_1)+\|w\|^2_{L^2(D_1)}\big)\,,
\end{equation}
for a constant $C=C(n,s)$. 
\end{lemma} 

\begin{proof}
As we already mentioned, the regularity theory is already known, and we take advantage of this to only derive estimate \eqref{lipestiappend}. Let us fix an arbitrary point $x_0\in D_{1/2}$. We consider the extension $w^\e$ which belongs to $H^1(B^+_{1/4}({\bf x}_0),|z|^a\de{\bf x})$ with ${\bf x}_0:=(x_0,0)$ by Lemma~\ref{hatH1/2toH1}. In view of Lemma \ref{repnormderfraclap}, it satisfies 
$$\int_{B_{1/4}^+({\bf x}_0)}z^a\nabla w^\e\cdot\nabla\Phi\,\de x=0 $$
for every $\Phi\in H^1(B_{1/4}^+({\bf x}_0),|z|^a\de{\bf x})$ such that $\Phi=0$ on $\partial^+B_{1/4}^+({\bf x}_0)$. Then we consider the even extension of $w^\e$ to the whole ball $B_{1/4}({\bf x}_0)$ that we still denote by $w^\e$ (i.e. $w^\e(x,z)=w^\e(x,-z)$). Then $w^\e\in H^1(B_{1/4}({\bf x}_0),|z|^a\de{\bf x})$, and arguing as in the proof of Corollary \ref{eqsymtrizedphase}, we infer that $w^\e$ is a weak solution of \eqref{maineqappendA} with $R=1/4$. According to \cite[Corollary 2.13]{Rob}, the weak derivatives $\partial_i w^\e$ belongs to $H^1(B_{1/8}({\bf x}_0),|z|^a\de {\bf x})$ for $i=1,\ldots,n$, and they are weak solutions of \eqref{maineqappendA} with $R=1/8$. Now, applying \cite[Theorem 2.3.12]{FKS} to $w^\e$ and $\partial_i w^\e$, we infer that $w^\e\in C^{1,\alpha}(B_{1/16}({\bf x}_0))$ for some exponent $\alpha=\alpha(n,s)\in(0,1)$, 
\begin{equation}\label{holdestisharmfct}
[w^\e]_{C^{0,\alpha}(B_{1/16}({\bf x}_0))}\leq C\|w^\e\|_{L^2(B_{1/8}({\bf x}_0),|z|^a\de{\bf x})}\,, 
\end{equation}
and
\begin{equation}\label{holdestisharmfct2}
[\nabla_x w^\e] _{C^{0,\alpha}(B_{1/16}({\bf x}_0))}\leq C\|\nabla_x w^\e\|_{L^2(B_{1/8}({\bf x}_0),|z|^a\de{\bf x})}\,,
\end{equation}
for a constant $C=C(n,s)$.

On the other hand, for every ${\bf x}\in B_{1/16}({\bf x}_0)$, we have (recall our notation in \eqref{weightvolball})
\begin{multline*}
|w^\e({\bf x})|\leq \Big| w^\e({\bf x})-\frac{1}{|B_{1/16}|_a}\int_{B_{1/16}(x_0)}|z|^a w^\e({\bf y})\de {\bf y} \Big|+ \frac{1}{|B_{1/16}|_a}\int_{B_{1/16}(x_0)}|z|^a|w^\e({\bf y})|\de {\bf y}\\
 \leq C\big([w^\e]_{C^{0,\alpha}(B_{1/16}({\bf x}_0))} +\|w^\e\|_{L^2(B_{1/16}({\bf x}_0),|z|^a\de{\bf x})} \big)\,.
\end{multline*}
Combining this estimate with \eqref{holdestisharmfct} and  Lemma \ref{hatH1/2toH1}  leads to 
$$\|w^\e\|^2_{L^\infty(B_{1/16}(x_0))} \leq C\big(\mathcal{E}_s(w,D_1)+\|w\|^2_{L^2(D_1)}\big)\,.$$
The same argument applied to $\nabla_xw^\e$ and using \eqref{holdestisharmfct2} instead of \eqref{holdestisharmfct} yields 
$$\|\nabla_x w^\e\|^2_{L^\infty(B_{1/16}(x_0))} \leq C \|\nabla_ xw^\e\|^2_{L^2(B_{1/8}({\bf x}_0),|z|^a\de{\bf x})} \leq C\mathcal{E}_s(w,D_1) \,,$$
thanks to Lemma \ref{hatH1/2toH1} again. Now the conclusion follows from the fact that  $w^\e=w$ and  $\nabla_xw^\e=\nabla w$ on $\partial^0B^+_{1/16}({\bf x}_0)$.
\end{proof}

%%%%%%%%%%%%%%%%%%%%%%%%%%%%%%%%%%%%%%%%%%%%%%%%%%%%%%%
%%%%%%%%%%%%%%%%%%%%%%%%%%%%%%%%%%%%%%%%%%%%%%%%%%%%%%%
   								       						%%%%%%%%%%%%%%%%%%%
\section{An embedding theorem between generalized $\mathcal{Q}_\alpha$-spaces}\label{appQspaces} 			 %%%%%%%%%
								 						%%%%%%%%%%%%%%%%%%%
%%%%%%%%%%%%%%%%%%%%%%%%%%%%%%%%%%%%%%%%%%%%%%%%%%%%%%%
%%%%%%%%%%%%%%%%%%%%%%%%%%%%%%%%%%%%%%%%%%%%%%%%%%%%%%%

In this appendix, our goal is to prove one of the crucial estimates used in the proof of Theorem~\ref{thmepsregholder}, Corollary \ref{coroinjQspaces} below. In turns out that this estimate 
does not explicitly appear in the existing literature (to the best of our knowledge), but it can be shortly derived from recent results in harmonic analysis. The purpose of this appendix is thus to explain how to combine those results to reach our goal. First, we need to recall some definitions and notations. 

The space $\mathscr{S}_\infty(\R^n)$ can be defined as the topological subspace of the Schwartz class $\mathscr{S}(\R^n)$ made of all functions $\varphi$ such that the semi-norm
$$\|\varphi\|_M :=\sup_{|\gamma|\leq M}\sup_{\xi\in\R^n}\big|\partial^\gamma\widehat{\varphi}(\xi)\big|(|\xi|^M+|\xi|^{-M})$$
is finite for every  $M\in\N$, where $\gamma=(\gamma_1,\ldots,\gamma_n)\in\N^n$, $|\gamma|:=\gamma_1+\ldots+\gamma_n$, and $\partial^\gamma:=\partial_1^{\gamma_1}\ldots\partial_n^{\gamma_n}$. Its topological dual is denoted by $\mathscr{S}^\prime_\infty(\R^n)$, and it is endowed with the weak $*$-topology, see e.g.~\cite{Trieb,YangYuan}. 
\vskip3pt

The following $\mathcal{Q}^{\alpha,q}_p$-spaces were introduced in \cite{CuYa,YangYuan}, generalizing the notion of $\mathcal{Q}_\alpha$-space (see \cite[Section 1.2.4]{SYY} and  references therein), in the sense that $\mathcal{Q}_\alpha(\R^n)=\mathcal{Q}^{\alpha,2}_{n/\alpha}(\R^n)$. 

\begin{definition}[\cite{CuYa,YangYuan}]
Given $\alpha\in(0,1)$, $p\in(0,\infty]$ and $q\in[1,\infty)$, define $\mathcal{Q}^{\alpha,q}_p(\R^n)$ as the space made of elements $f\in\mathscr{S}^\prime_\infty(\R^n)$ such that $f(x)-f(y)$ is a measurable function on $\R^n\times\R^n$ and 
$$\|f\|_{\mathcal{Q}^{\alpha,q}_p(\R^n)}:=\sup_Q\, |Q|^{\frac{1}{p}-\frac{1}{q}}\left(\iint_{Q\times Q}\frac{|f(x)-f(y)|^q}{|x-y|^{n+\alpha q}}\,\de x\de y\right)^{1/q}<+\infty\,,$$
where $Q$ ranges over all cubes of dyadic edge lengths in $\R^n$. 
\end{definition}

\begin{remark}\label{equivseminormQsp}
Endowed with $\|\cdot\|_{\mathcal{Q}^{\alpha,q}_p(\R^n)}$, the space $\mathcal{Q}^{\alpha,q}_p(\R^n)$ is a semi-normed vector space, and 
$$N_{\alpha,p,q}(f):=\sup_{D_r(x_0)\subset \R^n} r^{\frac{n}{p}-\frac{n}{q}}\left(\iint_{D_r(x_0)\times D_r(x_0)}\frac{|f(x)-f(y)|^q}{|x-y|^{n+\alpha q}}\,\de x\de y\right)^{1/q}$$
provides an equivalent semi-norm. 
\end{remark}

The following embeddings between $\mathcal{Q}^{\alpha,q}_p$-spaces hold. 

\begin{theorem}\label{prfthminj}
If $0<\alpha_1<\alpha_2<1$, $1\leq q_2<q_1<\infty$, and $0<\lambda\leq n$ are such that 
\begin{equation}\label{conditionembed}
 \alpha_1-\frac{\lambda}{q_1}=\alpha_2-\frac{\lambda}{q_2}\,,
 \end{equation}
then $\mathcal{Q}^{\alpha_2,q_2}_{\frac{nq_2}{\lambda}}(\R^n)\hookrightarrow \mathcal{Q}^{\alpha_1,q_1}_{\frac{nq_1}{\lambda}}(\R^n)$ continuously. 
\end{theorem}

As we briefly mentioned at the beginning of this appendix, this theorem actually follows quite directly from a more general embedding result between some homogeneous Triebel-Lizorkin-Morrey-Lorentz spaces \cite{Ho} together with an identification result between various definitions of homogeneous Triebel-Lizorkin-Morrey type spaces \cite{SaYY}, and a characterization of the $\mathcal{Q}^{\alpha,q}_p$-spaces within this scale of spaces~\cite{YangYuan}. We refer to the monograph \cite{SYY} for what concerns the spaces involved here, and we limit ourselves to their basic definition. To this purpose, we consider a reference 
bump function $\psi\in \mathscr{S}(\R^n)$ such that 
$$
{\rm spt}\,\widehat{\psi} \subset \Big\{\xi\in\R^n : \frac{1}{2}\leq|\xi|\leq
2\Big\}\quad\text{and}\quad |\widehat\psi(\xi)|\geq C>0 \textrm{ for }\frac{3}{5}\leq |\xi|\leq\frac{5}{3}\,.
$$
(In particular, $\psi\in \mathscr{S}_\infty(\R^n)$.) For $j\in\Z$, we denote by $\psi_j$ the function defined by 
$$\psi_j(x):=2^{jn}\psi(2^jx) \,.$$

\begin{definition}
Given $p,q\in(0,\infty)$, $s\in\R$, and $\tau\in[0,\infty)$, the homogeneous Triebel-Lizorkin space $\dot F^{s,\tau}_{p,q}(\R^n)$ is defined to be the set of all $f\in\mathscr{S}^\prime_\infty(\R^n)$ such that 
$$\|f\|_{\dot F^{s,\tau}_{p,q}(\R^n)}:=\sup_Q\, \frac{1}{|Q|^{\tau}}\left(\int_Q\bigg(\sum_{j=j_Q}^\infty \big(2^{js}|\psi_j*f(x)|\big)^q\bigg)^{p/q} \de x\right)^{1/p}<+\infty\,, $$
where $Q$ ranges over all cubes of dyadic edge lengths in $\R^n$, and $j_Q:=-\log_2\ell(Q)$ with $\ell(Q)$ the edge length of $Q$. 
\end{definition}

\begin{definition}
Given $0<p\leq u<\infty$\,, $0<q<\infty$, and $s\in\R$,  the homogeneous Triebel-Lizorkin-Morrey space $\dot{\mathcal{E}}^s_{p,q,u}(\R^n)$ is defined to be the set of all $f\in\mathscr{S}^\prime_\infty(\R^n)$ such that 
$$\|f\|_{\dot{\mathcal{E}}^s_{p,q,u}(\R^n)}:=\sup_Q\, |Q|^{\frac{1}{u}-\frac{1}{p}}\left(\int_Q\bigg(\sum_{j\in\Z} \big(2^{js}|\psi_j*f(x)|\big)^p\bigg)^{q/p} \de x\right)^{1/q}<+\infty\,,$$
where $Q$ ranges over all cubes of dyadic edge lengths in $\R^n$. 
\end{definition}

\begin{proof}[Proof of Theoremf \ref{prfthminj}]
In \cite{Ho}, the author introduced a more refined scale of homogeneous Triebel-Lizorkin spaces
of Morrey-Lorentz type, denoted by $\dot{F}^{s,u}_{M_{p,q,\lambda}}(\Rn)$. In the case $u=p=q$, those spaces coincide with the homogeneous Triebel-Lizorkin-Morrey spaces above, namely 
$$\dot{F}^{s,p}_{M_{p,p,\lambda}}(\Rn)=\dot{\mathcal{E}}^s_{p,p,\frac{np}{\lambda}}(\R^n) $$
for every $p\in(0,\infty)$, $\lambda\in(0,n]$, and $s\in \R$. More precisely, their defining semi-norms are equivalent (in one case the supremum is taken over all dyadic cubes, while in the other it is taken
over balls). By \cite[Theorem 4.1]{Ho}, under condition \eqref{conditionembed} the space $\dot{F}^{\alpha_2,q_2}_{M_{q_2,q_2,\lambda}}(\Rn)$ embeds continuously into $\dot{F}^{\alpha_1,q_1}_{M_{q_1,q_1,\lambda}}(\Rn)$. In other words, 
\begin{equation}\label{crucialinj}
\dot{\mathcal{E}}^{\alpha_2}_{q_2,q_2,\frac{nq_2}{\lambda}}(\R^n)\hookrightarrow \dot{\mathcal{E}}^{\alpha_1}_{q_1,q_1,\frac{nq_1}{\lambda}}(\R^n) 
\end{equation}
continuously. On the other hand, \cite[Theorem 1.1]{SaYY} tells us that 
$$\dot{\mathcal{E}}^{\alpha_1}_{q_1,q_1,\frac{nq_1}{\lambda}}(\R^n)= \dot F^{\alpha_1,\frac{n-\lambda}{nq_1}}_{q_1,q_1}(\R^n)\quad\text{and}\quad \dot{\mathcal{E}}^{\alpha_2}_{q_2,q_2,\frac{nq_2}{\lambda}}(\R^n)= \dot F^{\alpha_2,\frac{n-\lambda}{nq_2}}_{q_2,q_2}(\R^n)\,,$$
with equivalent semi-norms. Finally, by \cite[Theorem 3.1]{YangYuan} we have 
$$  \dot F^{\alpha_1,\frac{n-\lambda}{nq_1}}_{q_1,q_1}(\R^n) =  \mathcal{Q}^{\alpha_1,q_1}_{\frac{nq_1}{\lambda}}(\R^n) \quad\text{and}\quad  \dot F^{\alpha_2,\frac{n-\lambda}{nq_2}}_{q_2,q_2}(\R^n) =  \mathcal{Q}^{\alpha_2,q_2}_{\frac{nq_2}{\lambda}}(\R^n) \,,$$
with equivalent semi-norms. Hence, the conclusion follows from \eqref{crucialinj}.
\end{proof}

We are now ready to state the important corollary of Theorem \ref{prfthminj} used in the proof of Theorem~\ref{thmepsregholder}. Given $s\in(0,1)$, $p\in[1,\infty)$, and an open set $\Omega\subset\R^n$, we recall that the  Sobolev-Slobodeckij  $W^{s,p}(\Omega)$-semi-norm of a measurable function $f$ is given by  
\begin{equation}\label{defWspseminorm}
[f]_{W^{s,p}(\Omega)}:=\left(\iint_{\Omega\times\Omega}\frac{|f(x)-f(y)|^p}{|x-y|^{n+sp}}\,\de x\de y\right)^{1/p}\,.
\end{equation}

\begin{corollary}\label{coroinjQspaces}
Let $s\in(0,1)$ and $f\in L^1(\R^n)$ with compact support. If
\begin{equation}\label{condHsQspace}
\sup_{D_r(x)\subset\R^n} r^{2s-n}[f]^2_{H^s(D_r(x))}<+\infty\,,
\end{equation}
then,
$$\sup_{D_r(x)\subset\R^n} r^{\frac{2s-n}{3}}[f]^2_{W^{s/3,6}(D_r(x))}\leq C \sup_{D_r(x)\subset\R^n} r^{2s-n}[f]^2_{H^s(D_r(x))}\,,$$
for a constant $C=C(n,s)$. 
\end{corollary}

\begin{proof}
Since $f\in L^1(\R^n)$ has compact support, it clearly belongs to $\mathscr{S}^\prime_\infty(\R^n)$. Then, condition \eqref{condHsQspace} implies that $f\in\mathcal{Q}^{s,2}_{n/s}(\R^n)$. On the other hand, $\mathcal{Q}^{s,2}_{n/s}(\R^n)\hookrightarrow \mathcal{Q}^{s/3,6}_{3n/s}(\R^n)$ continuously by Theorem \ref{prfthminj}. Then the conclusion follows from the definition of $\mathcal{Q}^{s/3,6}_{3n/s}(\R^n)$ 
together with Remark \ref{equivseminormQsp}.
\end{proof}

\vskip10pt

\noindent{\it Acknowledgements.} V.M. is supported by the Agence Nationale de la Recherche
through the project ANR-14-CE25-0009-01 (MAToS). A.S. is supported by the Simons Fondation through the grant no. 579261.

%=======================
% BIBLIOGRAPHY AND INDEX
%=======================

\end{document}